\documentclass[10pt]{article}

\usepackage{authblk}

\usepackage{amsmath,amssymb}
\usepackage{amsmath, amssymb}
\usepackage{amsmath}
\usepackage{isotope}
\usepackage{dirtytalk}
\usepackage{amsbsy}
\usepackage{newclude}
\usepackage{nicefrac,xfrac}
\usepackage{bbm}
\usepackage{float}
\usepackage{makecell}
\usepackage{amsmath, amssymb}
\usepackage[utf8]{inputenc}

\usepackage[utf8]{inputenc}
\usepackage{hyperref}
\usepackage{tikz}
\usepackage{amssymb}
\usepackage{authblk}
\usepackage[maxbibnames=99,
backend=biber,
sorting=nyt
]{biblatex}
\newcommand\norm[1]{\lVert#1\rVert}
\usepackage{amsmath}
\usepackage{accents}
\newlength{\dhatheight}

\usepackage[T1]{fontenc}
\usepackage{amsmath, amssymb, amsfonts, amsthm, caption, framed, enumerate, mathtools, fullpage, subcaption, siunitx, tabularx, verbatim, relsize, stackrel, cancel, multicol, media9, xcolor}

\usepackage[normalem]{ulem}

\providecommand{\keywords}[1]
{	
  \textbf{\textit{Keywords:}} #1
}

\providecommand{\MSC}[1]
{	
  \textbf{\textit{MSC2020 subject classification:}} #1
}

\usepackage{booktabs}

\newtheorem{theorem}{Theorem}[section]
\newtheorem{lemma}[theorem]{Lemma}

\newtheorem{corollary}[theorem]{Corollary}

\newtheorem{proposition}[theorem]{Proposition}
\theoremstyle{definition}
\newtheorem{definition}[theorem]{Definition}
\theoremstyle{remark}
\newtheorem{remark}[theorem]{Remark}

\title{The on/off Brownian snake}
\author{Matthew Buckland*, Dave Jacobi**}
\affil{\small*Department of Statistics, University of Warwick, matt.buckland@warwick.ac.uk; **Institut für Mathematik, Technische Universität Berlin, Institut für Mathematik, Freie Universität Berlin, jacobi@math.tu-berlin.de}
\begin{document}

\maketitle

\begin{abstract}
    We define what we call an on/off Brownian snake. We use this to construct on/off super Brownian motion recently introduced to the literature by Blath and Jacobi and which is a measure-valued branching process with a dormant state and an active state. Our construction mirrors the construction of super Brownian motion from the Brownian snake by Le Gall. We use the on/off Brownian snake to obtain results concerning the support, range, and expected total mass of on/off super Brownian motion. 
\end{abstract}
\small
\newline
\MSC{Primary: 60J68, 60J80, 60J85; Secondary: 60B10, 60J25, 60J55, 60G57, 60J65, 92D25}

%Added: 60J80 (branching processes, galton watson processes), 60J55 (local time and additive processes), 60G57 (random measures), 60J65 (Brownian motion),  60J25 (Continuous-time Markov processes on general state spaces)

% \keywords{}
% \MSC{}
\normalsize

% \tableofcontents

\section{Introduction}\label{section snake introduction}
Superprocesses are measure-valued branching processes, see e.g. Etheridge \cite{etheridge2000introduction}. These processes typically arise as scaling limits of branching Markov processes. Super-Brownian motion (SBM) is the limit of branching Brownian motion (BBM) which has the spatial motion of Brownian motion and a binary branching mechanism (i.e. produces zero offspring with probability $1/2$ and produces two offspring with probability $1/2$). SBM can be constructed from Le Gall's Brownian snake, \cite{legall1993class}, which is a path-valued process. This construction provides a fascinating description of classical super-Brownian motion and ultimately a deep insight into the superprocess itself. The aim of this paper is to use an analogous path-valued process which we call the on/off Brownian snake to uncover properties about the long-term behaviour of the superprocess known as on/off super-Brownian motion (ooSBM), which was recently defined in \cite{blath2023off} and is motivated by interest in the effects of dormancy on populations. The superprocess ooSBM has an active state which follows the dynamics of the classical superprocess SBM and a dormant state which exhibits no branching or spatial motion. Mass is able to switch between the states and this provides unique characteristics, such as almost-sure long-term survival, but much remains unknown about the superprocess. In this paper, we provide results about the support and range of ooSBM and about how the expected total masses of the active and dormant components of the superprocess vary over time.
\par
Since the original paper on the Brownian snake, \cite{legall1993class}, this snake-type description of superprocesses has been extended to include superprocesses that are scaling limits of branching Markov processes with discontinuous spatial motion, Bertoin, Le Gall and Le Jan \cite{BertoinLeGallLeJan1997}, and also to describe superprocesses that are scaling limits of branching Markov processes with a non-binary branching mechanism, Le Gall \cite{le1999spatial} (the snakes in this second case are known as L\'evy snakes). There are further generalisations of superprocesses and snake constructions, see e.g. \cite{abraham2002poisson, delmas2003super}. 
\par
The superprocess ooSBM is the scaling limit of on/off branching Brownian motion (ooBBM) which is a branching Markov process with a feature where particles switch after exponentially distributed times between an active state where they move and branch and a dormant state where the particle is fixed in a spatial location and does not branch. While the spatial motion for ooBBM within the active state and the dormant state is continuous, the switching itself is discontinuous as the particle moves from one state to the other. The switching between active and dormant states could also be interpreted as a branching event where an active particle (respectively a dormant particle) produces a single offspring in the dormant state (respectively in the active state). Under this interpretation, the branching is non-local as the offspring is in a different state to the parent. Essentially, classical superprocesses and their pre-limiting branching Markov processes only have an active state and the ability to transition to a dormant state is the defining feature in our setting. It is because of these transitions that we use snakes defined in \cite{BertoinLeGallLeJan1997} which take the values of killed c\`adl\`ag paths to construct our on/off Brownian snake. Following their techniques, we construct exit measures from the snakes that correspond to ooSBM at different times. See e.g. Le Gall \cite[Chapter V]{le1999spatial} for an overview on exit measures.
\par
The snakes in \cite{BertoinLeGallLeJan1997} differ from the usual Brownian snakes used to define SBM which take the values of stopped continuous paths as the killed c\`adl\`ag paths allow for discontinuous spatial components for the snakes. As in \cite{BertoinLeGallLeJan1997}, we add to the spatial motion component of the snakes independent subordinator processes and we refer to these snakes as bivariate snakes due to the two independent spatial motion components. Where our method differs from \cite{BertoinLeGallLeJan1997} is in how we use the subordinator components of the bivariate snakes to construct the on/off Brownian snakes. The construction we give in this paper is the first example of snakes being used to describe superprocesses with dormancy components. 
\par
As mentioned above, the superprocess SBM is a scaling limit of BBM, see e.g. Etheridge \cite[Chapter 1]{etheridge2000introduction}, and the superprocess ooSBM is a scaling limit of ooBBM, \cite[Definition 1.1]{blath2023off}. We now formally define BBM and ooBBM in Definition \ref{definition bbm and oobbm} and we then state how to construct the scaling limits SBM and ooSBM in Propositions \ref{proposition SBM existence} and \ref{proposition existence of ooSBM}.
\begin{definition}[Branching Brownian motion (BBM) and on/off branching Brownian motion (ooBBM)]\label{definition bbm and oobbm}
    For $b$, $k \geq 0$, \emph{branching Brownian motion} with splitting rate $b$ and death rate $k$ (\emph{$\text{BBM}(b, k)$}) is a discrete particle system that evolves according to the following dynamics:
    \begin{enumerate}
        \item each particle independently moves according to the law of a Brownian motion in $\mathbb{R}^d$ during its lifetime;
        \item a particle splits into two particles at the same spatial location at rate $b$;
        \item a particle dies at rate $k$.
    \end{enumerate}
    When $b=k$ we set $\gamma = 2b$ and we refer to the particle system as a \emph{branching Brownian motion} with parameter $\gamma$ (\emph{$\text{BBM}(\gamma)$}).
    \par
    For $\gamma \geq 0$ and for $c$, $\Tilde{c}>0$, \emph{on/off branching Brownian motion} with parameter $\gamma$, dormancy rate $c$ and wake-up rate $\Tilde{c}$ (\emph{$\text{ooBBM}(\gamma, c, \Tilde{c})$}) is a discrete particle system where particles take values in $\mathbb{R}^d \times \{0,1\}$. The first component is the spatial position of the particle and the second component denotes whether a particle is active (taking value $1$) or dormant ($0$). We furthermore have the following.
    \begin{enumerate}
        \item Each \say{active} particle independently has the spatial motion of a Brownian motion in $\mathbb{R}^d$. It also has splitting rate $\gamma/2$ and death rate $\gamma/2$ for some $\gamma \geq 0$ as before. Finally an \say{active} particle switches to a \say{dormant} particle with rate $c$ for some $c>0$.
        \item Each \say{dormant} particle does not move or split or die. Each particle stays in the same spatial position and switches to an \say{active} particle with rate $\Tilde{c}$ for some $\Tilde{c}>0$.
    \end{enumerate}
    BBM takes values in the space of finite measures on $\mathbb{R}^d$, $\mathcal{M}_F(\mathbb{R}^d)$, and ooBBM takes values in the space of pairs of finite measures on $\mathbb{R}^d$, $\mathcal{M}_F(\mathbb{R}^d)^2$, where the first component is the sum of unit Dirac masses at the positions of dormant particles and the second component similarly represents the active particles.
\end{definition}
We have included the degenerate cases $\gamma=0$ in the definition above as these are useful notions later (in Lemma \ref{lemma single particle motion with age process}). In this case when there is a single particle we can think of the process taking values in $\mathbb{R}^d$ or $\mathbb{R}^d \times \{0,1\}$. For example, $\text{BBM}(0)$ with a single initial particle is effectively a Brownian motion. Similarly, we define the $\mathbb{R}^d \times \{0, 1\}$-valued process which describes the evolution of a particle of $\text{ooBBM}(0, c, \Tilde{c})$ as on/off Brownian motion with dormancy rate $c$ and wake-up rate $\Tilde{c}$ ($\text{ooBM}(c, \Tilde{c})$).  
\par
We now define the superprocesses SBM and ooSBM. The process SBM is well studied in the literature as one of the original measure-valued branching processes. It is also known as the Dawson-Watanabe process after its introduction by Watanabe \cite{watanabe1968limit} and its use in stochastic evolution equations by Dawson \cite{dawson1975stochastic}. The process ooSBM has been introduced more recently in Blath and Jacobi \cite{blath2023off} and follows current interest in the effects of dormancy in population biology, see e.g. \cite{blath2016new, lennon2021principles}. The effect of dormancy in ooSBM impacts the properties of the superprocess. For example, their research shows in \cite[Theorem 1.10]{blath2023off} that ooSBM does not die out in finite time, unlike SBM. Our research concerns the support and the range of ooSBM, where heuristically the range of a superprocess is the set of spatial locations ever visited by the process over its lifetime whereas the support of the superprocess at a point in time is the set of spatial locations occupied by the superprocess at that time, see Definition \ref{definition support and range} below. In Theorem \ref{theorem equality ranges} we show a coupling between the range of ooSBM and the range of SBM which resolves a conjecture given in \cite[Remark 1.13]{blath2023off}.
\par
SBM and ooSBM are both measure-valued processes. We define the Prokhorov metric for finite measures next and we then state some simple properties of this metric following \cite[Section A.2.5]{daley2003introduction}.
\par
For a metric space $(E, d_E)$ we define the \emph{Prokhorov metric} $\rho$ for a pair of finite measures on $(E, d_E)$, given by $P$, $Q \in \mathcal{M}_F(E)$, as
\begin{equation*}
    \rho(P,Q) = \inf \Big(\Big\{ \varepsilon >0: P(F) \leq Q(F^\varepsilon) + \varepsilon; \hspace{0.25cm} Q(F) \leq P(F^\varepsilon)+\varepsilon \hspace{0.25cm} \text{ for all } F \in \mathcal{C} \Big\}\Big)
\end{equation*}
where $\mathcal{C}$ is the collection of closed subsets of $E$ and $F^\varepsilon \coloneqq \{x \in E: \inf(\{d_E(x,y):y \in F\}) \leq \varepsilon\}$. The Prokhorov metric metrizes weak convergence and yields a complete and separable metric space if $(E,d_E)$ is complete and separable.
\par
SBM and ooSBM are scaling limits of BBM and ooBBM. To formulate this precisely, we denote by $\mathbb{D}([0,\infty),E)$ the space of measure-valued c\`adl\`ag paths $f:[0,\infty) \rightarrow E$ where $(E,d_E)$ is a complete and separable metric space, equipped with the Skorokhod topology and with the Skorokhod metric that we denote by $d^{\text{Sko}}$, see e.g. \cite[Section 16]{billingsley1999convergence}.
\begin{proposition}[The existence of super Brownian motion (SBM), see e.g. \cite{etheridge2000introduction}]\label{proposition SBM existence}
    Let $\mu \in \mathcal{M}_F(\mathbb{R}^d)$. Then there exists a measure-valued process $(\mathcal{X}_s)_{s \geq 0}$ that takes values in $\mathcal{M}_F(\mathbb{R}^d)$ with $\mathcal{X}_0=\mu$ which is the scaling limit of BBMs in the following sense. Let $\Tilde{\mathcal{Z}}^\varepsilon=(\Tilde{\mathcal{Z}}^\varepsilon_s)_{s \geq 0}$ be $\text{BBM}(\gamma/\varepsilon)$ with $\Tilde{\mathcal{Z}}_0^\varepsilon$ distributed as a Poisson random measure with intensity $\mu/\varepsilon$. Then we have 
\begin{equation*}
    (\varepsilon \Tilde{\mathcal{Z}}^\varepsilon_s)_{s \geq 0} \rightarrow (\mathcal{X}_s)_{s \geq 0}
\end{equation*}
as $\varepsilon$ goes to $0$, in distribution on  $\mathbb{D}([0,\infty), \mathcal{M}_F(\mathbb{R}^d))$. This process $(\mathcal{X}_s)_{s \geq 0}$ is called $\text{SBM}(\gamma)$ started at $\mu$.
\end{proposition}
\begin{proposition}[The existence of on/off super Brownian motion (ooSBM), Blath and Jacobi \cite{blath2023off}]\label{proposition existence of ooSBM}
Let $(\mu_{\rm d}, \mu_{\rm a}) \in \mathcal{M}_F(\mathbb{R}^d)^2$. Then there exists a measure-valued Markov process $((\mathcal{Y}^{(\rm d)}_s, \mathcal{Y}^{(\rm a)}_s), s \geq 0)$ that takes values in $\mathcal{M}_F(\mathbb{R}^d)^2$ with $(\mathcal{Y}^{(\rm d)}_0, \mathcal{Y}^{(\rm a)}_0)=(\mu_{\rm d}, \mu_{\rm a})$ which is the scaling limit of ooBBM in the following sense. For $\varepsilon>0$, let $((\mathcal{Z}^{(\varepsilon, \rm d)}_s, \mathcal{Z}^{(\varepsilon, \rm a)}_s))_{s \geq 0}$ be $\text{ooBBM}(\gamma/\varepsilon, c, \Tilde{c})$ with $\mathcal{Z}_0^\varepsilon$ distributed as a Poisson random measure with intensity $(\mu_{\rm d}/\varepsilon, \mu_{\rm a}/\varepsilon)$. Then we have 
\begin{equation*}
    ((\varepsilon\mathcal{Z}^{(\varepsilon, \rm d)}_s, \varepsilon\mathcal{Z}^{(\varepsilon, \rm a)}_s))_{s \geq 0} \rightarrow ((\mathcal{Y}^{(\rm d)}_s, \mathcal{Y}^{(\rm a)}_s))_{s \geq 0}
\end{equation*}
as $\varepsilon$ goes to $0$, in distribution on $\mathbb{D}([0,\infty), \mathcal{M}_F(\mathbb{R}^d)^2)$. The process $((\mathcal{Y}^{(\rm d)}_s, \mathcal{Y}^{(\rm a)}_s))_{s \geq 0}$ is called $\text{ooSBM}(\gamma, c, \Tilde{c})$ started at $(\mu_{\rm d}, \mu_{\rm a})$.
\end{proposition}
The superprocess ooSBM can also be characterised through its evolution equations, see \cite[Theorem 1.3]{blath2023off}. We make a remark about the distributional results in Proposition \ref{proposition existence of ooSBM}.
\begin{remark}\label{remark different measure representations}
    In \cite{blath2023off}, ooSBM and ooBBM are actually represented as processes in $\mathcal{M}_F(\mathbb{R}^d \times \{0,1\})$ but it will be more convenient for us to represent them in $\mathcal{M}_F(\mathbb{R}^d)^2$ where the two components are the projections to $\mathbb{R}^d$ of the measure to the dormant part $\mathbb{R}^d \times \{0\}$ and active part $\mathbb{R}^d \times \{1\}$. As the map from $\mathcal{M}_F(\mathbb{R}^d \times \{0,1\})$ to $\mathcal{M}_F(\mathbb{R}^d)^2$ is a homeomorphism, all results can be easily transferred. We use this alternative formulation in Section \ref{section moments}, see Table \ref{table different spaces}.
\end{remark}
We define the support of a measure noted in Cohn \cite[Section 7.4]{cohn2013measure} and the range of a measure-valued process from Le Gall \cite[Chapter IV]{le1999spatial}.
\begin{definition}[The support of a measure and the range of a measure-valued process]\label{definition support and range}
    Let $\mu \in \mathcal{M}_F(\mathbb{R}^d)$. Then $\mathbb{R}^d$ has a largest open subset $U_\mu$ such that $\mu(U_\mu)=0$ -- this is the union of all open subsets $U$ such that $\mu(U)=0$. The \emph{support} of $\mu$ is defined to be $\text{supp}(\mu) \coloneqq \mathbb{R}^d \backslash U_\mu$.
    \par
    For a measure-valued c\`adl\`ag process $\mathcal{W}=(\mathcal{W}_s)_{s \geq 0}$ we define the \emph{range} to be 
    \begin{equation*}
        \mathcal{R}(\mathcal{W}) = \bigcup_{\delta > 0} \overline{\bigcup_{y \geq \delta} \text{supp}(\mathcal{W}_y)}.
    \end{equation*}
\end{definition}
We note that $\text{supp}(\mu)$ is the smallest closed set whose complement has measure zero under $\mu$. Furthermore we note that a point $x \in \mathbb{R}^d$ is in $\text{supp}(\mu)$ if and only if every open neighbourhood of $x$ has positive measure under $\mu$. We use these definitions when we state our main results for ooSBM in Theorem \ref{theorem equality ranges}. Some of the literature takes the range of a superprocess $\mathcal{X}$ to be closure of the range $\mathcal{R}(\mathcal{X})$ defined in Definition \ref{definition support and range}. We discuss the difference in our results for SBM and ooSBM with this alternative definition in Remark \ref{remark alternative range operator}. 
\par
These sets given by the range and the support of a superprocess have an interesting structure. In Definition \ref{definition hausdorff dimension}, we define the Hausdorff dimension of the range and the support, following \cite[Chapter 6]{etheridge2000introduction}.
\begin{definition}[The Hausdorff dimension]\label{definition hausdorff dimension}
    Let $A \subseteq \mathbb{R}^d$ be a measurable set. For $\delta>0$ we call a $\delta$-cover any collection of open sets $U_j, j \in J$, such that 
    \begin{equation*}
        A \subseteq \bigcup_{j \in J}U_j; \hspace{0.5cm} \text{sup}_{j \in J}\text{Leb}(U_j) < \delta.
    \end{equation*}
    We define $\mathcal{H}^s_\delta(A) = \inf \sum_{j \in J}|U_j|^s$ where the infimum is taken over all $\delta$-covers, and $\mathcal{H}^s(A) = \lim_{\delta \rightarrow 0} \mathcal{H}^s_\delta(A)$. The \emph{Hausdorff dimension} of $A$, which we denote $\text{dim}(A)$, is defined to be the unique number such that
    \begin{equation*}
        \mathcal{H}^s(A) = \infty \text{   for all   } 0 \leq s < \text{dim}(A); \hspace{0.25cm} \text{   and   } \hspace{0.25cm} \mathcal{H}^s(A) = 0 \text{   for all   } \text{dim}(A) < s < \infty.
    \end{equation*}
\end{definition}
The Hausdorff dimension gives a notion of roughness for certain non-standard Euclidean subspaces (it coincides with the Euclidean dimension for more standard subspaces with smooth boundaries). We state well-known results about the Hausdorff dimension of the range and support of SBM, see e.g. \cite[Theorem 6.15]{etheridge2000introduction} and see also \cite[Theorem IV.7]{le1999spatial}. 
\begin{proposition}[The dimension of the support and the range of SBM]\label{proposition dimension range and support sbm}
Let $\mu \in \mathcal{M}_F(\mathbb{R}^d)$ be a non-zero finite measure, let $\gamma > 0$ and let $(\mathcal{X}_s, s \geq 0)$ be an $\text{SBM}(\gamma)$ started from $\mu$. Then we have that for $y>0$, that
\begin{equation*}
    \text{dim}(\text{supp}(\mathcal{X}_y)) = 2 \wedge d \hspace{0.5cm} \text{ 
  on   } \{\mathcal{X}_y \neq 0\} \text{    a.s.}
\end{equation*}
Furthermore, for $\delta>0$ we have that
\begin{equation*}
    \text{dim}\Big(\overline{\bigcup_{y \geq \delta}\text{supp}(\mathcal{X}_y)}\Big) = 4 \wedge d \hspace{0.5cm} \text{ 
  on   } \{\mathcal{X}_{\delta} \neq 0\} \text{    a.s.}
\end{equation*}
Finally we have that
\begin{equation*}
    \text{dim}(\mathcal{R}(\mathcal{X})) = 4 \wedge d \hspace{0.5cm} \text{a.s.}
\end{equation*}
\end{proposition}
The dimension of the range of $\text{SBM}(\gamma)$ in Proposition \ref{proposition dimension range and support sbm} follows almost immediately from dimension of $\overline{\bigcup_{y \geq \delta}\text{supp}(\mathcal{X}_y)}$. We have for all $k \in \mathbb{N}$ that the dimension of $\overline{\bigcup_{y \geq 1/k}\text{supp}(\mathcal{X}_y)}$ is $4 \wedge d$ which means that the dimension of the countable union $\bigcup_{k \in \mathbb{N}}\overline{\bigcup_{y \geq 1/k}\text{supp}(\mathcal{X}_y)}$ is also $4 \wedge d$ and this countable union is equal to $\mathcal{R}(\mathcal{X})$.
\par
The main aim of this paper is to better understand the nature of ooSBM and how it compares to SBM by calculating the Hausdorff dimension of the range and the support of ooSBM. Our first result addresses a conjecture given in \cite[Remark 1.13]{blath2023off} about the ranges of ooSBM and SBM. The conjecture states that the ranges of the two superprocesses are equal when suitably coupled. In Theorem \ref{theorem equality ranges} we show this is not the case for the definition of the range operator that we use in Definition \ref{definition support and range}. However, we also note in Remark \ref{remark alternative range operator} that the conjecture is true for the equality of the closure of the ranges of SBM and ooSBM.
\begin{theorem}[The range of ooSBM]\label{theorem equality ranges}
    Let $\mu_{\rm d}$, $\mu_{\rm a} \in \mathcal{M}_F(\mathbb{R}^d)$ be such that $\mu_{\rm d} + \mu_{\rm a}$ is a non-zero measure and let $\gamma$, $c$, $\Tilde{c}>0$. Let $\mathcal{Y}=(\mathcal{Y}^{(\rm d)}, \mathcal{Y}^{(\rm a)})$ be an $\text{ooSBM}(\gamma, c, \Tilde{c})$ with initial measure $(\mu_{\rm d}, \mu_{\rm a})$ and let $\mathcal{X}$ be an $\text{SBM}(\gamma)$ with initial measure $\mu_{\rm d} + \mu_{\rm a}$. Then there exists a coupling for $\mathcal{Y}$ and $\mathcal{X}$ such that 
    \begin{equation*}
        \mathcal{R}(\mathcal{Y}^{(\rm d)}) \cup \mathcal{R}(\mathcal{Y}^{(\rm a)}) = \text{supp}(\mu_{\rm d}) \cup \text{supp}(\mu_{\rm a}) \cup \mathcal{R}(\mathcal{X}) \hspace{0.25cm} \text{a.s.} 
    \end{equation*}
    In particular, we have that $\mathcal{R}(\mathcal{Y}^{(\rm d)}) \cup \mathcal{R}(\mathcal{Y}^{(\rm a)})$ is a closed set a.s. and is equal in distribution to 
    \newline
    $\text{supp}(\mu_{\rm d}) \cup \text{supp}(\mu_{\rm a}) \cup \mathcal{R}(\mathcal{X})$. Therefore
    \begin{equation*}
        \text{dim}(\mathcal{R}(\mathcal{Y}^{(\rm d)}) \cup \mathcal{R}(\mathcal{Y}^{(\rm a)})) = \text{dim}(\text{supp}(\mu_{\rm d})) \vee \text{dim}(\text{supp}(\mu_{\rm a})) \vee (4 \wedge d) \text{   a.s.}
    \end{equation*}
\end{theorem}
The dimension of the range follows from the range coupling and Proposition \ref{proposition dimension range and support sbm}. We note that for the pre-limiting processes BBM and ooBBM there is an immediate coupling for which the ranges are equal a.s.; however, as seen in Theorem \ref{theorem equality ranges}, the dimension of the range of ooSBM is not necessarily the same as the dimension of the range of SBM. This is because the range of SBM does not necessarily contain the support of the initial state unlike the range of ooSBM. We give a remark about an alternative definition of the range operator for which the range of SBM is equal to the range of ooSBM under a suitable coupling.
\begin{remark}[Alternative range operator]\label{remark alternative range operator}
Another potential definition of the range of a superprocess $\mathcal{W}$, which is again stated in Le Gall \cite{le1999spatial}, is $\overline{\mathcal{R}}(\mathcal{W}) := \overline{\bigcup_{s>0} \text{supp}(\mathcal{W}_s)}$. This is the closure of the range given in Definition \ref{definition support and range}, that is to say $\overline{\mathcal{R}}(\mathcal{W}) = \overline{\mathcal{R}(\mathcal{W})}$. In this case, the support of the initial measure $\text{supp}(\mathcal{W}_0)$ is contained in $\overline{\mathcal{R}}(\mathcal{W})$ (which is not necessarily the case for $\mathcal{R}(\mathcal{W})$). Therefore, under this alternative definition, we have the equality of the ranges of ooSBM and SBM, i.e. $\overline{\mathcal{R}}(\mathcal{Y}^{(\rm d)}) \cup \overline{\mathcal{R}}(\mathcal{Y}^{(\rm a)}) = \overline{\mathcal{R}}(\mathcal{X})$ a.s., when coupled as in Theorem \ref{theorem equality ranges}, and as conjectured in \cite[Remark 1.13]{blath2023off}.
\end{remark}
We now state our results concerning the support of ooSBM.
\begin{theorem}[The support of ooSBM]\label{theorem increasing support}
    Let $\mu_{\rm d}$, $\mu_{\rm a} \in \mathcal{M}_F(\mathbb{R}^d)$ be such that $\mu_{\rm d} + \mu_{\rm a}$ is a non-zero measure and let $\gamma$, $c$, $\Tilde{c}>0$. Let $\mathcal{Y}=(\mathcal{Y}^{(\rm d)}, \mathcal{Y}^{(\rm a)})$ be an $\text{ooSBM}(\gamma, c, \Tilde{c})$ with initial measure $(\mu_{\rm d}, \mu_{\rm a})$. Then we have that
    \begin{equation*}
        (\text{supp}(\mathcal{Y}_s^{(\rm d)}) \cup \text{supp}(\mathcal{Y}_s^{(\rm a)}))_{s \geq 0}
    \end{equation*}
    is an increasing process a.s. That is to say that a.s. the support of ooSBM is monotonically increasing at all times.
    \par
    Furthermore, we also have for any $s>0$ that
    \begin{equation*}
        \text{dim}\Big(\text{supp}(\mathcal{Y}_s^{(\rm d)}) \cup \text{supp}(\mathcal{Y}_s^{(\rm a)})\Big) = \text{dim}(\text{supp}(\mu_{\rm d})) \vee \text{dim}(\text{supp}(\mu_{\rm a})) \vee (4 \wedge d) \text{   a.s.}
    \end{equation*}
\end{theorem}
This provides a response to the conjecture in \cite[Remark 1.13]{blath2023off} that the union of the support of the active and dormant components of ooSBM should grow monotonically in time. The heuristic reasoning given behind this conjecture is that whatever spatial location the active component occupies a dormant sub-mass is deposited there and the size of this mass decays to $0$ but never hits $0$, and therefore that spatial location remains in the support of the dormant component for all subsequent times. We use the increasing support property of ooSBM to prove the dimension of the support of ooSBM. Theorem \ref{theorem increasing support} also provides an alternative proof that ooSBM does not die out a.s., which was given in \cite[Theorem 1.10]{blath2023off}, as the process clearly does not die out if the support is monotonically increasing. 
\par
We believe the results of Theorems \ref{theorem equality ranges} and \ref{theorem increasing support} should also apply to the dormant component $\mathcal{Y}^{(\rm d)}$ but not to the active component $\mathcal{Y}^{(\rm a)}$. The reason for this is that it was shown in \cite[Theorem 1.11]{blath2023off} that $(\mathcal{Y}^{(\rm a)}_s(\mathbb{R}^d))_{s \geq 0}$ hits $0$ with positive probability and at these times the support of $\mathcal{Y}^{(\rm a)}_s$ is clearly the empty set. On the other hand, we do show the monotonic increasing support holds for the dormant component for any pair of fixed times a.s. in Corollary \ref{corollary dormant support}.
\par
We use properties of the on/off Brownian snake to make calculations about  the expectations of the total mass of the process. For a measure $\nu \in \mathcal{M}_F(\mathbb{R}^d)$, let $\norm{\nu}$ denote the total mass of $\nu$, i.e. $\norm{\nu} \coloneqq \nu(\mathbb{R}^d)$. 
\begin{theorem}[Expectation of total mass]\label{theorem expectation total masses}
    Let $\mu_{\rm d}$, $\mu_{\rm a} \in \mathcal{M}_F(\mathbb{R}^d)$ be such that $\mu_{\rm d} + \mu_{\rm a}$ is a non-zero measure and let $\gamma$, $c$, $\Tilde{c}>0$. Let $\mathcal{Y}=(\mathcal{Y}^{(\rm d)}, \mathcal{Y}^{(\rm a)})$ be an $\text{ooSBM}(\gamma, c, \Tilde{c})$ with initial measure $(\mu_{\rm d}, \mu_{\rm a})$. Then, for any $s \geq 0$ we have
    \begin{equation*}
        \mathbb{E}(\norm{\mathcal{Y}^{(\rm d)}_s}) = \norm{\mu_{\rm d}} \frac{c+ \Tilde{c}e^{-s(c+\Tilde{c})}}{c+\Tilde{c}} + \norm{\mu_{\rm a}} \frac{c(1-e^{-s(c+\Tilde{c})})}{c+\Tilde{c}}; \hspace{0.1cm} \mathbb{E}(\norm{\mathcal{Y}^{(\rm a)}_s}) = \norm{\mu_{\rm d}} \frac{\Tilde{c}(1-e^{-s(c+\Tilde{c})})}{c+\Tilde{c}} + \norm{\mu_{\rm a}} \frac{ce^{-s(c+\Tilde{c})} + \Tilde{c}}{c+\Tilde{c}}.
    \end{equation*}
\end{theorem}
This is proven in Section \ref{section moments}. The expectations in Theorem \ref{theorem expectation total masses} are constructed from the transition semigroup of a two-state continuous-time Markov chain with transition rates $c$ and $\Tilde{c}$ which corresponds to transitions between the active state and the dormant state in our setting. A noteworthy aspect of Theorem \ref{theorem expectation total masses} is that while the total mass process $(\norm{\mathcal{Y}^{(\rm d)}_s} \cup \norm{\mathcal{Y}^{(\rm a)}_s})_{s \geq 0}$ converges to $0$ a.s., \cite[Proposition 3.6]{blath2023off}, the expectations for the total mass of the active and dormant components do not converge to $0$ as $s \rightarrow \infty$. In fact, it can be seen from Theorem \ref{theorem expectation total masses} that the expectation of the total mass of the union of the dormant and active components of ooSBM is constant for all positive times. However, this is not too unexpected as this is a known property of critical branching processes; for example, it is also the case for SBM $\mathcal{X}$ started from a finite measure $\mu$ that the total mass converges to zero a.s. but the expectation of total mass of SBM stays constant, i.e. $\mathbb{E}(\mathcal{X}_s) = \norm{\mu}$ for all $s>0$. 
\par
As mentioned above, our main proof technique utilises the Brownian snake with discontinuous spatial motion as defined in \cite{BertoinLeGallLeJan1997} where we study a Brownian snake with spatial motion given by a Brownian component and a second independent component given by a compound Poisson process with unit drift. This snake has similarities with the Poisson snake introduced in Abraham and Serlet \cite{abraham2002poisson} which is a Brownian snake with spatial motion given by a Poisson process. We also note a related problem in Delmas and Dhersin \cite{delmas2003super} where they used the Brownian snakes along with time-changes to show the existence of superprocesses related to SBM that they refer to as SBM with interactions. In their work they apply a random time-change to each of the paths of the path-valued Brownian snake. Their work does not directly apply to our setting as they require the time change $\phi$ to satisfy $\phi'(s)>0$ for all $s>0$ and all paths and we would require $\phi'(s)=0$ during the dormancy periods of paths and so we instead use the aforementioned literature on Brownian snakes with discontinuous spatial motion.
\subsection{Proof idea and structure of the paper}\label{subsection proof idea}
The starting point of the following construction will be a so-called bivariate snake which we define in Theorem \ref{theorem bivariate brownian snake}. This is a Brownian snake that has two motion components: a Brownian motion, and a subordinator which consists of a compound Poisson process plus a unit drift and will be used to describe the dormancy behaviour. From this richer Brownian snake, we will then be able to construct an ooSBM and a classical SBM, which are automatically coupled. 
\par
To achieve this we identify a collection of ooBBMs inside the bivariate snake that converge a.s. to an ooSBM. Specifically, we show that the finite-dimensional distributions of the limiting process coincide with the finite-dimensional distributions of ooSBM. We only construct $\text{ooSBM}(4, c, \Tilde{c})$ as it is a simple rescaling to obtain $\text{ooSBM}(\gamma, c, \Tilde{c})$ for general $\gamma > 0$ and we explain this rescaling in the proof of Theorem \ref{theorem equality ranges} in Section \ref{section range coupling proof}. In order to construct the ooBBM from the bivariate snake we will consider the contour function corresponding to the excursion and perform an $h$-erasure to obtain the contour function of a Galton-Watson tree. This Galton-Watson tree will serve as the genealogy of the ooBBM, which we obtain by adding the respective spatial motion processes to it. In Section \ref{section snake identification bbm} we define the \emph{$h$-erased contour function} for an excursion and use this to construct BBM and ooBBM from the same on/off Brownian snake in Proposition \ref{proposition identification oobmm h erased tree}. One of the core arguments for the representation of ooSBM via the on/off Brownian snake is the construction of suitable notions of active local time and dormant local time, which corresponds to the local times of the excursions at active and dormant levels.
\par
Since the rescaling of the ooBBM to the limit ooSBM, see \cite[Theorem 1.5]{blath2023off}, is the same as the rescaling for the local time in the classical downcrossing argument for approximation of local time, see e.g. \cite{morters2010brownian}, there exists a representation of the limiting ooSBM in terms of the limiting active/dormant local time. To show this, in Lemma \ref{lemma alternative characterisation ooBBM from snake} we give an alternative characterisation of the ooBBM and BBM constructed in Proposition \ref{proposition identification oobmm h erased tree} which we use in Theorem \ref{theorem ooSBM ffd convergence}, along with the notion of an exit local time given in Proposition \ref{proposition snake reflecting brownian motion}, to construct exit measures with the distributions of ooSBM at fixed times. Via standard arguments given in Proposition \ref{proposition ooSBM continuous modification} one can then also obtain a continuous version of this limiting superprocess. As a result we get an immediate coupling of ooSBM with classical SBM, without needing to couple the prelimiting ooBBMs and BBMs. This is particularly useful for the range coupling result in Theorem \ref{theorem equality ranges}, see Section \ref{section range coupling proof} for the proof, as the range operator is in general lacking continuity properties and therefore it was not clear that the range coupling would survive a limiting procedure. We use the finite propagation of SBM to prove the increasing support property and the dimension of the support of ooSBM in Theorem \ref{theorem increasing support}. Finally, in Section \ref{section moments}, we prove Theorem \ref{theorem expectation total masses} and use the snake to calculate the expected total dormant and active masses.
\section{Construction of the on/off Brownian snake}\label{section snake construction}
We recall the set-up given in \cite[Section 4]{BertoinLeGallLeJan1997} for a Brownian snake with discontinuous spatial motion. Let $\xi$ be a c\`adl\`ag Borel right Markov process with values in a complete and separable metric space $(E, d_E)$ and let the process $\xi$ be defined on the Skorokhod space $\mathbb{D}([0,\infty), E)$. We fix constants the $\gamma$, $c$, $\Tilde{c}>0$ throughout this section and the rest of this paper with the aim to construct $\text{ooSBM}(\gamma, c, \Tilde{c})$. 
\par
We define a killed path in $E$ to be a c\`adl\`ag function $w:[0,\zeta) \rightarrow E$, where we call $\zeta = \zeta(w) \in (0,\infty)$ the lifetime of $w$. Furthermore, for all $x \in E$, we call $x$ a killed path with lifetime $0$. Let $\mathcal{W}$ be the set of all killed paths and for $x \in E$ let $\mathcal{W}_x$ denote the set of all the killed paths $w$ such that $w(0) = x$.
\par
We define, for $w$, $w' \in \mathcal{W}$ 
\begin{equation}\label{equation snake killed path metric}
    d(w,w') = d_E(w(0), w'(0)) + |\zeta(w) - \zeta(w')| + \int_{y=0}^{\zeta(w) \wedge \zeta(w')} (d^{\text{Sko}}(w|_{[0,y]}, w'|_{[0,y]}) \wedge 1) dy,
\end{equation}
where $w|_{[0,y]}$ denotes the restriction of $w$ to $[0,y]$. This defines a metric $d$ on $\mathcal{W}$, and $(\mathcal{W}, d)$ is a complete and separable metric space. However we note that a killed path does not necessarily have a left limit at $\zeta(w)$. 
\par
Consider $w \in \mathcal{W}_x$ with lifetime $\zeta(w)>0$, i.e. $w=(w(y), 0 \leq y < \zeta(w))$. If $0 \leq a < \zeta(w)$, and $b \geq a$, we let $Q_{a,b}(w, dw')$ be the unique probability measure on $\mathcal{W}_x$ such that
\begin{enumerate}
    \item $\zeta(w') = b$, $Q_{a,b}(w, dw')$ a.s.
    \item $w'(y) = w(y)$, for all $y \in [0,a]$, $Q_{a,b}(w, dw')$ a.s.
    \item the law under $Q_{a,b}(w, dw')$ of $(w'(a+y), 0 \leq y < b-a)$ is the law of $(\xi_y, 0 \leq y < b-a)$ under $P^\xi_{w(a)}$.
\end{enumerate}
When the left limit $w(a-)$ exists we can extend this definition to the case $a=\zeta(w)$ and in this instance we modify the third point so that the law of $(w'(a+y), 0 \leq y < b-a)$ is the law of $(\xi_y, 0 \leq y < b-a)$ under $P^\xi_{w(a-)}$.
\par
We write $\delta_x$ to denote the Dirac measure in $x$. Note that $Q_{0,0}(w, dw')= \delta_x(dw')$ and $Q_{0,b}(w, dw')$ is the law of $(\xi_y, 0 \leq y < b)$ under $P^\xi_x$ which we denote by $P_x^{\xi,b}(dw')$. We also set $Q_{0,b}(x, dw')=P_x^{\xi,b}(dw')$. 
\par
Let $R=(R_t, t \geq 0)$ be a reflecting Brownian motion on $[0,\infty)$ and for $r \geq 0$ let $P^R_r$ denote the law of $R$ with $R_0 = r$. We denote by $\theta^r_t(dadb)$ the joint distribution of $(\inf(\{R_{t'}: t' \in [0,t]\}), R_t)$ under $P^R_r$:
\begin{equation*}
    \theta^r_t(dadb) = \frac{2(r+b - 2a)}{\sqrt{2\pi t^3}} \exp -\Big(\frac{(r+b-2a)^2}{2t}\Big) \mathbbm{1}_{(0 < a < r \wedge b)} da db + \frac{\sqrt{2}}{\sqrt{\pi t}} \exp - \Big(\frac{(r + b)^2}{2t}\Big) \mathbbm{1}_{(0 < b)} \delta_0(da)db.
\end{equation*}
For $r>0$ we define the set $\mathcal{G}_r \subset \mathcal{C}([0,\infty), [0,\infty))$ of functions $g:[0,\infty) \rightarrow [0,\infty)$ with $g(0)=r$ and the following properties:
\begin{enumerate}
    \item for some $t>0$ we have $\inf(\{g(t'), 0 \leq t' \leq t\})<g(0)$;
    \item for all $\alpha \in (0,1/2)$ we have that $g$ has H\"older continuous paths with exponent $\alpha$. 
\end{enumerate}
In our setting $E=\mathbb{R}^d \times [0,\infty)$ and $d_E$ is the Euclidean metric. Furthermore our spatial motion $\xi$ is of the form $(B, S)$ where $B$ and $S$ are independent, $B$ is a $d$-dimensional Brownian motion and $S$ is a compound Poisson subordinator with unit drift and jump intensity $\pi(ds) = c\Tilde{c}e^{-\Tilde{c}s} ds$. For $x \in \mathbb{R}^d \times [0,\infty)$ let $P_x^\xi$ be the distribution on $\mathbb{D}([0,\infty), \mathbb{R}^d \times [0,\infty))$ of $\xi$ with initial value $x$.
\begin{theorem}[Bivariate Brownian snake]\label{theorem bivariate brownian snake}
Let $\xi$ be a c\`adl\`ag Borel right Markov process with values in a complete and separable metric space $(E, d_E)$ and let the process $\xi$ be defined on the Skorokhod space $\mathbb{D}([0,\infty), E)$. Let $x \in E$ and let $w \in \mathcal{W}_x$. Then there exists a continuous time-homogeneous strong Markov $\mathcal{W}_x$-valued process $X=(X_t, t \geq 0)$ with $X_0=w$ with the following properties:
\begin{enumerate}
\item The transition kernels are given by the formula
\begin{equation*}
    K_t(w', dw'') = \int_{[0,\infty)^2} \theta^{\zeta(w')}_t(dadb) Q_{a,b}(w', dw'').
\end{equation*}
Let $\mathbb{P}_w$ denote the distribution for this snake on $\mathcal{C}([0,\infty),\mathcal{W}_x)$. 
\item Furthermore the process $R = (R_t, t \geq 0)$ where $R_t \coloneqq \zeta(X_t)$, which we call the \emph{contour process} for the snake, is a reflecting Brownian motion on $[0,\infty)$ (started from $\zeta(w)$); 
\item Finally let $g \in \mathcal{G}_{\zeta(w)}$. Then the conditional distribution of $(X_t, t \geq 0)$ given the contour process $R = g$ is that of a continuous time-inhomogeneous Markov process. Let $w' \in \mathcal{W}_x$, suppose that either $w'(\zeta(w')-)$ exists or $\inf(\{g_u, t \leq u \leq t'\})<g_t$ for some $t'>t$. Then we can explicitly state the transition kernel as
\begin{equation*}
    K^g_{t,t'} (w', dw'') = Q_{\inf(\{g_u, t \leq u \leq t'\}), g_{t'}}(w', dw''),
\end{equation*}
and this collection of time-inhomogeneous Markov processes form a family of regular conditional distributions for the strong Markov process defined above.
\end{enumerate}
The process $(X_t, t \geq 0)$ is a \emph{Brownian snake} or more simply a \emph{snake}. In the setting where $E=\mathbb{R}^d \times [0,\infty)$ and the spatial motion of $\xi$ is of the form $(B, S)$ where $B$ and $S$ are independent, $B$ is a $d$-dimensional Brownian motion and $S$ is a compound Poisson subordinator with unit drift and jump intensity $\pi(ds) = c\Tilde{c}e^{-\Tilde{c}s} ds$, we call this the \emph{bivariate Brownian snake} or more simply the \emph{bivariate snake}.
    \begin{proof}
        The first two points are stated in \cite[Proposition 5]{BertoinLeGallLeJan1997}. Their proof quotes the proof of Le Gall \cite[Theorem 1.1]{legall1993class}. They explain how the proof in \cite{legall1993class} generalises to their more general setting in \cite{BertoinLeGallLeJan1997} and as the third point is an intermediate result in the proof of \cite{legall1993class} we can conclude this additional result by the same logic. In particular, in \cite{legall1993class}, the proof first considers a fixed contour process, in our case $g \in \mathcal{G}_r$, then constructs a probability measure for the snake by stating the marginals for a fixed number of time points $t_1 < \dots < t_n$ using the kernels $K_{t_i, t_{i+1}}^g$ and then applies the Kolmogorov-\v{C}entsov theorem. This applies and suffices in our setting, as $K_{t,t'}^g(w',dw'')$-a.e. $w''$ has a left limit at $\zeta(w'')$ and $P^R_r(\mathcal{G}_r) = 1$.
    \end{proof}
\end{theorem}
For $x \in E$, $w \in \mathcal{W}_x$, let $\mathbb{E}_w$ denote the expectation operator for the Brownian snake started at $w$. The transition kernels for the strong Markov process above are such that for every $0 \leq t' \leq t$ and for every non-negative measurable functional $\Phi$ on $\mathcal{W}_x$, we have
\begin{equation*}
    \mathbb{E}_w(\Phi(X_t) | (X_r, r \leq t')) = \int K_{t-t'}(X_{t'}, dw)\Phi(w), \hspace{0.25cm} \mathbb{P}_w-\text{a.s.}
\end{equation*}
where $\mathbb{E}_w$ is the expectation operator for $X$ under $\mathbb{P}_w$. 
\par
For $x \in E$, $w \in \mathcal{W}_x$ and $g \in \mathcal{G}_{\zeta(w)}$, we define $\mathbb{P}_{w | g}$ to be the probability measure for the time-inhomogeneous Markov process with contour process $g$ on the space of measure-valued functions. Let $\mathbb{E}_{w|g}$ be the expectation operator. Then for any bounded continuous $\mathcal{W}_x$-valued function $\Phi$ we have that
\begin{equation*}
    \mathbb{E}_{w|g}(\Phi(X_t)|(X_r, r \leq t')) = \int K^g_{t',t}(X_{t'}, dw')\Phi(w'), \hspace{0.25cm} \mathbb{P}_{w|g}-\text{a.s.}
\end{equation*}
Note we have for $x \in E$ and $w \in \mathcal{W}_x$ that $\mathbb{P}_w(\cdot) = \int_g P^R_{\zeta(w)}(dg)\mathbb{P}_{w | g}(\cdot)$. We can also start the process from $x$, the path of length $0$ and we write $\mathbb{P}_x$ for this distribution and $\mathbb{E}_x$ for this expectation operator.
\par
In our setting for $X_t=(B_t,S_t)$ we cannot exclude the case where $\lim_{y \rightarrow \zeta(X_t)-}S_t(y)=\infty$ and so we must take care when we deal with this component of the bivariate snake. It is noted in \cite[Section 4]{BertoinLeGallLeJan1997} that $\mathbb{P}_w$-a.s. for every $t<t'$, the killed paths $X_t$, $X_{t'}$ coincide for $r<m(t,t') \coloneqq \inf(\{\zeta(X_{r'}): r' \in [t,t']\})$ and that they also coincide at $r=m(t,t')$ provided that $m(t, t') < \zeta(X_t) \wedge \zeta(X_{t'})$. We use this property throughout and we refer to this as the \emph{snake property}. 
\par
In the setting of Le Gall \cite{le1999spatial} the spatial motion is continuous and he works with stopped paths $\mathcal{W}'$ where $w \in \mathcal{W}'$ is a continuous map $w:[0,\zeta] \rightarrow E' \coloneqq \mathbb{R}^d$, where $\zeta=\zeta(w)<\infty$ as before. There is a complete and separable metric space $(\mathcal{W}', d')$ with metric $d'$ defined by
\begin{equation*}
    d'(w, w') = |\zeta(w) - \zeta(w')| + \sup\big(\{d_{E'}(w(y \wedge \zeta(w), w'(y \wedge \zeta(w')), y \geq 0\}\big). 
\end{equation*}
which is a stronger metric than $d$ on $\mathcal{W}'$ obtained by replacing $E$ and $\mathcal{W}$ by $E'$ and $\mathcal{W}'$ in \eqref{equation snake killed path metric}. We give a brief discussion about this continuous map setting in Remark \ref{remark head of the snake} below.
\begin{remark}[Head of the snake]\label{remark head of the snake}
    In the setting of Le Gall \cite{le1999spatial} the spatial motion is continuous. In this setting there exists a Brownian snake $(W_t)_{t \geq 0}$, with distribution $\mathbb{P}_{w'}'$ on $\mathcal{C}([0,\infty), \mathcal{W}')$, that is continuous with respect to the metric $d'$ and has a continuous \emph{head of the snake} process $(\hat{W}_t = W_t(\zeta(W_t)))_{t \geq 0}$ $\mathbb{P}'_{w'}$-a.s.
    \par
    We can also define a snake $(V_t)_{t \geq 0}$ with spatial motion given by a subordinator $S$ with the same contour process that is continuous with respect to the metric $d$ on a common probability space. We can show that the path-valued process with both spatial components $(W_t,V_t)_{t \geq 0}$ is continuous with respect to the metric $d$ on the product space. By a suitable coupling with the bivariate snake $X$ given in Theorem \ref{theorem bivariate brownian snake} we have that $(W,V)$ is a modification of $X$ and both processes are path-continuous with respect to the metric $d$. 
    \par
    We can conclude, for all $w \in \mathcal{W}$ whose left limit at $\zeta(w)$ exists, that $\mathbb{P}_w$-a.s. the Brownian component $(B_t, t \geq 0)$ of the snake $X$ has a head $\hat{B}_t = B_t(\zeta(B_t)-)$ for all $t \geq 0$ and that $(\hat{B}_t, t \geq 0)$ is continuous. This follows as two continuous processes that are modifications of one another are indistinguishable, see e.g. Revuz-Yor \cite[Definition I.1.7]{revuz2013continuous}.
\end{remark}
We explore more properties of the bivariate Brownian snake in Section \ref{section snake local time measure}, in particular the notion of an exit local time. Our next aim is to construct a Poisson random measure on the space of snakes which we give in Definition \ref{definition oo snake}. In order to do this we require the excursion measure $n$ for reflecting Brownian motion and we give a summary of this excursion measure here. We define the space $\mathcal{E} \subset \mathcal{C}([0,\infty),[0,\infty))$ of excursions as continuous functions $f$ on $[0,\infty)$ that take positive values on some interval $(0,\sigma)$ and the value $0$ on $\{0\} \cup [\sigma, \infty)$ for some $\sigma \in (0,\infty)$; we write this $\sigma$ as $\sigma(f)$ and call it the \emph{lifetime of the excursion $f$}. For an excursion $f = (f_t, t \geq 0)$ and for $y>0$ we define $T_y(f) \coloneqq \inf (\{t \in [0,\infty) : f_t = y\})$ and we call $\sigma(f)$.
\begin{definition}[Excursion measure $n$ for reflecting Brownian motion]\label{definition brownian excursion}
    We define the excursion measure $n$ for reflecting Brownian motion $R$ on $[0,\infty)$) as the unique sigma-finite measure on $\mathcal{E}$ with the following properties, see e.g. \cite[Chapter XII]{revuz2013continuous}.
    \begin{enumerate}
        \item For all $y \in (0,\infty)$ we have $n(T_y(f)<\infty) = 1/2y$.
        \item For all $y \in (0,\infty)$, the conditional distribution $n(\cdot| T_y(f) < \infty)$ is the distribution of the following process $Z = (Z_t)_{t \geq 0}$. We have that $(Z_t)_{0 \leq t \leq T_y(Z)}$ has the law of a Bessel process of dimension $3$ started from $0$, and let us denote this by $P_0^{\text{Bes}(3)}$. Then the process $(Z_{T_y(Z) + t})_{t \geq 0}$ is independent of $(Z_t)_{0 \leq t \leq T_y(Z)}$ and has the law of a Brownian motion on $[0,\infty)$ started at $y$ and absorbed when it hits $0$. The Bessel process of dimension $3$ is a diffusion with infinitesimal drift $\mu(y) = 1/y$ and infinitesimal variance $v(y) = 1$, see e.g. Revuz-Yor \cite[Chapter XI]{revuz2013continuous}.
    \end{enumerate}
\end{definition}
Note that, for $x \in E$, we can define an excursion measure for the bivariate snake denoted by $\mathbb{N}_x$ and given by $\mathbb{N}_x(\cdot) = \int_f n(df)\mathbb{P}_{x | f}(\cdot)$. This normalisation does not affect the conditional distributions $n(\cdot| T_y(f) < \infty)$, $\mathbb{N}_x(\cdot| T_y(f) < \infty) = \int_f n(df| T_y(f) < \infty)\mathbb{P}_{x | f}(\cdot)$.
\par
We now return to the spatial process $\xi$ itself and prove a lemma that we use later.
\begin{lemma}[On/off Brownian motion from $\xi$]\label{lemma single particle motion with age process}
    Let $b \in \mathbb{R}^d$ and consider a process $\xi=(\xi_y, y \geq 0)$ with distribution $P^\xi_{(b,0)}$. Let $B$ and $S$ be the components of $\xi$, so that $\xi=(B,S)$. Let
    \begin{equation*}
        H^{(s)}(\xi) \coloneqq \inf \{ y \in [0,\infty): S(y) \geq s\} 
    \end{equation*}
    and let $A^{(s)}(\xi) \coloneqq 1$ if $H^{(s+\varepsilon)}(\xi)>H^{(s)}(\xi)$ for all $\varepsilon>0$ and $A^{(s)}(\xi) \coloneqq 0$ otherwise. Then the process $(B(H^{(s)}(\xi)), A^{(s)}(\xi))_{s \geq 0}$ is an $\text{ooBM}(c, \Tilde{c})$ started from $(b,1) \in \mathbb{R}^d \times \{0,1\}$, i.e. a single particle that does not die or split and that switches between dormant and active states (which is equivalent to the position of the single particle of an $\text{ooBBM}(0,c,\Tilde{c})$ started at $\delta_{(b,1)}$).
    \begin{proof}
        Let $C=(C_s, s \geq 0)$ be given by $C_s \coloneqq B(H^{(s)}(\xi))$ and let $\Tilde{A}=(\Tilde{A}_s, s \geq 0)$ be given by $\Tilde{A}_s \coloneqq A^{(s)}(\xi)$. We show that $(C, \Tilde{A})$ has the correct spatial motion and dormancy dynamics. Let
        \newline
        $T\coloneqq \inf(\{y:S(y-) \neq S(y)\})$. The time $T$ represents the first discontinuity point of $S$ (and $\xi$). We have that 
        \begin{equation*}
            C_s = B_s, \hspace{0.5cm} \Tilde{A}_s = 1, \hspace{0.5cm}  0 \leq s < T.
        \end{equation*}
        Let $T'\coloneqq S(T) - S(T-)$. Then we have that 
        \begin{equation*}
            C_s = B_T, \hspace{0.5cm} \Tilde{A}_s = 0, \hspace{0.5cm} T \leq s < T + T'.
        \end{equation*}
        Now by construction we have that $T \sim \text{Exp}(c)$ and $T' \sim \text{Exp}(\Tilde{c})$ which gives the correct dormancy dynamics. The spatial motion is also correct as $C$ evolves with the dynamics of a Brownian motion when $\Tilde{A}=1$ and $C$ does not move when $\Tilde{A}=0$. The result follows by applying the same argument to each subsequent discontinuity point of $S$.
    \end{proof}
\end{lemma}
We now define what we call the \emph{age process} for a snake.
\begin{definition}[Age process]\label{definition age process}
    Fix $s>0$, $x \in E$ and $w \in \mathcal{W}_x$. Consider a bivariate snake $X = (X_t, t \geq 0)$ started from $w$. For $t \geq 0$ let $B_t$ and $S_t$ be the Brownian and subordinator components of $X_t$, so that $X_t = (B_t, S_t)$. Let $R = (R_t, t \geq 0)$ be the contour process for $X$ given by, for $t \geq 0$, $R_t \coloneqq \zeta(X_t)$. For $t \geq 0$ we define 
    \begin{equation*}
        H^X_t(s) \coloneqq \inf(\{y \in [0,R_t): S_t(y) \geq s \} \cup \{R_t\})
    \end{equation*}
    and we note $H^X_t(s) = H^{(s)} \circ X_t$.
    We say that $t$ is active at (time) $s$ if $H^X_t(s + \varepsilon)>H^X_t(s)$ for all $\varepsilon > 0$ and $t$ is dormant at (time) $s$ otherwise. We define a function $A^X_t:[0,\infty) \rightarrow \{0,1\}$ which is $A^X_t(s) = 1$ if $t$ is active at $s$ and $A^X_t(s) = 0$ otherwise (note again that $A^X_t(s) = A^{(s)} \circ X_t$). Finally we define what we call the residual process $(\text{Res}^X_t(s))_{t \geq 0}$ at $s$ by $\text{Res}^X_t(s) \coloneqq R_t - H^X_t(s)$. We drop the superscript $X$ notation and write $H_t \coloneqq H^X_t$, $A_t \coloneqq A^X_t$, $\text{Res}_t \coloneqq \text{Res}^X_t$ when the snake $X$ is clear from context.
\end{definition}
As with the Brownian snake for SBM, the parameter $t \geq 0$ in Definition \ref{definition age process} can be interpreted as a particle index for a continuum of particles and we can imagine integrating over the collection of these particles to obtain the measure-valued process ooSBM. If we consider a graph of the contour function, then heuristically the parameter $t$ would be on the $x$-axis and the parameter $s$ moves up the page in the direction of the $y$-axis (see Figure \ref{figure snake age process function}). We now state some properties of the age process.
\begin{lemma}[Structure of the age process/stopping line]\label{lemma age process properties}
    Let $s>0$, let $x \in E$, let $w \in \mathcal{W}_x$. For $X$ a bivariate snake under $\mathbb{P}_w$ (respectively $\mathbb{N}_x$) we have that $(H^X_t(s), t \geq 0)$, $(\text{Res}^X_t(s), t \geq 0)$ are both continuous $\mathbb{P}_w$-a.s. (respectively $\mathbb{N}_x$-a.e.). Let $R=(R_t, t \geq 0)$ be the contour process defined by $R_t \coloneqq \zeta(X_t)$. Define $V_s \coloneqq \{t \geq 0: \text{Res}^X_t(s) > 0 \}$. We have that $V_s$ is an open set $\mathbb{P}_w$-a.s. ($\mathbb{N}_x$-a.e.). Also, for the open set $V_s = \cup_{j \in J_s} (a_j,b_j)$, where the sets $(a_j,b_j)$ are disjoint open intervals over some countable index set $J_s$, we have for each $j \in J_s$ the function $t \mapsto H^X_t(s)$ is constant on $[a_j,b_j]$ $\mathbb{P}_w$-a.s. ($\mathbb{N}_x$-a.e.).
    \begin{proof}
        Let $t>0$. We consider two cases. First assume $H^X_t(s) < R_t$ and let $\delta \coloneqq R_t - H^X_t(s)$. Then, by continuity of the contour process $(R_{t'}, t' \geq 0)$ there exists $\delta'>0$ such that for $t' \in [t-\delta', t+\delta']$ we have $R_{t'}>H^X_t(s)+ \delta/2$. In this case, by the snake property of $X$, we have that $t' \mapsto H^X_{t'}(s)$ is constant on $[t-\delta', t+\delta']$ and so $t' \mapsto H^X_{t'}(s)$ is continuous at $t$. It also shows us that $V_s$ is open.
        \par
        Before we show continuity for the case where $H^X_t(s) = R_t$, we show the final part of the lemma. Let $(a_j,b_j)$ be an open interval of $V_s$. Let $(c_j, d_j)$ be the largest open interval with $(a_j,b_j) \subseteq (c_j,d_j)$ and for which constancy extends, i.e. such that $t \mapsto H^X_t(s)$ is constant on $(c_j,d_j)$. It follows by the same arguments as in the previous paragraph that $(c_j,d_j)=(a_j,b_j)$. Moreover, we have that $X_{a_j}|_{[0,R_{a_j})} = X_t|_{[0,R_{a_j})}$ for $t \in (a_j,b_j)$ by the snake property and so we can conclude that the constancy extends to the boundary point $a_j$. By the same logic we have that the constancy extends to the other boundary point $b_j$.
        \par
        Now we assume that $H^X_t(s) = R_t$ and we show continuity at $t$ in this case. We show left- and right-continuity separately. Let $(t_n, n \in \mathbb{N})$ be a non-decreasing sequence that converges to $t$, i.e. $t_n \uparrow t$. If $t_n \in V_s$ there exists $j_n \in J_s$ such that $t_n \in ((a_{j_n}, b_{j_n}))$, $t' \mapsto H^X_{t'}(s)$ is constant on $[a_{j_n}, b_{j_n}]$ and $a_n$, $b_n \notin V_s$. Define the sequence $t'_n$ by $t'_n = b_{j_n}$ for $t_n \in V_s$ and $t'_n=t_n$ otherwise. Note that for this new sequence we have $t'_n \uparrow t$. Therefore we have 
        \begin{equation*}
            \lim_{n \rightarrow \infty} H^X_{t_n}(s) = \lim_{n \rightarrow \infty} H^X_{t'_n}(s) = \lim_{n \rightarrow \infty} R_{t'_n} = R_t = H^X_t(s).
        \end{equation*}
        This shows left-continuity at $t$. The same argument holds for right-continuity where we use $a_n$ in place of $b_n$ to define the sequence $(t'_n, n\geq 1)$.
    \end{proof}
\end{lemma}
We give a remark about the age process/stopping line.
\begin{remark}[Stopping lines]
    For a bivariate snake $X$ and a fixed time $s$, the age process $(H^X_t(s))_{t \geq 0}$ can be interpreted as a stopping line for $X$, see e.g. Chauvin \cite{chauvin1991product} where each path $X_t$ is stopped individually at the stopping time for their path. The snake property ensures the dependence required for the stopping times to form a stopping line.
    \par
    Furthermore, there are similarities to the notion of a frost for fragmentation processes as defined in Bertoin \cite{bertoin2002self}. In the setting of \cite{bertoin2002self}, we have a fragmentation process where each fragment is frozen independently of the others and in other setting we have that each path $X_t$ is stopped individually at the stopping time, but the snake property means that all paths that correspond to the same fragment in the corresponding fragmentation process are stopped at the same time. 
\end{remark}
At this point we briefly mention a possible confusion due to notation in this section. An element of $\mathbb{R}^d \times \{0,1\}$ denotes the position of a particle of an ooBBM as defined in Definition \ref{definition bbm and oobbm} and we use this in Lemma \ref{lemma single particle motion with age process}. On the other hand, an element of $E=\mathbb{R}^d \times [0,\infty)$ denotes the starting point of a bivariate snake as in Theorem \ref{theorem bivariate brownian snake}, and we use this latter notation in the next definition.
\begin{definition}[The on/off Brownian snake]\label{definition oo snake}
    Consider finite measures $\mu_{\rm d}$ and $\mu_{\rm a}$ on $\mathbb{R}^d$. Let $(\Omega, \mathcal{F}, \mathbb{P}_{(\mu_{\rm d},\mu_{\rm a})})$ be a probability space, let $\mathbb{E}_{(\mu_{\rm d},\mu_{\rm a})}$ be the expectation operator, and suppose there exists on this space a Poisson random measure $\sum_{i \in I} \delta_{f^i}$ on the space of excursions in $\mathbb{R}_+$ with intensity measure $(\norm{\mu_{\rm d}} + \norm{\mu_{\rm a}})n(\cdot)$, where $n$ is the It\^o excursion measure for Brownian motion, for which we can mark the excursions in the following way. For $i \in I$:
    \begin{enumerate}
        \item with probability $\norm{\mu_{\rm d}}/(\norm{\mu_{\rm d}} + \norm{\mu_{\rm a}})$ we mark $f^i$ with $(b_i,s_i)$ where $b_i$ is sampled from the normalised probability distribution $\Tilde{\mu}_d \coloneqq \mu_{\rm d}/\norm{\mu_{\rm d}}$ and $s_i$ is sampled independently from the $\text{Exp}(\Tilde{c})$ distribution (the exponential distribution with density $\Tilde{c}e^{-\Tilde{c}z}$, $z>0$);
        \item with probability $\norm{\mu_{\rm a}}/(\norm{\mu_{\rm d}} + \norm{\mu_{\rm a}})$ we mark $f^i$ with $(b_i,s_i)$ where $b_i$ is sampled from the normalised probability distribution $\Tilde{\mu}_a \coloneqq \mu_{\rm a}/\norm{\mu_{\rm a}}$ and $s_i = 0$.
    \end{enumerate}
    We then mark the pair $((b_i,s_i),f^i)$ with a snake $x^i$ where the snake is sampled from the probability distribution $\mathbb{P}_{(b_i,s_i)|f^i}(\cdot)$. Let
    \begin{equation}\label{equation snake poisson random measure}
        \sum_{i \in I} \delta_{((b_i,s_i), f^i, x^i)}
    \end{equation}
    be the Poisson random measure we obtain. We call this \emph{the on/off Brownian snake (with initial condition $(\mu_{\rm d}, \mu_{\rm a})$)}. 
\end{definition}
We have abused terminology somewhat here as what we have defined as an on/off Brownian snake is not a path-valued process that evolves in time (as is the usual setting for snakes) but instead is a Poisson random measure of path-valued processes that evolve in time. This is just so that we refer to a single snake when we refer to a single superprocess. 
\section{Identification of ooBBM with the on/off Brownian snake}\label{section snake identification bbm}
In this section we construct, for all $\varepsilon > 0$, both  $\text{BBM}(4/\varepsilon)$ and $\text{ooBBM}(4/\varepsilon, c, \Tilde{c})$ from the same on/off Brownian snake defined in Definition \ref{definition oo snake}. Recall the stopping time $T_h(f)$ for $h>0$ for an excursion function $f \in \mathcal{E}$ from the previous section. We extend this definition to a $\mathcal{W}$-valued $X$ and define $T_h(X) \coloneqq \inf(\{t \in [0,\infty): \zeta(X_t) = h \})$. We also extend the definition of the lifetime $\sigma$, defined for excursions, and write $\sigma(X)$ for the lifetime of the contour process $(\zeta(X_t), t \geq 0)$ in the same way and call this the \emph{lifetime of the snake}. We give a decomposition for the distribution $\mathbb{N}_x(\cdot | T_h<\infty)$ on $\mathcal{C}([0,\infty), \mathcal{W})$.
\begin{proposition}\label{proposition bessel and brownian snakes}
    Let $x \in E$ and let $h>0$. The snake $X$ under the probability measure $\mathbb{N}_x(\cdot | T_h<\infty)$ can be expressed as the concatenation of the following two snakes.
    \begin{enumerate}
        \item $(\zeta(X_t), 0 \leq t \leq T_h(X))$ has the distribution $P_0^{\text{Bes}(3)}$. Conditional on $(\zeta(X_t), 0 \leq t \leq T_h(X))$ being given by $g \coloneqq (g_t)_{0 \leq t \leq T_h(g)}$, the process $(X_t, 0 \leq t \leq T_h(X))$ is a time-inhomogeneous snake with distribution $\mathbb{P}_{x|g}(\cdot)$.
        \item Conditional on $X_{T_h(X)}=w$ for an $E$-valued killed path $w$, the processes $(X_t, 0 \leq t \leq T_h(X))$ and $(X_{T_h(X)+t}, t \geq 0)$ are independent and the latter has law $\mathbb{P}_w(\cdot)$.
    \end{enumerate}
    \begin{proof}
        The proof follows almost immediately from Theorem \ref{theorem bivariate brownian snake} and the representation of the excursion measure $n$ for Brownian motion given in Definition \ref{definition brownian excursion}. The second part follows by the Markov property of the time-inhomogeneous snake $X$ conditionally given the contour process and of the decomposition $\mathbb{P}_w(\cdot) = \int_g P^R_{\zeta(w)}(dg)\mathbb{P}_{w|g}(\cdot)$.
    \end{proof}
\end{proposition}
\begin{remark}
    The \say{stopped Bessel snake} given by $(X_t, 0 \leq t \leq T_h(f))$ under $\mathbb{N}_x(\cdot|T_h<\infty)$ in Proposition \ref{proposition bessel and brownian snakes} is also a strong Markov process and the proof follows the same logic as for the Brownian snake. However we do not require this here. We also note that we can represent the distribution of the complete snake process $(X_t, t \geq 0)$ under $\mathbb{N}_x(\cdot|T_h<\infty)$ as $\mathbb{N}_x(\cdot | T_h<\infty) = \int_f \mathbb{P}_{x|f}(\cdot)n(df|T_h<\infty)$.
\end{remark}
We now recall the construction of $h$-erased contour processes from Le Gall \cite{le1989marches} with vertices that form an alternating exponential random walk, see also Neveu and Pitman \cite{neveu1989branching}. We use these processes to construct BBM and ooBBM. For an excursion $f=(f_t, t \geq 0) \in \mathcal{E}$ with $T_h(f)<\infty$ we define the following. Let $h>0$, let $\alpha^{(h)}_0=0$ and for at most finitely many $p' \geq 0$, following the construction set out in \cite[Section 4]{le1989marches}, we can define
\begin{equation}\label{equation beta}
    \gamma^{(h)}_{2p'}=\beta^{(h)}_{p'} = \inf \Big\{ t > \alpha^{(h)}_{p'} : f_t \geq \inf\{f_{t'}: \alpha^{(h)}_{p'} \leq t' \leq t \} + h\Big\}, \hspace{0.25cm} Z^{(h)}_{2p'} = f_{\beta^{(h)}_{p'}} - h,
\end{equation}
\begin{equation}\label{equation alpha}
    \gamma^{(h)}_{2p'+1} = \alpha^{(h)}_{p'+1} = \inf \Big\{ t > \beta^{(h)}_{p'} : f_t \leq \sup \{f_{t'} : \beta^{(h)}_{p'} \leq t' \leq t\} - h \Big\}, \hspace{0.25cm} Z^{(h)}_{2p'+1} = f_{\alpha^{(h)}_{p'+1}}.
\end{equation}
These stopping times define an alternating walk $(Z^{(h)}_p)_{0 \leq p \leq M^{(h)}}$ where $M^{(h)}$ is the largest $p \geq 0$ for which $\gamma^{(h)}_p<\infty$. Let $\Tilde{M}^{(h)} \coloneqq \lfloor M^{(h)}/2 \rfloor$. Let $Z^{(h)}_{M^{(h)}+1} \coloneqq 0$ and $\gamma^{(h)}_{M^{(h)}+1} \coloneqq \sigma(f)$. We define $f^{(h)}:[0,\infty) \rightarrow [0,\infty)$ to be the \emph{$h$-erased contour function} with local extrema given by the points $(\gamma^{(h)}_p)_{0 \leq p \leq M^{(h)}+1}$, i.e.
\begin{equation*}
        f^{(h)}(t) \coloneqq \left\{
            \begin{matrix}
                & 0, &  0 \leq t <  \gamma^{(h)}_0 \hspace{0.25cm} \text{ or } \hspace{0.25cm} t \geq \gamma_{M^{(h)}+1}^{(h)} \\
                & Z^{(h)}_p + \frac{(t-\gamma^{(h)}_p)(Z^{(h)}_{p+1} - Z^{(h)}_p)}{(\gamma^{(h)}_{p+1} - \gamma^{(h)}_p)}, & \gamma^{(h)}_p \leq t < \gamma^{(h)}_{p+1}, 0 \leq p \leq M^{(h)}.
            \end{matrix}
        \right.  
\end{equation*}
We recall the construction of a tree from a contour function, see e.g. Le Gall \cite[Chapter III]{le1999spatial}.
\begin{definition}[Tree from contour function]\label{definition tree from a contour function}
    We define a shifted excursion to be a continuous function $g:[0,\infty) \rightarrow [0,\infty)$ such that there exists $r \geq 0$ such that $g(t)=0$ for $t \in [0,r]$ and the function $\Tilde{g}:[0,\infty) \rightarrow [0,\infty)$ is an excursion where we have $\Tilde{g}(t) \coloneqq g(r+t)$ for $t \geq 0$. We define $\sigma(g) \coloneqq r+\sigma(\Tilde{g})$.
    \par
    We define a tree from a shifted excursion $g:[0,\infty) \rightarrow [0,\infty)$ as the following:
    \begin{enumerate}
        \item Each point $t \in [0,\infty)$ corresponds to a vertex of the tree at height $g(t)$;
        \item for $t$, $t' \in [0,\infty)$, $t<t'$, the vertex $t$ is an ancestor of the vertex $t'$ if and only if $g(t) = \inf(\{g(r): r \in [t,t']\})$. More generally, the quantity $\inf(\{g(r):r \in [t,t']\})$ is the height of the most recent common ancestor to $t$ and $t'$ and we call any $r$ for which this infimum is attained the most recent common ancestor of $t$ and $t'$;
        \item the distance between the vertices $t$ and $t'$ is defined to be $\Tilde{d}(t, t') = g(t) + g(t') - 2 \inf(\{g(r): r \in [t,t']\})$ and we identify $t$ and $t'$ (we write $t \sim t'$) if $\Tilde{d}(t,t')=0$.
    \end{enumerate}
    The tree coded by $g$ is the quotient set $[0,\infty)/\sim$ with the distance $\Tilde{d}$ and the genealogical relation defined in (2). We call the equivalence class of level $0$ the root and we note that the root is an ancestor of all the vertices.
    \par
    Finally we note that $[0,\infty)/\sim$ is equivalent to $[0,\sigma(g)]/\sim$. We use this second formulation throughout as it highlights the shifted excursion $g$ from which the tree is constructed.
\end{definition}
We note that $f^{(h)}$ defined above is a shifted excursion and we consider the tree $[0,\sigma(f^{(h)})]/\sim$ constructed from it as in Definition \ref{definition tree from a contour function}. Such trees are known as the \emph{$h$-erased trees of $f$}, see Neveu \cite{neveu1986erasing}. The tree $[0,\sigma(f^{(h)})]/\sim$ consists of leaves at distance $(Z^{(h)}_{2p'+1})_{0 \leq p' \leq \Tilde{M}^{(h)}}$ from the root and branch points at distance $(Z^{(h)}_{2p'})_{1 \leq p' \leq \Tilde{M}^{(h)}}$ from the root. We can split the rooted tree into its edges $(e^{(h)}_p)_{1 \leq p \leq M^{(h)}}$ which have as endpoints the vertices $(\gamma^{(h)}_p)_{0 \leq p \leq M^{(h)}}$ of the tree where for each $p$ the upper end of the edge $e_p^{(h)}$ has height $Z^{(h)}_p$. Furthermore, the edge $e^{(h)}_p$ is isometric to the interval $[Y^{(h)}_p,Z^{(h)}_p]$ where $Y_p^{(h)}$ is the height of the most recent common ancestor of $\gamma_p^{(h)}$ among the $\gamma_{2p'}^{(h)}$. See Figure \ref{figure snake contour function} for an illustration. 
\begin{remark}
The tree in Definition \ref{definition tree from a contour function} could also be encoded as a genealogical tree with edge lengths and Ulam-Harris indexing of vertices, see e.g. \cite{le1989marches}, but we code it as above to keep indexing notation consistent with the bivariate snake. 
\end{remark}
We define a distribution for an alternating random walk that we refer to below. Fix $h>0$ as above and let $(E^{(h)}_j)_{j \in \mathbb{N}}$ be a family of independent Exponential random variables where $E^{(h)}_j \sim \text{Exp}(1/h)$. Let $C^{(h)}_k \coloneqq \sum_{j=1}^k (-1)^{j+1} E^{(h)}_j$ and let $K^{(h)} \coloneqq \inf \{ n \geq 1: C^{(h)}_k < 0 \}$. Finally, let $(\Tilde{C}^{(h)}_k)_{1 \leq k \leq K^{(h)}}$ be defined by $\Tilde{C}^{(h)}_k \coloneqq C^{(h)}_k$ for $1 \leq k < K^{(h)}$, $\Tilde{C}^{(h)}_k \coloneqq 0$ for $k = K^{(h)}$. We denote the distribution of $(\Tilde{C}^{(h)}_k)_{1 \leq k \leq K^{(h)}}$ by $Q^{(h)}$.
\par
Let us now assume that $f$ is an excursion under the probability distribution $n(\cdot|T_1<\infty)$. Under this assumption Le Gall shows in the proof of \cite[Theorem 6]{le1989marches} that the process $(Z^{(1)}_p)_{1 \leq p \leq M^{(1)}+1}$ has distribution $Q^{(1)}$ as defined above. Le Gall also proves in \cite[Theorem 3]{le1989marches} that the rooted tree given by $[0,\sigma(f^{(1)})]/\sim$ is the genealogical tree of a critical binary branching process with iid $\text{Exp}(2)$ lifetimes.
\begin{remark}\label{remark branching general h case}
    Everything stated above generalises to any $h>0$ in the sense that $[0,\sigma(f^{(h)})]/\sim$ is the genealogical tree of a critical binary branching process with $\text{Exp}(2/h)$ lifetimes under $n(\cdot|T_h<\infty)$. 
\end{remark}
\begin{figure}
    \centering
    \includegraphics[width=0.8\textwidth]{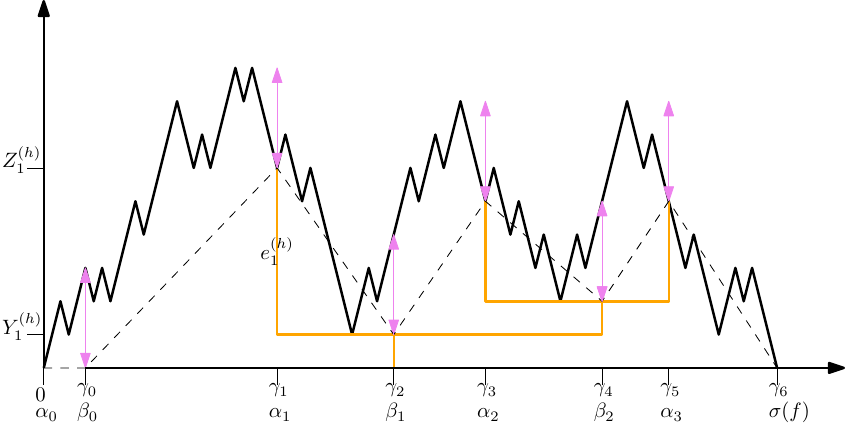}
    \caption{In black bold we have the excursion $f$ and in black dashed we have the contour function $f^{(h)}$. In orange we have the tree $[0,\sigma(f^{(h)})]/\sim$. We have labelled one edge $e^{(h)}_1$ of $[0,\sigma(f^{(h)})]/\sim$ and the corresponding values $Y^{(h)}_1$ and $Z^{(h)}_1$. The purple arrowed lines indicate the differences of $h$ that determine the times $(\gamma^{(h)}_p)_{0 \leq p \leq M^{(h)}}$. In this case $M^{(h)}=5$, the number of vertical edges of the tree.}
    \label{figure snake contour function}
\end{figure}
Let $h>0$ and now assume that $f \in \mathcal{E}$ is a random excursion with probability distribution $n(\cdot|T_h < \infty)$. We make some observations on the structure of the tree. The root of the tree is at $\gamma_0^{(h)}$,the first edge ends at $\gamma^{(h)}_{q^*}$ for the $1 \leq q^* \leq M^{(h)}$ for which $Z^{(h)}_{q^*}$ is minimal, $q^* \coloneqq \text{argmin}(\{Z^{(h)}_p: 1 \leq p \leq M^{(h)}\})$. With probability $1/2$ we have $\gamma^{(h)}_{q^*} = \alpha^{(h)}_1$ in which case the tree consists of a single edge or $\gamma^{(h)}_{q^*} = \beta^{(h)}_r$ for some $r$. In the second case we have the that the ancestor has two subtrees. It is shown in \cite[Lemma 4]{le1989marches} that $Z^{(h)}_{q^*} \sim \text{Exp}(2/h)$ and that, conditional on $\gamma^{(h)}_{q^*}$ having descendants, we have that the processes
\begin{equation*}
    (Z^{(h)}_p - Z^{(h)}_{q^*})_{1 \leq p \leq q^*}; \hspace{0.5cm} ((Z^{(h)}_{q^*+p} - Z^{(h)}_{q^*})^+)_{1 \leq p \leq M^{(h)} + 1-q^*}
\end{equation*}
are independent both with law $Q^{(h)}$. In the second case here we take the positive part denoted by $(\cdot)^+$ for notational convenience so that for the final point $(Z^{(h)}_{M^{(h)}+1} - Z^{(h)}_{q^*})^+=0$ as we require. The process then repeats recursively for both the subtrees.
\par
We also note from the proof of \cite[Theorem 6]{le1989marches} for all $1 \leq p \leq M^{(h)}$ that $Z_p^{(h)} \leq f_{\gamma_p^{(h)}}$. This is because for $Z_{2p'+1}^{(h)} = f_{\gamma_{2p'+1}^{(h)}}$ and $Z_{2p'}^{(h)} = f_{\gamma_{2p'}^{(h)}} - h$ for $0 \leq p' \leq \Tilde{M}^{(h)}$.
\par
Furthermore, for $0 \leq p' \leq \Tilde{M}^{(h)} - 1$, Le Gall states that $Z^{(h)}_{2p'} = \inf(\{f_t: \alpha_{p'}^{(h)} \leq t \leq \beta_{p'}^{(h)}\})$ and we can further note that $Z^{(h)}_{2p'} \leq \inf(\{ f_t:\beta_{p'}^{(h)} \leq t \leq \alpha_{p'+1}^{(h)}\})$ because
\begin{equation*}
    f_{\beta^{(h)}_{p'}} \leq \sup \{f_{t'} : \beta^{(h)}_{p'} \leq t' \leq t\} \hspace{0.25cm} \text{   for all   } t>\beta^{(h)}_{p'} \hspace{0.25cm} \text{ and so } \hspace{0.25cm} \alpha^{(h)}_{p'+1} \leq \inf\{t' > \beta^{(h)}_{p'} : f_{t'} = f_{\beta^{(h)}_{p'}} - h\}
\end{equation*}
by definition of the points $(\alpha^{(h)}_{p'}, 1 \leq p' \leq \Tilde{M}^{(h)} + 1)$. Therefore
\begin{equation}\label{equation branching points are minima}
    Z^{(h)}_{2p'} = \inf(\{f_t: \alpha_{p'}^{(h)} \leq t \leq \alpha_{p'+1}^{(h)}\}) \hspace{0.5cm} \text{ for all   } 0 \leq p' \leq \Tilde{M}^{(h)} - 1
\end{equation}
which also verifies for us that the tree does indeed lie beneath the excursion, as illustrated in Figure \ref{figure snake contour function}.
\par
During the lifetime of a particle the splitting and dying rates do not affect the spatial motion. We can use Lemma \ref{lemma single particle motion with age process} in the next proposition to verify that each particle individually follows the correct spatial motion. In this proposition we have a snake $X$ and we condition on its contour process to be $f$. We consider, for $h>0$, the process $f^{(h)}$ and the tree $[0,\sigma(f^{(h)})]/\sim$ and we use the edges of this tree represented as the intervals $[Y^{(h)}_p, Z^{(h)}_p)$ along with the time changes given by the function $H^{(s)}$ to define a particle system. We use the $\mathbb{R}^d$ component of the snake to give spatial locations to these particles in \eqref{equation ooBBM construction}. All this together gives ooBBM.
\vspace{2cm}
\begin{proposition}\label{proposition single ooBBM with branching}
    Let $x=(b,0) \in E=\mathbb{R}^d \times [0,\infty)$ and let $\varepsilon>0$. Let $h \coloneqq \varepsilon /2$ and let $X=(W,V)$ be a snake with distribution $\mathbb{N}_x(\cdot | T_h < \infty)$. Then the process $(\Tilde{\mathcal{Z}}^{(\varepsilon, \rm d)}_s, \Tilde{\mathcal{Z}}^{(\varepsilon, \rm a)}_s)_{s \geq 0}$ defined by
    \begin{equation}\label{equation ooBBM construction}
        (\Tilde{\mathcal{Z}}^{(\varepsilon, \rm d)}_s, \Tilde{\mathcal{Z}}^{(\varepsilon, \rm a)}_s) \coloneqq \Big(\sum_{\substack{p: A^X_{\gamma^{(h)}_p}(s)=0,\\ H^X_{\gamma^{(h)}_p}(s) \in [Y_p^{(h)}, Z_p^{(h)}),}} \delta_{W_{\gamma^{(h)}_p}(H^X_{\gamma^{(h)}_p}(s))}, \sum_{\substack{p: A^X_{\gamma^{(h)}_p}(s) = 1 \\ H^X_{\gamma^{(h)}_p}(s) \in [Y_p^{(h)}, Z_p^{(h)}),}} \delta_{W_{\gamma^{(h)}_p}(H^X_{\gamma^{(h)}_p}(s))}\Big)
    \end{equation}
    is an $\text{ooBBM}(4/\varepsilon, c, \Tilde{c})$ and with initial state $(0,\delta_b)$, i.e. starting from a single active particle at spatial location $b$.
    \begin{proof}
        We drop the superscripts $X$ for this proof as we only deal with a single snake $X$ throughout. We consider the $h$-erased tree of its contour process. On the event where this tree consists of a single edge $e_1^{(h)}$ the argument is simple.
        \par
        As noted in Remark \ref{remark branching general h case}, this edge has an $\text{Exp}(h/2)$-distributed length. As $\mathbb{N}_x(\cdot) = \int_f \mathbb{P}_{x | f}(\cdot) n(df)$ and hence $\mathbb{N}_x(\cdot | T_h < \infty) = \int_f \mathbb{P}_{x | f}(\cdot) n(df | T_h < \infty)$, we obtain from the third part of Theorem \ref{theorem bivariate brownian snake} that $X_{\gamma_1^{(h)}}$ is distributed as $(\xi(y), 0 \leq y < E^{(2/h)})$, where $\xi \sim P_{(b,0)}^\xi$ and $E^{(2/h)} \sim \text{Exp}(2/h)$ are independent. Hence $(\Tilde{\mathcal{Z}}^{(\varepsilon, \rm d)}_s, \Tilde{\mathcal{Z}}^{(\varepsilon, \rm a)}_s)_{s \geq 0}$ consists of a single particle that starts at position $b$ in the active state and we have the spatial motion and dormancy of the particle by Lemma \ref{lemma single particle motion with age process} have the required dynamics of a single particle in an ooBBM.
        \par
        Note that the points $\gamma_p^{(h)}$ for $1 \leq p \leq M^{(h)}$ are non-random under $\mathbb{P}_{x|f}$. Furthermore the left limits $X_{\gamma_p^{(h)}}(Z_p^{(h)}-)$ exist for $1 \leq p \leq M^{(h)}$ (this is immediately apparent when $p$ is even as $Z_p^{(h)}< f_{\gamma_p^{(h)}}$ and also follows when $p$ is odd where $Z_p^{(h)}= f_{\gamma_p^{(h)}}$). Therefore we can condition on the contour process of $X$ and use the snake property at the points $(\gamma_p^{(h)})_{1 \leq p \leq M^{(h)}}$ up to heights $(Z_p^{(h)})_{1 \leq p \leq M^{(h)}}$. Condition on $X$ to have contour process $f$. For the given genealogical tree $[0,\sigma(f^{(h)})]/\sim$ we show the particle system has the correct spatial motion and dormancy dynamics.
        \par
        We show the joint distribution for each pair of particles. We do this below and we use the snake property of $X$ and the most recent common ancestor of the pair in the genealogical tree. The argument generalises for the joint distribution of all particles by repeating this argument by looking at all the most recent common ancestors for all subsets of particles and repeating the steps below. The spatial motion and dormancy dynamics for each individual particle follows from Lemma \ref{lemma single particle motion with age process}.
        \par
        Let $k$, $j \in \mathbb{N}$ such that $1 \leq k < j \leq M^{(h)}$. Consider the particles denoted by the vertices $\gamma^{(h)}_k$ and $\gamma^{(h)}_j$ in $[0,\sigma(f^{(h)})]/\sim$. The spatial motion and dormancy dynamics of the particle indexed by $\gamma^{(h)}_k$ is given by 
        \begin{equation*}
            (W_{\gamma_k^{(h)}}(H_{\gamma_k^{(h)}}(s)), A_{\gamma_k^{(h)}}(s))_{s'_k \leq s < s''_k}
        \end{equation*}
        where 
        \begin{equation*}
            s'_k \coloneqq \sup\Big(\Big\{s:H_{\gamma_k^{(h)}}(s) = Y^{(h)}_k\Big\}\Big); \hspace{0.5cm}
            s''_k \coloneqq \inf\Big(\Big\{s: 
            H_{\gamma_k^{(h)}}(s) = Z^{(h)}_k\Big\}\Big).
        \end{equation*}
        More generally for $1 \leq p \leq M^{(h)}$ define $s'_p$ as we had defined $s'_k$. We note that the particle $\gamma^{(h)}_k$ has the correct dynamics over its lifetime by Lemma \ref{lemma single particle motion with age process}. We get the same result for the dynamics of the particle at $\gamma^{(h)}_j$. It remains to verify that the initial states and the joint law of the spatial motions of the particles follow the correct distribution. 
        \par
        We first note that $\inf(\{f_t:t \in [\gamma^{(h)}_k, \gamma^{(h)}_j]\}) = Z^{(h)}_{2l}$ for some $k \leq 2l \leq j$, $\gamma^{(h)}_k \leq \gamma^{(h)}_{2l} \leq \gamma^{(h)}_j$ (by \eqref{equation branching points are minima}, note it is possible that $\gamma^{(h)}_{2l}=\gamma^{(h)}_k$ or $\gamma^{(h)}_{2l}=\gamma^{(h)}_j$). We consider the snake at $\gamma^{(h)}_k$ and $\gamma^{(h)}_j$. By the snake property we have that 
        \begin{equation*}
            (X_{\gamma_k^{(h)}}(y))_{0 \leq y < Z^{(h)}_{2l}} = (X_{\gamma_j^{(h)}}(y))_{0 \leq y < Z^{(h)}_{2l}}
        \end{equation*}
        and therefore
        \begin{equation*}
            (W_{\gamma_k^{(h)}}(H_{\gamma_k^{(h)}}(s)), A_{\gamma_k^{(h)}}(s))_{0 \leq s < s''_{2l}} = (W_{\gamma_k^{(h)}}(H_{\gamma_j^{(h)}}(s)), A_{\gamma_j^{(h)}}(s))_{0 \leq s < s''_{2l}}.
        \end{equation*}
        Furthermore we have $(X_{\gamma_k^{(h)}}(y))_{Z^{(h)}_{2l} < y < Z^{(h)}_k}$ is conditionally independent of $(X_{\gamma_j^{(h)}}(y))_{Z^{(h)}_{2l} < y < Z^{(h)}_j}$ given $X_{\gamma^{(h)}_{2l}}$ and that therefore $(W_{\gamma_k^{(h)}}(H_{\gamma_k^{(h)}}(s)), A_{\gamma_k^{(h)}}(s))_{s''_{2l} < s \leq s''_k}$ is conditionally independent of 
        \newline
        $(W_{\gamma_j^{(h)}}(H_{\gamma_j^{(h)}}(s)), A_{\gamma_j^{(h)}}(s))_{s''_{2l} < s \leq s''_j}$.
        \par
        Finally we note that the contour of $X$ has distribution $n(\cdot|T_h<\infty)$ and is such that 
        \newline
        $(Z^{(h)}_p, 1 \leq p \leq M^{(h)}+1)$ has distribution $Q^{(h)}$ and under this distribution we have particles of the tree split and die during their active states with rate $1/h = 2/\varepsilon$ each which are the correct rates for $\text{ooBBM}(4/\varepsilon, c, \Tilde{c})$.
    \end{proof}
\end{proposition}
\begin{figure}
    \centering
    \includegraphics[width=0.8\textwidth]{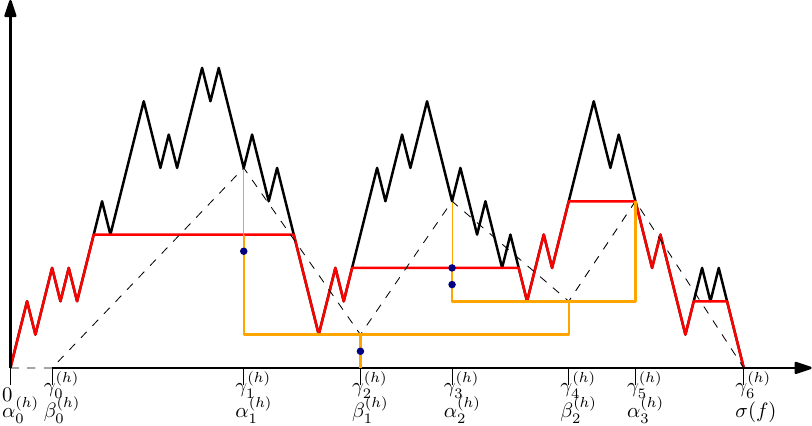}
    \caption{In black bold we have the excursion $f$ for the same excursion as in Figure \ref{figure snake contour function}. In black dashed we have the contour process $f^{(h)}$. In red we have the age process $(H_t(s))_{0 \leq t \leq \sigma(f)}$ and in orange we have the tree $[0,\sigma(f^{(h)})]/\sim$ which is bold up to where it is stopped. The dark blue dots represent jumps for paths of the subordinator component of the snake that determines the age process.}
    \label{figure snake age process function}
\end{figure}
We now consider the Poisson random measure that we call the on/off Brownian snake $\sum_{i \in I} \delta_{((b_i,s_i), f^i, x^i)}$ which we defined in Definition \ref{definition oo snake}. For $i \in I$, and for $s$, $t \geq 0$, let $H^i_t(s) \coloneqq H^{x^i}_t(s)$ and $A^i_t(s) \coloneqq A^{x^i}_t(s)$. Also let $Y^{(h)}_{p,i}$, $Z^{(h)}_{p,i}$, $\gamma^{(h)}_{p,i}$ be analogous to $Y^{(h)}_p$, $Z^{(h)}_p$, $\gamma^{(h)}_p$ for the excursion $f^i$. Let $x^i \coloneqq (w^i, v^i)$ where $w^i$ denotes the Brownian spatial component of the snake and $v^i$ denotes the subordinator component of the snake. Note also for $i \in I$, $s<s_i$ and $0 \leq t \leq \sigma(f^i)$ we have $A^i_t(s)=0$.
We can now give our main result of the section.
\begin{proposition}\label{proposition identification oobmm h erased tree}
    Recall the on/off Brownian snake from Definition \ref{definition oo snake} with index set $I$. Let $\varepsilon > 0$ and let $h \coloneqq \varepsilon / 2$. We define the following measure-valued processes for $i \in I$ and for $s \geq 0$:
    \begin{equation}\label{equation oobbm tree}
        (\mathcal{Z}_{s,i}^{(\varepsilon, \rm d)}, \mathcal{Z}_{s,i}^{(\varepsilon, \rm a)}) \coloneqq \Big(\sum_{\substack{p: A^i_{\gamma^{(h)}_{p, i}}(s)=0,\\ H^i_{\gamma^{(h)}_{p, i}}(s) \in [Y_{p, i}^{(h)}, Z_{p, i}^{(h)})}} \delta_{w^i_{\gamma^{(h)}_{p, i}}(H^i_{\gamma^{(h)}_{p, i}}(s))}, \sum_{\substack{p: A^i_{\gamma^{(h)}_{p, i}}(s) = 1, \\ H^i_{\gamma^{(h)}_{p, i}}(s) \in [Y_{p, i}^{(h)}, Z_{p, i}^{(h)})}} \delta_{w^i_{\gamma^{(h)}_{p, i}}(H^i_{\gamma^{(h)}_{p, i}}(s))}\Big).
    \end{equation}
    Then we have that $((\mathcal{Z}_s^{(\varepsilon, \rm d)}, \mathcal{Z}_s^{(\varepsilon, \rm a)}))_{s \geq 0}$ defined as $(\mathcal{Z}_s^{(\varepsilon, \rm d)}, \mathcal{Z}_s^{(\varepsilon, \rm a)}) \coloneqq (\sum_{i \in I} \mathcal{Z}_{s,i}^{(\varepsilon, \rm d)}, \sum_{i \in I} \mathcal{Z}_{s,i}^{(\varepsilon, \rm a)})$ for $s \geq 0$ is an $\text{ooBBM}(4/\varepsilon, c, \Tilde{c})$ with initial condition given by dormant particles and active particles distributed as Poisson random measures with intensities $\mu_{\rm d}/\varepsilon$, $\mu_{\rm a}/\varepsilon$.
    \par
    Furthermore the process $(\mathcal{Z}_y^{(\varepsilon)})_{y \geq 0}$, where for $y \geq 0$ we define
    \begin{equation*}
        \mathcal{Z}_y^{(\varepsilon)} \coloneqq \sum_{i \in I} \sum_{p: y \in [Y_{p, i}^{(h)}, Z_{p, i}^{(h)})} \delta_{w^i_{\gamma^{(h)}_{p, i}}(y)}, 
    \end{equation*}
    is a $\text{BBM}(4/\varepsilon)$ starting from a Poisson random measure with intensity $(\mu_{\rm d}+\mu_{\rm a})/\varepsilon$.
    \begin{proof}
        We first show that $\mathcal{Z}_0^{(\varepsilon)}$ and $(\mathcal{Z}_0^{(\varepsilon, \rm d)}, \mathcal{Z}_0^{(\varepsilon, \rm a)})$ have the correct distributions. Let $\mu' \coloneqq \mu_{\rm d} + \mu_{\rm a}$. The random measure $\mathcal{Z}_0^{(\varepsilon)}$ is then a Poisson random measure with intensity
        \begin{equation*}
            n(T_h<\infty) \mu'(\cdot) = \mu'(\cdot)/ 2h = \mu'(\cdot)/\varepsilon
        \end{equation*}
        and $\mathcal{Z}_0^{(\varepsilon, \rm d)}$, $\mathcal{Z}_0^{(\varepsilon, \rm a)}
        $ are particles under a Poisson random measures with intensities $\mu_{\rm d}/2h = \mu_{\rm d}/\varepsilon$,
        \newline
        $\mu_{\rm a}/2h = \mu_{\rm a}/\varepsilon$ according to Definition \ref{definition oo snake}. For the ooBBM case we have by Proposition \ref{proposition single ooBBM with branching} for each initial active particle the subsequent behaviour then has the correct distribution and the particles with a dormant initial state switch to an active particle at the correct rate since the $s_i$ are sampled from $\text{Exp}(\Tilde{c})$. Finally, note that particles evolve independently under the Poisson random measure. The case for BBM is simpler.
    \end{proof}
\end{proposition}
Note that $\mathcal{Z}^{(\varepsilon)}$ and $(\mathcal{Z}^{(\varepsilon, \rm d)}, \mathcal{Z}^{(\varepsilon, \rm a)})$ in Proposition \ref{proposition identification oobmm h erased tree} are coupled via the on/off Brownian snake. We now show the presence of another BBM which will be useful to us later.
\begin{corollary}\label{corollary bbm extra killing}
    In the setting of Proposition \ref{proposition identification oobmm h erased tree}, the process $(\Tilde{\mathcal{Z}}_y^{(\varepsilon)})_{y \geq 0}$, defined to be, for $y \geq 0$,
    \begin{equation*}
        \Tilde{\mathcal{Z}}_y^{(\varepsilon)} \coloneqq \sum_{i \in I} \sum_{\substack{p: A^i_{\gamma^{(h)}_{p, i}}(s)=1 \hspace{0.1cm} \forall 0 \leq s \leq y, \\ y \in [Y_{p, i}^{(h)}, Z_{p, i}^{(h)})}} \delta_{w^i_{\gamma^{(h)}_{p, i}}(y)}, 
    \end{equation*}
    is a $\text{BBM}(2/\varepsilon, 2/\varepsilon + c)$ starting from a Poisson random measure with intensity $\mu_{\rm a}/\varepsilon$.
    \begin{proof}
        This follows in the same way as for the $\text{BBM}(4/\varepsilon)$ (which is the same as $\text{BBM}(2/\varepsilon, 2/\varepsilon)$) given in Proposition \ref{proposition identification oobmm h erased tree}. The difference here is that once an active particle moves to the dormant state it is no longer included in the sum, and this is what increases the death rate by $c$.
    \end{proof}
\end{corollary}
\section{Local time measure for ooSBM}\label{section snake local time measure}
The aim of this section is to produce a representation for $\text{ooSBM}(4, c, \Tilde{c})$ with the on/off Brownian snake 
\begin{equation*}
    \sum_{i \in I} \delta_{((b_i,s_i), f^i, x^i)}
\end{equation*}
defined in Definition \ref{definition oo snake}. For $i \in I$, we write $x^i=(w^i,v^i)$ for the Brownian component $w^i$ and the subordinator component $v^i$ of the snake $x^i$. We also write $\hat{w}^i$ for $w^i(\zeta(x^i)-)$ (recall Remark \ref{remark head of the snake}) and we also write $\hat{v}^i$ for $v^i(\zeta(x^i)-)$ when this exists. Also we set $y_i \coloneqq (b_i,s_i)$. We define this Poisson random measure on a probability space $(\Omega, \mathcal{F}, \mathbb{P}_{(\mu_{\rm d}, \mu_{\rm a})})$. We begin with the corresponding representation for $\text{SBM}(4)$, see e.g. \cite[Chapter IV]{le1999spatial}. In the following, for a measure $\mu$ and a function $\phi$, we write $\langle \mu, \phi \rangle$ for $\int \phi d\mu$. For the excursions $f^i$, $i \in I$, $t>0$ we define the local times $L^i_t(\cdot)$ as
\begin{equation}\label{equation normalised local time}
    L^i_y(t) \coloneqq \lim_{\varepsilon \rightarrow 0} \frac{1}{\varepsilon} \int_{r=0}^t \mathbbm{1}_{(y, y + \varepsilon)}(f^i_r) dr.
\end{equation}
Following the convention in Le Gall \cite{le1999spatial}, we write $dL^i_y(t)$ to denote the Lebesgue-Stieltjes measure as defined from the continuous increasing function $L^i_y$. Formally, $dL^i_y(t)$ is the Borel measure $\phi(dt)$ for which we have for any $0\leq r \leq r'$ that $\phi((r,r'))=L^i_y(r)-L^i_y(r')$. We also do this for other local occupation time measures in this section.
\begin{theorem}[The Brownian snake for super-Brownian motion]\label{theorem sbm snake}
    There exists a measure-valued process $(\mathcal{X}_y)_{y \geq 0}$ such that for any nonnegative bounded Borel measure function $g$ on $\mathbb{R}^d$ and any $y>0$ we have
    \begin{equation}\label{equation sbm integral rep}
        \langle \mathcal{X}_y, g \rangle = \sum_{i \in I} \int_{r=0}^{\sigma_i} dL^i_y(r) g(\hat{w}^i_r) \hspace{0.5cm} \mathbb{P}_{(\mu_{\rm d}, \mu_{\rm a})}\text{-a.s.}
    \end{equation}
    The process $(\mathcal{X}_y)_{y \geq 0}$ is an $\text{SBM}(4)$ started from $\mu \coloneqq \mu_{\rm d} + \mu_{\rm a}$.
    \par
    Furthermore we have that for any $y > 0$ that
    \begin{equation}\label{equation sbm supp snake rep}
        \text{supp}(\mathcal{X}_y) = \overline{\bigcup_{i \in I}\{\hat{w}^i_r: r \in [0,\sigma(f^i)], f^i_r = y\}} \hspace{0.5cm} \mathbb{P}_{(\mu_{\rm d}, \mu_{\rm a})}\text{-a.s.}
    \end{equation}
    This follows from \cite[Equation IV.10]{le1999spatial} which states the support for each component of the superprocess $\mathcal{X}_y$ constructed from each of the snakes $x^i$, $i \in I$. The closure of the union of the supports for these components is stated in \eqref{equation sbm supp snake rep}.
\end{theorem}
We now introduce the necessary theory for our construction of ooSBM from the on/off Brownian snake. We use the notion of downcrossings for local times, see e.g. M\"orters and Peres \cite{morters2010brownian}. For $R$ a reflecting Brownian motion and for $0 \leq a < b$ we can define $\tau_0 \coloneqq 0$, and for $n \geq 1$
\begin{equation*}
    \sigma_n \coloneqq \inf(\{t > \tau_{n-1}: R(t) = b \}), \hspace{0.5cm} \tau_n \coloneqq \inf(\{t>\sigma_n : R(t) = a \}),
\end{equation*}
and $R^{(n)}:[0,\tau_n - \sigma_n] \rightarrow [0,\infty)$ given by $R^{(n)}(t') \coloneqq R(\sigma_n+t')$ is called the $n^{\text{th}}$ downcrossing of $[a,b]$. We also refer to the points $(\tau_n, n \in \mathbb{N})$ as the downcrossings of $[a,b]$. We also refer to these as $(b-a)$-downcrossings to level $a$. We also define $\Tilde{R}^{(n)}:[0, \sigma_n - \tau_{n-1}] \rightarrow [0,\infty)$ by $\Tilde{R}^{(n)}(t') \coloneqq R(\tau_{n-1} + t')$ as the $n^{\text{th}}$ upcrossing of $[a,b]$. In this section we are interested in the downcrossings of $[0,h]$, for $h>0$, of the residual processes $(\text{Res}^X_t(s))_{t \geq 0}$ for $s\geq 0$, for snakes $X$. 
\par
We work with the Poisson random measure that we call the on/off Brownian snake in Definition \ref{definition oo snake} with index set $I$.
\par
For $i \in I$ we define $\text{Res}^i_t(s) \coloneqq \text{Res}^{x^i}_t(s) = f^i_t - H^i_t(s)$. For $h>0$, consider the finitely many stopping times $(S^{h, \rm a}_{s,n,i})_{1 \leq n \leq M^{h, \rm a}_i}$, $(S^{h, \rm d}_{s,n,i})_{1 \leq n \leq M^{h, \rm d}_i}$, $(T^{h, \rm a}_{s,n,i})_{0 \leq n \leq M^{h, \rm a}_i}$ and $(T^{h, \rm d}_{s,n,i})_{0 \leq n \leq M^{h, \rm d}_i}$ for the snake given by $T^{h, \rm a}_{s,0,i}=T^{h, \rm d}_{s,0,i}=0$ and for finitely many $n$ we have
\begin{equation*}
    S^{h, \rm a}_{s,n,i} \coloneqq \inf\{t > T^{h, \rm a}_{s, n-1, i} | \text{Res}^i_t(s) > h, A^i_t(s) = 1 \}; \hspace{0.25cm} T^{h, \rm a}_{s,n,i} \coloneqq \inf\{t > S^{h, \rm a}_{s,n,i} | \text{Res}^i_t(s) = 0\},
\end{equation*}
\begin{equation*}
    S^{h, \rm d}_{s,n,i} \coloneqq \inf\{t > T^{h, \rm d}_{s, n-1, i} | \text{Res}^i_t(s) > h, A^i_t(s) = 0 \}; \hspace{0.25cm} T^{h, \rm d}_{s,n,i} \coloneqq \inf\{t > S^{h, \rm d}_{s,n,i} | \text{Res}^i_t(s) = 0\},
\end{equation*}
where $M^{h, \rm d}_i \in \mathbb{N}$ (respectively $M^{h, \rm a}_i$) is the largest number $n$ such that $S^{h, \rm d}_{s,n,i}$ (respectively $S^{h, \rm a}_{s,n,i}$) is finite. If such an $n$ does not exist, we set $M_i^{h,d}$ (respectively $M_i^{h,a}$) to zero. Here $f^i$ plays the role of $R$ in Definition \ref{definition age process}. We also define the finitely many stopping times $(S^h_{s,n,i})_{1 \leq n \leq M^h_i}$ and $(T^h_{s,n,i})_{0 \leq n \leq M^h_i}$ for the snake given by $T^h_{s,0,i}=0$ and for finitely many $n$ we have
\begin{equation*}
    S^h_{s,n,i} \coloneqq \inf\{t > T^h_{s, n-1, i} | \text{Res}^i_t(s) > h\}; \hspace{0.25cm} T^h_{s,n,i} \coloneqq \inf\{t > S^h_{s,n,i} | \text{Res}^i_t(s) = 0\},
\end{equation*}
where $M^h_i \in \mathbb{N}$ and we note that $M^h_i = M^{h, \rm d}_i + M^{h, \rm a}_i$. We have the following identification of the of the ooBBM defined in the previous section. We show these stopping times in Figure \ref{figure snake downcrossings function}. 
\begin{lemma}\label{lemma alternative characterisation ooBBM from snake}
    Recall the Poisson random measure from Definition \ref{definition oo snake} and its index set $I$. Let $i \in I$, let $\varepsilon>0$ and let $s>0$. Let $h = \varepsilon/2$ and recall $(\mathcal{Z}^{(\varepsilon, \rm d)}_{s,i}, \mathcal{Z}^{(\varepsilon, \rm a)}_{s,i})$ from Proposition \ref{proposition identification oobmm h erased tree}. Then we have
    \begin{equation}\label{equation oobbm downcrossing}
        (\mathcal{Z}^{(\varepsilon, \rm d)}_{s,i}, \mathcal{Z}^{(\varepsilon, \rm a)}_{s,i}) = \Big(\sum_{n = 1}^{M^{h, \rm d}_i} \delta_{w^i_{S^{h, \rm d}_{s,n,i}}(H^i_{S^{h, \rm d}_{s,n,i}}(s))}, \sum_{n = 1}^{M^{h, \rm a}_i} \delta_{ w^i_{S^{h, \rm a}_{s,n,i}}(H^i_{S^{h, \rm a}_{s,n,i}}(s))}\Big).
    \end{equation}
    \begin{proof}
        We work on the $\mathbb{P}_{(\mu_{\rm d}, \mu_{\rm a})}$-almost sure event that, for all $i \in I$ and for all $p \in \mathbb{N}$ we have $H^i_{\gamma^{(h)}_{p, i}}(s) \neq Z^{(h)}_{p, i}$ which is almost sure as the $Z^{(h)}_{p, i}$ levels are continuously distributed. We work on this event as our proof relies on counting arguments for the atoms of the sums in \eqref{equation oobbm tree} and in \eqref{equation oobbm downcrossing} and the argument fails at the exceptional times where $H^i_{\gamma^{(h)}_{p, i}}(s) = Z^{(h)}_{p, i}$. As our claim is only for fixed $s$ this does not affect our results. We describe the argument for the dormant particles $\mathcal{Z}^{(\varepsilon, \rm d)}_{s,i}$ and leave the argument for the active particles as it is identical. The presence of dormant and active states adds a level of complexity to the arguments below, but the reader can assume in the first instance that all particles in the tree are dormant and follow the argument. The inclusion of active particles simply means that some points in the sequence $\gamma^{(h)}_{p,i}$, $Z^{(h)}_{p, i}$ are omitted from the arguments for the dormant case and vice versa.
        \par
        First consider the formulation of $\mathcal{Z}^{(\varepsilon, \rm d)}_{s,i}$ given in \eqref{equation oobbm tree} and consider an atom in this sum. For each atom $\delta_{w^i_{\gamma^{(h)}_{p, i}}(H^i_{\gamma^{(h)}_{p, i}}(s))}$ we have $H^i_{\gamma^{(h)}_{p,i}}(s) \in [Y_{p, i}^{(h)}, Z_{p, i}^{(h)})$ and $A^i_{\gamma^{(h)}_{p, i}}(s)=0$. We consider two cases for the index $p$ for $\gamma^{(h)}_{p, i}$: the first is that $p=2p'+1$ for some $p' \in \mathbb{N}$ and therefore $\gamma^{(h)}_{p,i} = \alpha^{(h)}_{p',i}$ and the second is that $p=2p'$ for some $p' \in \mathbb{N}$ and therefore $\gamma^{(h)}_{p,i} = \beta^{(h)}_{p',i}$. In the first case where $\gamma^{(h)}_{p,i} = \alpha^{(h)}_{p',i}$ we claim that there is a $n$ such that $T^{h, \rm d}_{s,n,i} \in [\alpha^{(h)}_{p',i}, \beta^{(h)}_{p',i})$. This can be seen from how the points $\alpha^{(h)}_{q,i}$, $\beta^{(h)}_{q,i}$, $q \in \mathbb{N}$, are defined in \eqref{equation beta} and \eqref{equation alpha}. Specifically, over the interval $[\alpha^{(h)}_{p',i}, \beta^{(h)}_{p',i})$ the value $f^i_t$ decreases to $H^i_{\gamma^{(h)}_{p, i}}(s)$ and the value of $H^i_t(s)$ is constant until $f^i_t$ hits $H^i_{\gamma^{(h)}_{p, i}}(s)$. Furthermore, on the interval $[\beta^{(h)}_{p'-1,i}, \alpha^{(h)}_{p',i})$ the supremum is $f^i_{\alpha^{(h)}_{p',i}}+h$ and there are no $h$-downcrossings of $f^i_t$ to any level. Therefore the first time $t$ over the interval $[\alpha^{(h)}_{p',i}, \beta^{(h)}_{p',i})$ that $f^i_t$ hits $H^i_{\gamma^{(h)}_{p, i}}(s)$ is a $h$-downcrossing time for the process $\text{Res}^i_t(s)$, given as $T^{h, \rm d}_{s,n,i}$ for some $n \in \mathbb{N}$, and it is the only $h$-downcrossing time for $\text{Res}^i_t(s)$ in the interval $[\alpha^{(h)}_{p',i}, \beta^{(h)}_{p',i})$. It is therefore clear that the $T^{h, \rm d}_{s,n,i}$ considered in this case are unique as they are in disjoint intervals $[\alpha^{(h)}_{p',i}, \beta^{(h)}_{p',i})$. Furthermore, we have by the snake property at points $S^{h, \rm d}_{s,n,i} - T_h(f^i)$, $\gamma^{(h)}_{p, i} - T_h(f^i)$, $T^{h, \rm d}_{s,n,i} - T_h(f^i)$ for the strong Markov process $(x^i_{T_h(f^i)+ t}, t \geq 0)$ (see Proposition \ref{proposition bessel and brownian snakes}) the equality between $w^i_{\gamma^{(h)}_{p, i}}(H^i_{\gamma^{(h)}_{p, i}}(s))$ and $w^i_{S^{h, \rm d}_{s,n,i}}(H^i_{S^{h, \rm d}_{s,n,i}}(s))$ and also that $A^i_{\gamma^{(h)}_{p, i}}(H^i_{\gamma^{(h)}_{p, i}}(s)) = A^i_{S^{h, \rm d}_{s,n,i}}(H^i_{S^{h, \rm d}_{s,n,i}}(s)) = 0$. Therefore the atoms $\delta_{w^i_{\gamma^{(h)}_{p, i}}(H^i_{\gamma^{(h)}_{p, i}}(s))}$ and $\delta_{w^i_{S^{h, \rm d}_{s,n,i}}(H^i_{S^{h, \rm d}_{s,n,i}}(s))}$ are present in the sums in \eqref{equation oobbm tree} and \eqref{equation oobbm downcrossing} and they are equal as required.
        \par
        Now we consider the second case where $\gamma^{(h)}_{p,i} = \beta^{(h)}_{p',i}$. In this case let $p'' = \inf\{q > p': Z^{(h)}_{2q, i} \leq Z^{(h)}_{2p', i}\}$. In this case we have that $f^i_t$ hits $H^i_{\gamma^{(h)}_{p, i}}(s)$ for $t$ in the interval $[\alpha^{(h)}_{p'',i}, \beta^{(h)}_{p'',i})$ and we let $T$ be the first such hitting time. We have that $f^i_t > H^i_{\gamma^{(h)}_{p, i}}(s)$ on the interval $[\gamma^{(h)}_{p, i}, T)$ and so we have equality of $H^i_{\gamma^{(h)}_{p, i}}(s)$ and $H^i_T(s)$. We can now see that, like in the previous case, $T$ is a $h$-downcrossing time of $\text{Res}^i_t(s)$ and therefore $T=T^{h, \rm d}_{s,n,i}$ for some $n \in \mathbb{N}$. We also have, as in the previous case, equality between $w^i_{\gamma^{(h)}_{p, i}}(H^i_{\gamma^{(h)}_{p, i}}(s))$ and $w^i_{S^{h, \rm d}_{s,n,i}}(H^i_{S^{h, \rm d}_{s,n,i}}(s))$, $A^i_{\gamma^{(h)}_{p, i}}(H^i_{\gamma^{(h)}_{p, i}}(s)) = A^i_{S^{h, \rm d}_{s,n,i}}(H^i_{S^{h, \rm d}_{s,n,i}}(s)) = 0$, and that the atoms $\delta_{w^i_{\gamma^{(h)}_{p, i}}(H^i_{\gamma^{(h)}_{p, i}}(s))}$ and $\delta_{w^i_{S^{h, \rm d}_{s,n,i}}(H^i_{S^{h, \rm d}_{s,n,i}}(s))}$ are present in the sums in \eqref{equation oobbm tree} and \eqref{equation oobbm downcrossing}. The details of these arguments are similar to the previous case and rely on the definitions of $\alpha^{(h)}_{q,i}$, $\beta^{(h)}_{q,i}$, $q \in \mathbb{N}$ in \eqref{equation beta} and \eqref{equation alpha}. We also note here that $H^i_{\alpha^{(h)}_{p'',i}}(s) \notin [Y_{2p''-1, i}^{(h)}, Z_{2p''-1, i}^{(h)})$ as $H^i_{\alpha^{(h)}_{p'',i}}(s) \leq Y_{2p''-1, i}^{(h)}$ and so $T^{h, \rm d}_{s,n,i}$ is not double-counted from any atoms in the first case. This is because by definition of $p''$ we either have $Z_{2p''-2, i}^{(h)} \geq Z^{(h)}_{2p, i}$ or $p''=p'+1$ and either way $H^i_{\alpha^{(h)}_{p'',i}}(s) \geq Y_{2p''-1, i}^{(h)}$ leads to a contradiction. To see this, suppose that there is $q'<p'$ with $H^i_{\gamma^{(h)}_{2q',i}}(s) \in [Y_{2q', i}^{(h)}, Z_{2q', i}^{(h)})$ and $A^i_{\gamma^{(h)}_{2q', i}}(s)=0$ as usual and also with $\inf\{q > q': Z^{(h)}_{2q, i} \leq Z^{(h)}_{2q', i}\} = \inf\{q > p': Z^{(h)}_{2q, i} \leq Z^{(h)}_{2p', i}\} = p''$. Then we must have that $Z^{(h)}_{2q',i} < Z^{(h)}_{2p', i}$. In fact we would have that $(Y^{(h)}_{2q', i}, Z^{(h)}_{2q', i}]$ and $(Y^{(h)}_{2p', i}, Z^{(h)}_{2p', i}]$ are disjoint intervals but that $H^i_t(s)$ would be constant over $[\beta^{(h)}_{q', i}, \beta^{(h)}_{p', i})$ and this would be a contradiction as it would mean that $H^i_{\gamma^{(h)}_{2q',i}}(s) = H^i_{\gamma^{(h)}_{2p',i}}(s)$ is in both the disjoint sets $(Y^{(h)}_{2q', i}, Z^{(h)}_{2q', i}]$ and $(Y^{(h)}_{2p', i}, Z^{(h)}_{2p', i}]$. Therefore there cannot exist such a $q'$ and we do not have any double-counting in the second case.
        \par
        To conclude the proof we are required to show that for each $n$ we have that $T^{h, \rm d}_{s,n,i}$ is counted by either the first or the second case given above. Fix $n$ and consider $T^{h, \rm d}_{s,n,i}$. There exists a $p'$ such that $T^{h, \rm d}_{s,n,i} \in [\alpha^{(h)}_{p',i}, \beta^{(h)}_{p',i})$ and we note here that $\alpha^{(h)}_{p',i} = \gamma^{(h)}_{2p'-1,i}$ and $\beta^{(h)}_{p',i} = \gamma^{(h)}_{2p',i}$. We first claim that in this setting we have $H^i_{T^{h, \rm d}_{s,n,i}}(s) \in [Z^{(h)}_{2p', i}, Z^{(h)}_{2p'-1, i}]$. This can be see again from the definitions of $\alpha^{(h)}_{q,i}$, $\beta^{(h)}_{q,i}$, $q \in \mathbb{N}$ that for $y>0$ and $p \in \mathbb{N}$ there is a $h$-downcrossing to $f^i_t$ to level $y$ in the interval $[\alpha^{(h)}_{p,i}, \beta^{(h)}_{p,i})$ if and only if $y \in [Z^{(h)}_{2p', i}, Z^{(h)}_{2p'-1, i}]$. If $H^i_{T^{h, \rm d}_{s,n,i}}(s) \in [Y^{(h)}_{2p'-1, i}, Z^{(h)}_{2p'-1, i}] \subseteq [Z^{(h)}_{2p' ,i}, Z^{(h)}_{2p'-1 , i}]$ then we are in the setting of the first case as an atom corresponding to $\alpha^{(h)}_{p',i} = \gamma^{(h)}_{2p'-1,i}$ will be included in the sum in \eqref{equation oobbm tree}. Otherwise we have for $p'' = \sup\{q<p':Y^{(h)}_{q, i} < H^i_{T^{h, \rm d}_{s,n,i}}(s)\}$ that $H^i_{T^{h, \rm d}_{s,n,i}}(s) \in [Y^{(h)}_{2p'', i}, Z^{(h)}_{2p'', i}]$ and that an atom corresponding $\beta^{(h)}_{p'',i} = \gamma^{(h)}_{2p'', i}$ is included in the sum in \eqref{equation oobbm tree} and we are in the second case described above. 
    \end{proof}
\end{lemma}
In Proposition \ref{proposition snake reflecting brownian motion} we give results about the residual process. This is an instance of a more general result from \cite{BertoinLeGallLeJan1997} that we state in Lemma \ref{lemma snake exit open set general}.
\begin{lemma}[Residual lifetime process, Section 4 of \cite{BertoinLeGallLeJan1997}]\label{lemma snake exit open set general}
    Let $E$ be a complete and separable metric space, let $x \in E$ and let $w \in \mathcal{W}_x$. Consider a Brownian snake $X$ under $\mathbb{P}_w$ (respectively $\mathbb{N}_x$) as defined in Theorem \ref{theorem bivariate brownian snake}. Let $D$ be an open subset of $E$, let $x \in D$ and let $F = E \backslash D$. For $v \in \mathcal{W}_x$, set $\tau(v) \coloneqq \inf(\{t \geq 0, v(t) \notin D \})$. Assume that $P^\xi_x(\tau<\infty)>0$. Then define
    \begin{equation*}
        U_t \coloneqq (\zeta(X_t) - \tau(X_t))^+ \hspace{0.5cm} \eta_r \coloneqq \inf\Big(\Big\{t, \int_0^t \mathbbm{1}_{\{\tau(X_u)<\zeta(X_u)\}} du > r \Big\}\Big),
    \end{equation*}
    and then the process $(U_{\eta_r}, r \geq 0)$ is a reflecting Brownian motion $\mathbb{P}_w$-a.s., (respectively is a stopped reflecting Brownian motion $\mathbb{N}_x$-a.e.). Furthermore, it follows that there exists a local time at level $0$ for $(U_t)_{t \geq 0}$ which is the limit
    \begin{equation*}
        \ell(t) \coloneqq \lim_{\varepsilon' \rightarrow 0}\frac{1}{\varepsilon'} \int_{u=0}^t \mathbbm{1}_{\{\tau(X_u) < \zeta(X_u) < \tau(X_u) + \varepsilon'\}} du
    \end{equation*}
    which exists for every $t \geq 0$ $\mathbb{P}_w$-a.s. (respectively $\mathbb{N}_x$-a.e.) and defines a continuous increasing process. As before, let $d\ell(t)$ denote the Lebesgue-Steiltjes measure associated with $\ell(t)$. Under $\mathbb{N}_x$, for a measurable function $\Phi$ we have
    \begin{equation}\label{equation local time general}
        \mathbb{N}_x\Big(\int_{t=0}^\sigma \Phi(X_t)d\ell(t)\Big) = E_x^\xi(\Phi(\xi) \mathbbm{1}_{\tau(\xi) < \infty})
    \end{equation}
    and in particular we have
    \begin{equation}\label{equation local time contribution}
        \mathbb{N}_x(\ell(\infty)) = P^\xi_x(\tau(\xi)<\infty).
    \end{equation}
    \begin{proof}
        This result is given in \cite{BertoinLeGallLeJan1997} at the start of Section 4 prior to the introduction of exit measures. In particular, \eqref{equation local time general} is given in \cite[Equation (35)]{BertoinLeGallLeJan1997} and \eqref{equation local time contribution} is a particular instance of this result.
    \end{proof}
\end{lemma}
We note that, in Lemma \ref{lemma snake exit open set general}, $\ell(\cdot)$ can be constructed from the local time for $(U_{\eta_r}, r \geq 0)$ with time-change given by the inverse of $\eta$. We refer to $\ell(\cdot)$ as the \emph{exit local time} from the set $D$. Exit local times are used to construct \emph{exit measures} which are the integral of the heads of the snakes against the exit local times. In this section, we define exit local times that we use to construct exit measures which have the finite-dimensional distributions of ooSBM.
\begin{figure}
    \centering
    \includegraphics[width=0.8\textwidth]{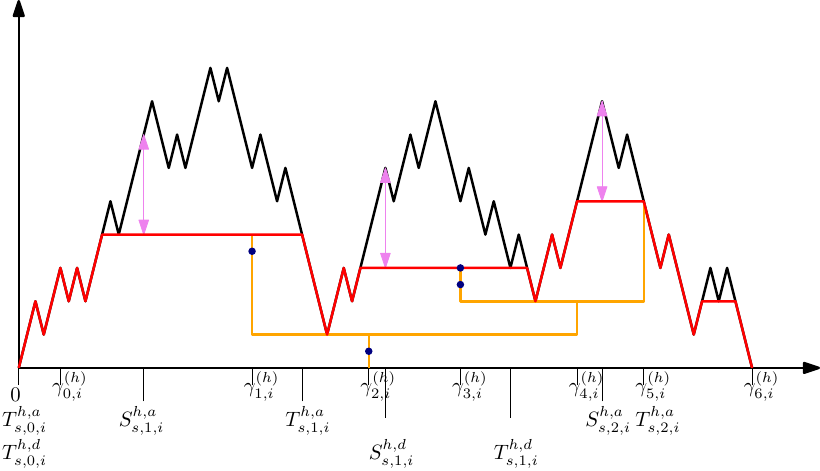}
    \caption{For $i \in I$ we show the downcrossing stopping times $(S^{h, \rm d}_{s,n,i}, T^{h, \rm d}_{s,n,i})_{1 \leq n \leq M^{h, \rm d}_i}$, $(S^{h, \rm a}_{s,n,i}, T^{h, \rm a}_{s,n,i})_{1 \leq n \leq M^{h, \rm a}_i}$. In black bold we have $f^i$ which is taken to be the same excursion $f$ as in Figures \ref{figure snake contour function} and \ref{figure snake age process function} for convenience. The tree given in orange is the tree for the $h$-erased contour process of $f^i$ and this tree is stopped at the age process $(H^i_t(s))_{0 \leq t \leq \sigma(f)}$ given in red. The dark blue dots represent jumps for paths of the subordinator components $v^i$ of the snake. The purple arrowed lines here indicate the differences of $h$ that determine the downcrossing stopping times.}
    \label{figure snake downcrossings function}
\end{figure}
\begin{proposition}\label{proposition snake reflecting brownian motion}
Fix $s > 0$ and let $x=(b,s') \in E$ with $s'<s$. We define the time changes
\begin{equation*}
    \eta(s,u,i) \coloneqq \inf \Big\{ t' \geq 0: \int_{r=0}^{t'} \mathbbm{1}_{\{\text{Res}^i_r(s) > 0\}} dr > u\Big\};
\end{equation*}
\begin{equation*}
    \eta(s,u,i, \rm a) \coloneqq \it \inf \Big\{ t' \geq 0: \int_{r=0}^{t'} \mathbbm{1}_{\{\text{Res}^i_r(s) > 0, A^i_r(s) = 1\}} dr > u\Big\};
\end{equation*}    
\begin{equation*}
    \eta(s,u,i, \rm d) \coloneqq \it \inf \Big\{ t' \geq 0: \int_{r=0}^{t'} \mathbbm{1}_{\{\text{Res}^i_r(s) > 0, A^i_r(s) = 0\}} dr > u\Big\}.
\end{equation*}
Then the maps $u \mapsto \text{Res}^i_{\eta(s,u,i)}(s)$, $u \mapsto \text{Res}^i_{\eta(s,u,i, \rm a)}(s)$ and $u \mapsto \text{Res}^i_{\eta(s,u,i, \rm d)}(s)$ are stopped reflecting Brownian motions $\mathbb{N}_x$-a.e. Furthermore we have that the local times of $(\text{Res}^i_{\eta(s,u,i)}(s))_{u \geq 0}$, $(\text{Res}^i_{\eta(s,u,i, \rm a)}(s))_{u \geq 0}$ and $(\text{Res}^i_{\eta(s,u,i, \rm d)}(s))_{u \geq 0}$ at level $0$, denoted by $\hat{\ell}_{s,i}$, $\hat{\ell}^{\rm a}_{s,i}$, $\hat{\ell}^{\rm d}_{s,i}$, are such that 
    \begin{equation*}
        \mathbb{N}_x\Big(\hat{\ell}_{s,i}(\infty)\Big) = \mathbb{N}_x\Big(\hat{\ell}^{\rm a}_{s,i}(\infty)\Big) +  \mathbb{N}_x\Big(\hat{\ell}^{\rm d}_{s,i}(\infty)\Big) = 1.
    \end{equation*}
    \begin{proof}
        Consider the open set $D=\mathbb{R}^d \times [0,s)$ in the setting of Lemma \ref{lemma snake exit open set general}. In this setting the process $(U_{\eta_r}, r \geq 0)$ is exactly the process given by $(\text{Res}^i_{\eta(s,u,i)}(s))_{u \geq 0}$ as
        \begin{equation*}
            \tau^s(X_t) = \inf \{r \geq 0 | (B_t(r), S_t(r)) \notin \mathbb{R}^d \times [0,s)\} = \inf \{r \geq 0 |  S_t(r) \geq s \} = H_t(s),
        \end{equation*}
        and because the Residual process is the difference between the excursion $\zeta^i_t$ and the age process $H^i_t(s)$.
        \par
        Now instead consider the set $D=\mathbb{R}^d \times [0,\infty) \backslash \{s\}$ in the setting of the previous lemma. In this instance the process $(U_{\eta_r}, r \geq 0)$ is exactly the process given by $(\text{Res}^i_{\eta(s,u,i, \rm a)}(s))_{u \geq 0}$  
        \begin{equation*}
            \tau^s(X_t) = \inf \{r \geq 0 | (B_t(r), S_t(r)) \notin \mathbb{R}^d \times [0,\infty) \backslash \{s\} \} = \inf \{r \geq 0 | S_t(r) = s\}, 
        \end{equation*}
        i.e. if the subordinator $S_t(\cdot)$ can take the value $s$, then the particle is in the active state.
        \par
        The final case follows as 
        \begin{equation*}
            \text{Res}^i_t(s) = \text{Res}^i_t(s)\mathbbm{1}_{A_t(s) = 0} + \text{Res}^i_t(s)\mathbbm{1}_{A_t(s) = 1}
        \end{equation*}
        and so the process $(\text{Res}^i_{\eta(s,u,i, \rm d)}(s))_{u \geq 0}$ is a reflecting Brownian motion as it is defined by first taking the process $(\text{Res}^i_{\eta(s,u,i)}(s))_{u \geq 0}$ and removing excursions that correspond to $(\text{Res}^i_{\eta(s,u,i, \rm a)}(s))_{u \geq 0}$. For the final part we use \eqref{equation local time contribution} and note that
        \begin{equation*}
            \mathbb{N}_x\Big(\hat{\ell}_{s,i}(\infty)\Big) = P^\xi_x\Big(\inf(\{y: S(y) \geq s\})<\infty \Big) = 1 >  P^\xi_x\Big(\inf(\{y: S(y) = s\})<\infty\Big) = \mathbb{N}_x\Big(\hat{\ell}^{\rm a}_{s,i}(\infty)\Big),
        \end{equation*}
        where the inequality can be seen in e.g. \cite[Theorem III.5]{bertoin1996levy} and where $S$ denotes the subordinator component of $\xi$. Furthermore we have that 
        \begin{equation*}
            \mathbb{N}_x\Big(\hat{\ell}^{\rm d}_{s,i}(\infty)\Big) = 1 - P^\xi_x(\inf(\{y: S(y) = s\})<\infty) = P^\xi_x(\inf(\{y: S(y) = s\})=\infty) \in (0,1),
        \end{equation*}
        where we note that the final probability is in $(0,1)$ as the subordinator can either cross or hit level $s$, see e.g.\ \cite[Theorem III.5]{bertoin1996levy}.
    \end{proof}
\end{proposition}
We later apply this result above to the snakes in the Poisson random measure given in Definition \ref{definition oo snake}. Before that though we briefly make a remark on particles that start dormant and do not wake up before time $s$.
\begin{remark}\label{remark snake dormant until after time t}
    Fix $s \geq 0$. In the setting of the Definition \ref{definition oo snake}, for $i \in I$, we have that for $s_i > s$ that $H^i_t(s)=0$, $A^i_t(s)=0$ and $\text{Res}^i_t(s) = f^i_t$ for all $t \in [0,\sigma(f^i)]$. In this instance the stopping times $S^{h, \rm d}_{s,n,i}$, $T^{h, \rm d}_{s,n,i}$ are then simply the up- and down-crossing times of $[0,h]$ for $f^i$. Since $f^i$ is an excursion, this means we have one such downcrossing if $T_h(f^i)<\infty$ and a zero downcrossings otherwise.
    \par
    The case where $s_i=s$ does not affect our analysis as for fixed $s$ we have $s \neq s_i$ for all $i \in I$, $\mathbb{P}_{(\mu_{\rm d}, \mu_{\rm a})}$-a.s.
\end{remark}
We will define the representation of ooSBM in Theorem \ref{theorem ooSBM ffd convergence} but first we state a technical lemma that we will require in the proof of the theorem. 
\begin{lemma}\label{lemma poisson h downcrossings}
    Consider a (possibly countably infinite) collection of independent reflecting Brownian motions $(R^j, j \in J)$ stopped at the local times at zero $(C_j, j \in J)$ on some common probability space $(\Omega, \mathcal{F}, P)$. Let $C=\sum_{j \in J}C_j$ and suppose $C$ is finite. For $j \in J$ let $D^j_h$ be the number of downcrossings from level $h$ to level $0$ for $R^j$. Then we have that
    \begin{equation*}
        \lim_{h \downarrow 0} h \sum_{j \in J} D^j_h = \frac{C}{2} \hspace{0.5cm} \text{a.s.}
    \end{equation*}
    \begin{proof}
        We first note that for $h>0$, $(D^j_h, j \in J)$ form a family of independent Poisson random variables with $D^j_h \sim \text{Poisson}(C_j/2h)$. We recall that the sum of independent Poisson random variables is again a Poisson random variable and so $\sum_{j \in J} D^j_h \sim \text{Poisson}(C/2h)$. The result then follows by the strong law of large numbers for Poisson processes.
    \end{proof}
\end{lemma}
We now give our representation of ooSBM as exit measures of the on/off Brownian snake. 
\begin{theorem}\label{theorem ooSBM ffd convergence}
    Fix $s>0$. For $t \geq 0$, $s_i>s$ let $\text{Res}^{i, \rm d}_t(s) \coloneqq \text{Res}^i_t(s)\mathbbm{1}_{A^i_t(s)=0}$ and $\text{Res}^{i, \rm a}_t(s) \coloneqq \text{Res}^i_t(s)\mathbbm{1}_{A^i_t(s)=1}$. Firstly, there exist local time processes $\ell^{\rm a}_{s,i}(\cdot)$ and $\ell^{\rm d}_{s,i}(\cdot)$ at level $0$ for $(\text{Res}^{i, \rm d}_t(s))_{t \geq 0}$ and $(\text{Res}^{i, \rm a}_t(s))_{t \geq 0}$.
    \par
    Secondly, we have
    \begin{equation*}
        (\varepsilon \mathcal{Z}^{(\varepsilon, \rm d)}_s, \varepsilon\mathcal{Z}^{(\varepsilon, \rm a)}_s) \xrightarrow[\varepsilon \rightarrow 0]{a.s.} (\mathcal{Z}^{(\rm d)}_s, \mathcal{Z}^{(\rm a)}_s),
    \end{equation*}
    on the space of $\mathcal{M}_F(\mathbb{R}^d)^2$. For $s=0$ set $\mathcal{Z}^{(\rm d)}_s$, $\mathcal{Z}^{(\rm a)}_s$ to be $\mu_{\rm d}$, $\mu_{\rm a}$. Then the limits $((\mathcal{Z}^{(\rm d)}_s, \mathcal{Z}^{(\rm a)}_s))_{s \geq 0}$ have the finite-dimensional distributions of $\text{ooSBM}(4, c, \Tilde{c})$ started at $(\mu_{\rm d}, \mu_{\rm a})$. 
    \par
    Finally, for fixed $s>0$, and for any continuous bounded function $\Phi:\mathbb{R}^d \rightarrow [0,\infty)$ we have that
    \begin{equation}\label{equation oosbm snake dormant rep}
        \langle \mathcal{Z}^{(\rm d)}_s, \Phi \rangle \coloneqq \int_x \Phi(x) \mathcal{Z}^{(\rm d)}_s(dx) = e^{-\Tilde{c}s}\int_x \Phi(x) \mu_{\rm d}(dx) + \sum_{i \in I: s_i<s} \int_{t=0}^{\sigma(f^i)} \Phi(\hat{w}^i_t)d\ell^{\rm d}_{s, i}(t)
    \end{equation}
    and
    \begin{equation}\label{equation oosbm snake active rep}
        \langle \mathcal{Z}^{(\rm a)}_s, \Phi \rangle \coloneqq \int_x \Phi(x) \mathcal{Z}^{(\rm a)}_s(dx) = \sum_{i \in I: s_i<s} \int_{t=0}^{\sigma(f^i)} \Phi(\hat{w}^i_t)d\ell^{\rm a}_{s, i}(t).
    \end{equation}
    \begin{proof}
        For the first part of the theorem suppose $i \in I$ and $s_i<s$. The processes $(\text{Res}^{i, \rm d}_t(s))_{t \geq 0}$, $(\text{Res}^{i, \rm a}_t(s))_{t \geq 0}$ are simply the processes $(\text{Res}^i_{\eta(s,t',i, \rm a)})_{t' \geq 0}$, $(\text{Res}^i_{\eta(s,t',i, \rm d)})_{t' \geq 0}$ from Proposition \ref{proposition snake reflecting brownian motion} with a time change and so by Proposition \ref{proposition snake reflecting brownian motion} this implies the existence of the local time processes $\ell^{\rm d}_{s,i}(\cdot)$ and $\ell^{\rm a}_{s,i}(\cdot)$.
        \par
        Now we show the second part of the theorem and first we consider $i \in I$ with $s_i>s$ i.e. the particles that started dormant and have remained in the dormant state until time $s$. As before let $\varepsilon > 0$ and let $h = \varepsilon/2$. We have by Remark \ref{remark snake dormant until after time t} that $M_i^{h,d}$ is equal to one iff $\sup\{f^i_t, 0 \leq t \leq \sigma(f^i)\}$ is greater than $h$, otherwise $M_i^{h,d}$ is equal to zero. Furthermore we have that $M_i^{h,a} = 0$, because if $s_i > s$, then $A_\cdot^i(s) = 0$. Therefore we have
        \begin{equation*}
            \Big(\varepsilon \sum_{i \in I:s_i>s} \sum_{n=1}^{M^{h, \rm d}_i} \delta_{S^{h, \rm d}_{s,n,i}},  \varepsilon \sum_{i \in I:s_i> s} \sum_{n=1}^{M^{h, \rm a}_i} \delta_{S^{h, \rm a}_{s,n,i}}\Big) = \Big(\varepsilon \sum_{i \in I:s_i>s, T_h(f^i)<\infty} \delta_{T_h(f^i)}, 0\Big),
        \end{equation*}
        which means that 
        \begin{equation*}
            \Big(\varepsilon \sum_{i \in I:s_i>s} \mathcal{Z}^{(\varepsilon, \rm d)}_{s,i}, \varepsilon \sum_{i \in I:s_i>s} \mathcal{Z}^{(\varepsilon, \rm a)}_{s,i}\Big) = \Big(\varepsilon \sum_{i \in I:s_i>s, T_h(f^i)<\infty} \delta_{b_i} , 0\Big)
        \end{equation*}
        and this is a Poisson random measure with intensity $(e^{-\Tilde{c}s}\mu_{\rm d}(db)/\varepsilon, 0)$ rescaled by $\varepsilon$ and this converges to the measure $(e^{-\Tilde{c}s}\mu_{\rm d}(db),0)$ as $\varepsilon$ tends to zero $\mathbb{P}_{(\mu_{\rm d},\mu_{\rm a})}$-a.s. by the strong law of large numbers for Poisson processes, see Dudley \cite[Theorem 11.4.1]{dudley2002real}.
        \par
        Now we return to $i \in I$ with $s_i<s$ (note as discussed in Remark \ref{remark snake dormant until after time t} there are no $i \in I$ with $s_i=s$ $\mathbb{P}_{(\mu_{\rm d},\mu_{\rm a})}$-a.s.). The points $(S^{h, \rm a}_{s,n,i})_{1 \leq n \leq M^{h, \rm a}_i}$ are the left-endpoints of the $h$-downcrossings to level $0$ for the process $(\text{Res}^i_{\eta(s, \rm a,t,i)})_{t \geq 0}$ and so we have that 
        \begin{equation}\label{equation snake local time limits}
            \Big(\varepsilon \sum_{n = 1}^{M^{h, \rm d}_i} \delta_{S^{h, \rm d}_{s,n,i}}(\cdot), \varepsilon \sum_{n = 1}^{M^{h, \rm a}_i} \delta_{S^{h, \rm a}_{s,n,i}}(\cdot)\Big) \xrightarrow[a.s.]{} (d\ell^{\rm d}_{s,i}(\cdot), d\ell^{\rm a}_{s,i}(\cdot))
        \end{equation}
        on the space $\mathcal{M}_F([0,\infty)^2)$, see \cite[Chapter 6]{morters2010brownian}. Strictly that reference applies to reflecting Brownian motion but the convergence is not affected by the time changes we have here. Let $\ell_{s,i} \coloneqq \ell^{\rm d}_{s,i} + \ell^{\rm a}_{s,i}$ and for $x \in E$ let $\ell_s$ denote the corresponding measure under $\mathbb{N}_x$. Then we have by the compensation formula for Poisson random measures that 
        \begin{equation*}
            \mathbb{E}_{(\mu_{\rm d}, \mu_{\rm a})} \Big[\sum_{i \in I: s_i < s} \ell_{s, i}(\infty)\Big] = \int_{b \in \mathbb{R}^d} \mathbb{N}_{(b,0)}(\ell_s(\infty)) \mu_{\rm a}(db) + \int_{b \in \mathbb{R}^d} \int_{q=0}^s \mathbb{N}_{(b,q)}(\ell_s(\infty)) \Tilde{c}e^{-\Tilde{c}q} dq \mu_{\rm d}(db).
        \end{equation*}
        We note here that for $b \in \mathbb{R}^d$, $0<q<s$ that $\mathbb{N}_{(b,q)}(\ell_s(\infty) \in \cdot) = \mathbb{N}_{(b,0)}(\ell_{s-q}(\infty) \in \cdot)$. By the second part of Proposition \ref{proposition snake reflecting brownian motion} we have $\mathbb{N}_{(b,0)}(\ell_s(\infty))$ and $\mathbb{N}_{(b,q)}(\ell_s(\infty))$ are equal to one and so the expectation in the equation above is equal to $\norm{\mu_{\rm a}}+(1- e^{-\Tilde{c}s})\norm{\mu_{\rm d}}$ which is finite and therefore $\sum_{i \in I, s_i<s}\ell_{s,i}(\infty)<\infty$ $\mathbb{P}_{(\mu_{\rm d},\mu_{\rm a})}$-a.s.
        \par
        Now let $\varepsilon'>0$ and let $M \coloneqq \sup(\{\Phi(b), b \in \mathbb{R}^d\})<\infty$. Let $N \in \mathbb{N}$ and $i_1, \dots, i_N \in I$ such that $\sum_{i \in I: i \neq i_k \forall k \in [N]} \ell^{\rm a}_{s,i}(\infty)<2\varepsilon'/M$. By Lemma \ref{lemma poisson h downcrossings} the downcrossings of the remaining contributions (conditional of the local times at level $0$) are such that
        \begin{equation}\label{equation snake local time remaining contributions}
            \lim_{\varepsilon \rightarrow 0} \varepsilon \sum_{i \in I: i \neq i_k \forall k \in [N], s_i<s} (M^{h, \rm d}_i+M^{h, \rm a}_i) < \varepsilon'/M. 
        \end{equation}
        We now consider
        \begin{multline*}
            \liminf_{\varepsilon \rightarrow 0} \sum_{i \in I} \int_{t=0}^{\sigma(f^i)} \Phi(w^i_t(H^i_t(s))) 2h \sum_{n=1}^{M^{h, \rm a}_i} \delta_{S^{h, \rm a}_{s,n,i}}(dt) \\
            \geq \sum_{k=1}^N \liminf_{\varepsilon \rightarrow 0} \int_{t=0}^{\sigma(f^{i_k})} \Phi(w^{i_k}_t(H^{i_k}_t(s))) 2h \sum_{n=1}^{M^{h, \rm a}_{i_k}} \delta_{S^{h, \rm a}_{s,n,i_k}}(dt) \geq \int_{b \in \mathbb{R}^d} \Phi(b) \mathcal{Z}^{(\rm a)}_s(db) - \varepsilon'
        \end{multline*}
        where the last inequality is by \eqref{equation snake local time remaining contributions} where we have the total mass of all the contributions we have not included has size at most $\varepsilon'/M$. The convergence is due to the fact in the limit as $\varepsilon$ tends to $0$ we are evaluating at the head of the Brownian component of the snake, $\hat{w}^i = w^i(\zeta(x^i)-)$, $\ell$-a.e. and we recall that the head of the snake is continuous as noted in Remark \ref{remark head of the snake}. Similarly
        \begin{multline*}
            \limsup_{\varepsilon \rightarrow 0} \sum_{i \in I} \int_{t=0}^{\sigma(f^i)} \Phi(w^i_t(H^i_t(s))) 2h \sum_{n=1}^{M^{h, \rm a}_i} \delta_{S^{h, \rm a}_{s,n,i}}(dt) \\
            \leq \sum_{k=1}^N \limsup_{\varepsilon \rightarrow 0}  \int_{t=0}^{\sigma(f^{i_k})} \Phi(w^{i_k}_t(H^{i_k}_t(s))) 2h \sum_{n=1}^{M^{h, \rm a}_{i_k}} \delta_{S^{h, \rm a}_{s,n,i_k}}(dt) + \varepsilon' \leq \int_{b \in \mathbb{R}^d} \Phi(b) \mathcal{Z}^{(\rm a)}_s(db) + \varepsilon'.
        \end{multline*}
        As $\varepsilon'$ is arbitrary we have the convergence for the active component at time $s$. The argument for the dormant component is the same except there is the contribution from the $i \in I$ with $s_i>s$ to include but this is shown to converge as $\varepsilon \rightarrow 0$ in the first part of this proof. 
        \par
        To show the convergence in finite-dimensional distributions recall from 
        Proposition \ref{proposition existence of ooSBM} that ooSBM is the limit of of ooBBM. Consider any $(s_1, \dots, s_l)$. Then we have the a.s convergence holds jointly for $s_1$, $\dots$, $s_l$ and by Proposition \ref{proposition identification oobmm h erased tree} these finite-dimensional distributions converge to the finite-dimensional distributions of ooSBM. 
        \par
        The final part of the theorem follows from the previous parts, namely \eqref{equation snake local time limits}.
    \end{proof}
\end{theorem}
The representations given in \eqref{equation oosbm snake dormant rep} and \eqref{equation oosbm snake active rep} show us that $\mathcal{Z}^{(\rm d)}_s$, $\mathcal{Z}^{(\rm a)}_s$ are exit measures constructed from the exit local times $\ell^{(\rm d)}_{s, i}$, $\ell^{(\rm a)}_{s, i}$ for $i \in I$. Such measures are constructed in \cite{BertoinLeGallLeJan1997} but under an additional assumption which is not satisfied in our setting.
\par
We give a remark which gives a connection of Theorem \ref{theorem ooSBM ffd convergence} to SBM.
\begin{remark}
    If we take the switching rate from active to dormant as $c=0$ in everything we have done above, then we obtain a limit which has the finite-dimensional distributions of SBM. This has already been done in Le Jan \cite{le1991superprocesses} and our result above can be thought as a generalisation to the case of ooSBM. 
\end{remark}
We state a result about the support of the dormant component of the Brownian snake that immediately follows from the finite-dimensional distributions result given in Theorem \ref{theorem ooSBM ffd convergence}.
\begin{corollary}\label{corollary dormant support}
    Let $(\mathcal{Z}^{(\rm d)}_r, \mathcal{Z}^{(\rm a)}_r)_{r \geq 0}$ be the process defined in Theorem \ref{theorem ooSBM ffd convergence} and let $s \geq s' \geq 0$. Firstly we have $\text{supp}(\mu_{\rm d}) \subseteq \text{supp}(\mathcal{Z}^{(\rm d)}_s)$ $\mathbb{P}_{(\mu_{\rm d}, \mu_{\rm a})}$-a.s. and secondly we have $\text{supp}(\mathcal{Z}^{(\rm d)}_{s'}) \subseteq \text{supp}(\mathcal{Z}^{(\rm d)}_s)$ $\mathbb{P}_{(\mu_{\rm d}, \mu_{\rm a})}$-a.s.
    \begin{proof}
        The first part follows immediately from the construction of $\mathcal{Z}^{(\rm d)}_s$ in Theorem \ref{theorem ooSBM ffd convergence}. The second part follows from the first part and the Markov property of ooSBM.
    \end{proof}
\end{corollary}
The intuition behind the result in Corollary \ref{corollary dormant support} is that any mass from the dormant component of the snake has some exponentially decaying part which has not switched to the active state, and so the support of the dormant component cannot decrease at any point.
\par
We also obtain a limit which gives the finite-dimensional distributions of another superprocess that will be useful in Section \ref{section range coupling proof} to prove Theorem \ref{theorem equality ranges}. We define another Brownian snake in Definition \ref{definition damped Brownian snake} to construct this superprocess. 
\begin{definition}[Damped Brownian snake]\label{definition damped Brownian snake}
    We construct the \emph{damped Brownian snake} from the on/off Brownian snake $\sum_{i \in I} \delta_{((b_i,s_i), f^i, x^i)}$ given in Definition \ref{definition oo snake} on the probability space $(\Omega, \mathcal{F}, \mathbb{P}_{(\mu_{\rm d},\mu_{\rm a})})$. Recall, for $i \in I$, $x^i=(w^i, v^i)$ where $w^i$ is the Brownian motion component and $v^i$ is the subordinator component. Let $\Tilde{I} \coloneqq \{ i \in I: s_i = 0\}$. For $i \in \Tilde{I}$, for $r \in [0,\sigma(f^i)]$, we define a path-valued process $\Tilde{v}_r^i$ to be the indicator function for whether a dormancy has occurred. Specifically, we set $\Tilde{v}_r^i(0)=0$ and for $s>0$ we set $\Tilde{v}_r^i(s) = 1$ if there exists $u \in (0,s]$ such that $v_r^i(u) >v^i_r(u-)$ (i.e., if there is a jump in the subordinator $v_r^i$ at time $s$ or earlier) and we set $\Tilde{v}_r^i(s)=0$ otherwise.  Also, for $i \in \Tilde{I}$, $r \in [0,\sigma(f^i)]$ let $t^i$ be the unit drift deterministic process defined by $t^i_r(s)=s$ for all $s \geq 0$. We then call 
    \begin{equation*}
        \sum_{i \in \Tilde{I}} \delta_{(b^i, f^i, (w^i, t^i, \Tilde{v}^i))}
    \end{equation*}
    the \emph{damped Brownian snake}.
\end{definition}
By construction we have that $(w^i, t^i, \Tilde{v}^i)$, $i \in \Tilde{I}$, are Brownian snakes with the discontinuous spatial motion of a Markov process with three independent components given by: a Brownian motion on $\mathbb{R}^d$; a deterministic unit drift; an independent component that starts at $0$ and moves to $1$ at time given by an Exponential random variable with parameter $c$. We use these snakes to construct another process with the finite dimensional distributions of a different superprocess.
\begin{proposition}[Damped super-Brownian motion]\label{proposition damped sbm}
    Recall the processes $(\Tilde{\mathcal{Z}}_y^{(\varepsilon)})_{y \geq 0}$, $\varepsilon>0$ from Corollary \ref{corollary bbm extra killing} on the probability space $(\Omega, \mathcal{F}, \mathbb{P}_{(\mu_{\rm d},\mu_{\rm a})})$. Then, for $s>0$, we have the existence of a scaling limit $\Tilde{\mathcal{Z}}_s$ given by 
    \begin{equation*}
        \varepsilon \Tilde{\mathcal{Z}}^{(\varepsilon)}_s \xrightarrow[\varepsilon \rightarrow 0]{a.s.} \Tilde{\mathcal{Z}}_s.
    \end{equation*}
    Furthermore, for any continuous bounded function $\Phi: \mathbb{R}^d \rightarrow [0,\infty)$, we have 
    \begin{equation}\label{equation damped sbm integral rep}
        \langle \Tilde{\mathcal{Z}}_s, \Phi \rangle \coloneqq \int_x \Phi(x) \Tilde{\mathcal{Z}}_s(dx) = \sum_{i \in \Tilde{I}} \int_{r=0}^{\sigma(f^i)} \Phi(\hat{w}^i_r)\mathbbm{1}_{\hat{\Tilde{v}}^i_r =0}dL^i_s(r)
    \end{equation}
    where, for $i \in \Tilde{I}$, $L^i_s(\cdot)$ are the local times for the excursions $f^i$ defined in \eqref{equation normalised local time}. Furthermore this measure-valued process $(\Tilde{\mathcal{Z}}_s)_{s \geq 0}$ has the finite dimensional distributions of a superprocess with continuous sample paths and with initial state $\Tilde{\mathcal{Z}}_0 = \mu_{\rm a}$. We call this superprocess \emph{damped super-Brownian motion with parameter $4$ and damping rate $c$} or \emph{$\text{DSBM}(4, c)$}. 
    \par
    Finally, when the prelimiting processes $\Tilde{\mathcal{Z}}^{(\varepsilon)}_s$ which are $\text{BBM}(2/\varepsilon, 2/\varepsilon+c)$ are replaced by $\text{BBM}(\gamma/2\varepsilon, \gamma/2\varepsilon + c)$ then the limit in distribution is another superprocess that we call \emph{damped super-Brownian motion with parameter $\gamma$ and damping rate $c$} or \emph{$\text{DSBM}(\gamma, c)$}. 
    \begin{proof}
        We can apply Lemma \ref{lemma snake exit open set general} to the Brownian snakes that form the damped Brownian snake of Definition \ref{definition damped Brownian snake}. For the snakes process $(w^i, t^i, \Tilde{v}^i)$, $i \in \Tilde{I}$, we let the exit set $D$ be given by $\mathbb{R}^d \times [0, y) \times \{0\}$. It is elementary to show that the exit local time in this case has associated Lebesgue-Steiltjes measure given by $\mathbbm{1}_{\hat{\Tilde{v}}^i_r =0}dL^i_s(r)$. By the analogous arguments that were given in Theorem \ref{theorem ooSBM ffd convergence} we obtain the a.s. convergence and the representation for $\Tilde{\mathcal{Z}}_y$.
        \par
        The scaling limit here is analogous to the standard scaling limit result for SBM given in Proposition \ref{proposition SBM existence} and it is an instance of the general result given in \cite[Chapter II]{le1999spatial} which shows the existence of superprocesses as scaling limits.
    \end{proof}
\end{proposition}
In the setting of Le Gall \cite{le1999spatial}, the $\text{DSBM}(\gamma, c)$ has branching mechanism $\phi(u) = (\gamma/2)u^2 + cu$ compared to $\text{SBM}(\gamma)$ which has branching mechanism $\phi(u) = (\gamma/2)u^2$. The process $\text{DSBM}(\gamma, c)$ has been studied in the literature, \cite{abraham2002poisson, abraham2002representations, Dhersin_Serlet_2000} along with generalisations of these processes. See also \cite[Proposition 4]{abraham2002poisson} to see how the thinning procedure from the term $\mathbbm{1}_{\hat{\Tilde{v}}^i_r =0}$ in \eqref{equation damped sbm integral rep} is equivalent to using excursions $f^i_r$ for Brownian motion with negative drift, and how this then corresponds to the branching mechanism $\phi(u) = (\gamma/2)u^2 + cu$.
\par
We note that another way to see the a.s.\ convergence result in Proposition \ref{proposition damped sbm} is to take $\Tilde{c} = 0$ in Theorem \ref{theorem ooSBM ffd convergence} and apply the results there. In that case, the active component of the ooBBM is a $\text{BBM}(2/\varepsilon, 2/\varepsilon + c)$ with an initial cluster of particles distributed as a Poisson random measure with intensity $\mu_{\rm a}$ and this converges a.s.\ at each fixed time.
\par
We define the probability measure $\Tilde{\mathbb{P}}_{\mu_{\rm a}}(\cdot) \coloneqq \mathbb{P}_{(\mu_{\rm d}, \mu_{\rm a})}((\Tilde{\mathcal{Z}}_y)_{y \geq 0} \in \cdot)$; note that this definition does not depend on $\mu_{\rm d}$. For $x \in \mathbb{R}^d \times [0,\infty)$ recall the probability measure $\mathbb{P}_x$ from the description below Theorem \ref{theorem bivariate brownian snake}. We construct $(\Tilde{\mathcal{Z}}_y)_{y \geq 0}$ under $\mathbb{P}_x$ in the analogous way for our construction of $(\Tilde{\mathcal{Z}}_y)_{y \geq 0}$ under $\mathbb{P}_{(\mu_{\rm d}, \mu_{\rm a})}$ given in Proposition \ref{proposition damped sbm}. Furthermore, for $b \in \mathbb{R}^d$, we define the corresponding excursion measure 
\newline
$\Tilde{\mathbb{N}}_b(\cdot) \coloneqq \int_f n(df) \mathbb{P}_{(b,0)|f}((\Tilde{\mathcal{Z}}_y)_{y \geq 0} \in \cdot)$, recall Definition \ref{definition brownian excursion} and the description below that definition.
\par
We now state a general result that is useful in our context of defining measures using Brownian snakes.
\begin{lemma}\label{lemma support continuous integral}
    Let $f:[0,\sigma] \rightarrow \mathbb{R}^d$ be a continuous function and let $L:[0,\sigma] \rightarrow [0,\infty)$ be a continuous increasing function with $L(0)=0$. Consider the Lebesgue-Stieltjes measure $dL$ defined from $L$. Let $Y$ be the measure defined via Borel-measurable functions $\Phi:\mathbb{R}^d \rightarrow [0,\infty)$ as 
    \begin{equation*}
        \langle Y, \Phi \rangle \coloneqq \int_{b \in \mathbb{R}^d} \Phi(b)Y(db) := \int_{r=0}^\sigma \Phi(f(r)) dL(r).
    \end{equation*}
    Then we have that
    \begin{equation*}
        \text{supp}(Y) = \{f(r): r \in \text{supp}(dL)\}.
    \end{equation*}
\end{lemma}
This result is elementary and $Y$ is referred to as the \emph{pushforward} of $dL$ under $f$. We note from Lemma \ref{lemma support continuous integral} that, under suitable conditions, the support of an exit measure is the range of the head of the snake over the support of the exit local time a.s. We conclude this section with the support properties of $(\mathcal{Z}^{(\rm d)}_s, \mathcal{Z}^{(\rm a)}_s)_{s \geq 0}$ and $(\Tilde{\mathcal{Z}}_y)_{y \geq 0}$ defined in Theorem \ref{theorem ooSBM ffd convergence} and Proposition \ref{proposition damped sbm}. For $s \geq 0$ we define
\begin{equation}\label{equation overlayed supports}
    \mathcal{Z}_s \coloneqq \mathcal{Z}^{(\rm d)}_s + \mathcal{Z}^{(\rm a)}_s.
\end{equation}
\begin{lemma}\label{lemma supports sbm and oosbm}
    Let $(\mathcal{Z}^{(\rm d)}_r, \mathcal{Z}^{(\rm a)}_r)_{r \geq 0}$ be the process given in Theorem \ref{theorem ooSBM ffd convergence} and let $(\mathcal{Z}_r)_{r \geq 0}$ be constructed from it as in \eqref{equation overlayed supports}. Let $(\Tilde{\mathcal{Z}}_r)_{r \geq 0}$ be the process given in Proposition \ref{proposition damped sbm}. Let $s > 0$. Then we have $\mathbb{P}_{(\mu_{\rm d}, \mu_{\rm a})}$-a.s. that
    \begin{equation*}
        \text{supp}(\mathcal{Z}_s) = \text{supp}(\mathcal{Z}^{(\rm d)}_s) \cup \text{supp}(\mathcal{Z}^{(\rm a)}_s) = \text{supp}(\mu_{\rm d}) \cup \overline{\bigcup_{i \in I} \{\hat{w}^i_{T^{1/m}_{s,n,i}}: m \in \mathbb{N}, 1 \leq n \leq M_i^{1/m}\}};
    \end{equation*}
    \begin{equation*}
        \text{supp}(\Tilde{\mathcal{Z}}_s) = \overline{\bigcup_{i \in \Tilde{I}}\{\hat{w}^i_r: r \in [0,\sigma(f^i)], f^i_r = s, \hat{\Tilde{v}}^i_r=0\}},
    \end{equation*}
    \begin{proof}
        Fix $s>0$. We split the measure $\mathcal{Z}_s$ into measures $\mathcal{Z}_{s, i}$, for $i \in I$, which consist of the active and dormant contributions from the snakes $x^i$, $i \in I$. Specifically, for $i \in I$, $\mathcal{Z}_{s, i}$ is the measure such that for any Borel-measurable function $\Phi$ we have
        \begin{equation}\label{equation components of total measure oosbm}
            \langle \mathcal{Z}_{s, i}, \Phi \rangle = \int_{r=0}^{\sigma(f^i)} \Phi(\hat{w}^i_r) d\ell_{s, i}(r)
        \end{equation}
        and it follows that, by Theorem \ref{theorem ooSBM ffd convergence} that
        \begin{equation*}
            \mathcal{Z}_s = e^{-\Tilde{c}s}\mu_{\rm d} + \sum_{i \in I:s_i<s}\mathcal{Z}_{s, i}.
        \end{equation*}
        The local time processes $(\ell_{s, i}(r))_{r \geq 0}$ are all continuous and increasing a.s. and so, as the processes $(\hat{w}^i_r)_{r \geq 0}$ are also all continuous a.s., we have, by Lemma \ref{lemma support continuous integral}, that
        \begin{equation*}
            \text{supp}(\mathcal{Z}_{s, i}) = \{\hat{w}^i_r: r \in \text{supp}(d\ell_{s, i}) \}
        \end{equation*}
        for all $i \in I$ a.s. For the remainder of this proof of this first result we work under this almost-sure event. Fix $i \in I$. We note that 
        \begin{equation*}
            \{\hat{w}^i_r: r \in \text{supp}(d\ell_{s, i}) \} = \{\hat{w}^i_{\eta(s, u, i)}: u \in \text{supp}(d\hat{\ell}_{s, i})\} 
        \end{equation*}
        as $\ell_{s, i}$ and $\hat{\ell}_{s, i}$ are related by a time change given by the inverse of $\eta$. To see this, take $\tau = \int_{r=0}^{\sigma(f^i)} \mathbbm{1}_{\text{Res}^i_r(s)>0} dr$ and define a right inverse $\Tilde{\eta}:[0,\sigma(f^i)] \rightarrow [0,\tau]$ of $\eta(s, \cdot, i)$ by $\Tilde{\eta}(w) = \int_{r=0}^w \mathbbm{1}_{\text{Res}^i_r(s)>0} dr$. Then $\Tilde{\eta}$ is a continuous function and $\ell_{s, i}$ is a continuous increasing function, and therefore we have for any Borel-measurable $\Phi$
        \begin{equation*}
            \langle d\hat{\ell}_{s, i}, \Phi \rangle = \int_{x=0}^{\sigma} \Phi(\Tilde{\eta}(x))d\ell_{s, i}(x).
        \end{equation*}
        Therefore, by Lemma \ref{lemma support continuous integral}, we have that $\text{supp}(\hat{\ell}_{s, i}) = \{\Tilde{\eta}(x): x \in \text{supp}(\ell_{s, i})\}$ and so
        \begin{equation*}
            \text{supp}(d\ell_{s, i}) = \{\eta(s, u, i): u \in \text{supp}(d\hat{\ell}_{s, i})\}.
        \end{equation*}
        \par
        The set $\{r \in [0,\sigma(f^i)], f^i_r > H^i_r(s)\}$ is open, see Lemma \ref{lemma age process properties}, and so we can write it as $\bigcup_{j \in J_i} (a_j, b_j)$ for disjoint intervals $(a_j, b_j)$ of $[0,\sigma(f^i)]$ and some index set $J_i$. Note that 
        \begin{equation*} 
            \eta(s, [0,\tau], i) = \{r \in [0, \sigma(f^i)] | \exists r' \in [0,\tau]: \eta(s,r',i)=r\} = \bigcup_{j \in J_i} [a_j, b_j).
        \end{equation*}
        From this it follows that
        \begin{equation*}
            \text{supp}(d\ell_{s, i}) \subseteq \overline{\bigcup_{j \in J_i}[a_j, b_j)}.
        \end{equation*}
        For any $x \in (a_j, b_j)$ it is a simple argument to show that $x \notin \text{supp}(\ell_{s, i})$ and so
        \begin{equation*}
            \text{supp}(d\ell_{s, i}) \subseteq \overline{\bigcup_{j \in J_i}\{a_j, b_j\}}.
        \end{equation*}
        As $\hat{w}^i_{a_j} = \hat{w}^i_{b_j}$ (by the snake property) we then have
        \begin{equation*}
            \text{supp}(\mathcal{Z}_{s, i}) \subseteq \overline{\bigcup_{j \in J_i} \{\hat{w}^i_{b_j}\}}.
        \end{equation*}
        Almost surely, the support of local time at level $0$ for reflecting Brownian motion is precisely the zero-set (the set of points for which the process takes the value $0$), see e.g.\ \cite[Proposition VI.2.5]{revuz2013continuous}. Therefore, we also have almost surely that the zero-set of $u \mapsto \text{Res}^i_{\eta(s,u,i)}$ is exactly the support of $\hat{\ell}_{s, i}$, for all $i \in I$ a.s. Putting all this together we find that almost surely for all $i \in I$
        \begin{multline*}
            \overline{\{\hat{w}^i_{T^{1/m}_{s,n,i}}: m \in \mathbb{N}, 1 \leq n \leq M_i^{1/m}\}} = \overline{\bigcup_{j \in J_i} \{\hat{w}^i_{b_j}}\} \subseteq \overline{\{\hat{w}^i_{\eta(u,s,i)}: u \in \text{supp}(d\hat{\ell}_{s, i})\}} \\
            = \text{supp}(\mathcal{Z}_{s, i}) \subseteq \overline{\bigcup_{j \in J_i} \{\hat{w}^i_{b_j}\}} = \overline{\{\hat{w}^i_{T^{1/m}_{s,n,i}}: m \in \mathbb{N}, 1 \leq n \leq M_i^{1/m}\}}
        \end{multline*}
        and therefore all the set-inclusions are in fact set-equalities. Taking a union over $i \in I$ and then taking the closure of this union gives the almost-sure result for the support of $\mathcal{Z}_s$.
        \par
        For DSBM our proof is slightly different. We split the measure $\Tilde{\mathcal{Z}}_s$ into components $\Tilde{\mathcal{Z}}_{s, i}$, $i \in \Tilde{I}$, from the snakes $x^i$. For $i \in \Tilde{I}$, consider
        \begin{equation*}
            \Tilde{L}^i_s(t) \coloneqq \int_{r=0}^t \mathbbm{1}_{\hat{\Tilde{v}}^i_r = 0} dL^i_s(r).
        \end{equation*}
        We note that $(\Tilde{L}^i_s(t))_{t \geq 0}$, $i \in I$, are all a.s. continuous and increasing processes and, with the associated Lebesgue-Stieltjes measures $\Tilde{L}^i_s(\cdot)$, we consider the measures $\Tilde{\mathcal{Z}}_{s, i}$ that satisfy
        \begin{equation*}
            \langle \Tilde{\mathcal{Z}}_{s, i}, \Phi \rangle = \int_{b \in \mathbb{R}^d} \Phi(b) \Tilde{\mathcal{Z}}_{s, i}(db) = \int_{r=0}^{\sigma(f^i)} \Phi(f^i_s(r))d\Tilde{L}^i_s(r)
        \end{equation*}
        for any Borel-measurable function $\Phi$. It follows that $\Tilde{\mathcal{Z}}_s$ is the sum of the measures $\Tilde{\mathcal{Z}}_{s, i}$, $i \in I$. We also note that $d\Tilde{L}^i_s(t) = \mathbbm{1}_{\hat{\Tilde{v}}^i_t = 0} dL^i_s(t)$. From this representation we claim that
        \begin{equation*}
            \text{supp}(d\Tilde{L}^i_s) = \overline{\{r \in [0, \sigma(f^i)]:  f^i_r=s, \hat{\Tilde{v}}^i_r = 0\}}.
        \end{equation*}
        To see this, we can use the classical result for SBM (Theorem \ref{theorem sbm snake}) that a.s. the support for the local time of the Brownian excursion at level $s$ is given by $\{r \in [0, \sigma(f^i)]:  f^i_r=s\}$. From here is it a simple argument with $d\Tilde{L}^i_s(t) = \mathbbm{1}_{\hat{\Tilde{v}}^i_t = 0} dL^i_s(t)$ to show that $x$ is in the support of $\Tilde{L}^i_s$ if and only if $x$ is in $\overline{\{r \in [0, \sigma(f^i)]:  f^i_r=s, \hat{\Tilde{v}}^i_r = 0\}}$. Now, since the support of a sum of measures is the closure of the union of the supports of those measures, then, again by applying Lemma \ref{lemma support continuous integral}, we have that the support of $\Tilde{\mathcal{Z}}_s$ is as stated in the proposition.
    \end{proof}
\end{lemma}
\section{Range coupling and monotone increasing support results}\label{section range coupling proof}
In this section we prove Theorems \ref{theorem equality ranges} and \ref{theorem increasing support}. We continue to work with the on/off Brownian snake 
\begin{equation*}
    \sum_{i \in I} \delta_{((b_i,s_i), f^i, x^i)}
\end{equation*}
on the probability space $(\Omega, \mathcal{F}, \mathbb{P}_{(\mu_{\rm d},\mu_{\rm a})})$. On this probability space we define couplings of $\text{ooSBM}(4, c, \Tilde{c})$ with $\text{DSBM}(4, c)$ and $\text{SBM}(4)$. Our results apply for general initial measures $\mu_{\rm d}$, $\mu_{\rm a} \in \mathcal{M}_F(\mathbb{R}^d)$ such that $\mu_{\rm d} + \mu_{\rm a} \neq 0$ but for some of the intermediate results we consider specific cases.
\par
We state two standard results in Propositions \ref{proposition ooSBM continuous modification} and \ref{proposition closure of range of continuous process}: Proposition \ref{proposition ooSBM continuous modification} states the existence of a continuous modification of a process with the finite-dimensional distributions of a continuous process; Proposition \ref{proposition closure of range of continuous process} concerns the range of measure-valued processes that are continuous with respect to the Prokhorov metric.
\begin{proposition}\label{proposition ooSBM continuous modification}
    Let $(E',d')$ be a complete and separable metric space and $(\Omega, \mathcal{F}, \mathbb{P})$ a probability space on which we have random variables $\mathcal{Q}_t: \Omega \rightarrow E'$, $t \geq 0$. Suppose further that there is another probability space $(\Omega', \mathcal{F}', \mathbb{P}')$ on which random variables $\mathcal{Q}'_t: \Omega' \rightarrow E'$ form a path-continuous process with the same finite-dimensional distributions as $\mathcal{Q} = (\mathcal{Q}_t, t \geq 0)$. Then $(\mathcal{Q}_t, t \geq 0)$ has a modification $(\Tilde{\mathcal{Q}}_t, t \geq 0)$ with continuous paths, i.e. for each fixed $t \geq 0$ we have $\mathbb{P}(\Tilde{\mathcal{Q}}_t = \mathcal{Q}_t) = 1$. 
\end{proposition}
This result is elementary and the idea behind the proof is that uniform continuity on $[0,T] \cap \mathbb{Q}$, for $T>0$, is $\mathcal{B}(E')^{[0,\infty) \cap \mathbb{Q}}$-measurable and from this we can define a continuous modification of $(\mathcal{Q}_t, t \geq 0)$. 
\begin{corollary}\label{corollary continuous modifications}
    There exist continuous modifications of the processes $(\mathcal{Z}^{(\rm d)}_s, \mathcal{Z}^{(\rm a)}_s)_{s \geq 0}$ and $(\Tilde{\mathcal{Z}}_s)_{s \geq 0}$ from Theorem \ref{theorem ooSBM ffd convergence} and Proposition \ref{proposition damped sbm} which we denote by $(\mathcal{Y}^{(\rm d)}_s, \mathcal{Y}^{(\rm a)}_s)_{s \geq 0}$ and $(\Tilde{\mathcal{Y}}_s)_{s \geq 0}$ and which are $\text{ooSBM}(4, c, \Tilde{c})$ and $\text{DSBM}(4, c)$.
    \par
    For $s \geq 0$, we define
    \begin{equation*}
        \mathcal{Y}_s \coloneqq \mathcal{Y}^{(\rm d)}_s + \mathcal{Y}^{(\rm a)}_s.
    \end{equation*}
\end{corollary}
\begin{proposition}\label{proposition closure of range of continuous process}
    Let $(\mathcal{Q}_t)_{t \geq 0}$ be $\mathcal{M}_F(E)$-valued function that is continuous with respect to the Prokhorov metric. Then we have for all $t \geq 0$ that
    \begin{equation*}
        \overline{\bigcup_{0 \leq s \leq t} \text{supp}(\mathcal{Q}_s)} = \overline{\bigcup_{s \in \mathbb{Q}, 0 \leq s \leq t} \text{supp}(\mathcal{Q}_s)}.
    \end{equation*}
\end{proposition}
The idea behind the proof for Proposition \ref{proposition closure of range of continuous process} is that if $x \in \text{supp}(\mathcal{Q}_s)$ for some $s \geq 0$ then for every open set $U$ with $x \in U$ we have $\mathcal{Q}_s(U)>0$. By the continuity of $(\mathcal{Q}_t, t \geq 0)$ for a sequence of rationals $t_n \rightarrow s$ we have $\liminf_{n \rightarrow \infty} \mathcal{Q}_{t_n}(U) \geq \mathcal{Q}_s(U)>0$ by the Portmanteau Theorem, see e.g. Kallenberg \cite{KallenbergOlav2021Fomp}. By choosing the open sets $U$ containing open balls of size $\varepsilon$ and take $\varepsilon$ to $0$ we can show that there exists a sequence of points $x_k \in \bigcup_{s \in \mathbb{Q}, 0 \leq s \leq t} \text{supp}(\mathcal{Q}_s)$ such that $x_k \rightarrow x$.
\begin{lemma}\label{lemma support sbm interval}
    Let $(\mathcal{Y}^{(\rm d)}_s, \mathcal{Y}^{(\rm a)}_s)_{s \geq 0}$ be the $\text{ooSBM}(4, c, \Tilde{c})$ defined in Corollary \ref{corollary continuous modifications} and let $\delta > 0$. Then,
    \newline
    $\mathbb{P}_{(\mu_{\rm d}, \mu_{\rm a})}$-a.s. we have
    \begin{equation*}
        \overline{\bigcup_{s \geq \delta} \text{supp}(\mathcal{Y}_s)} = \text{supp}(\mu_{\rm d}) \cup \overline{\bigcup_{i \in I} \{\hat{w}^i_r: r \in [0,\sigma(f^i)], f^i_r > H^i_r(\delta)\}}.
    \end{equation*}
    \begin{proof}
        We first note from Proposition \ref{proposition closure of range of continuous process} that, for a continuous measure-valued process, the closure of the union of the supports over the rational points is equal to the closure of the union over all points strictly greater than $\delta$ and therefore we have
        \begin{equation*}
            \overline{\bigcup_{s \geq \delta} \text{supp}(\mathcal{Y}_s)} = \overline{\bigcup_{s > \delta, s \in \mathbb{Q}} \text{supp}(\mathcal{Z}_s)} \hspace{0.5cm} \text{a.s.}
        \end{equation*}
        by Corollary \ref{corollary continuous modifications} and Proposition \ref{proposition closure of range of continuous process}. By Lemma \ref{lemma supports sbm and oosbm} this is equal to the support of $\mu_{\rm d}$ and
        \begin{equation}\label{equation support greater than delta}
             \overline{\bigcup_{i \in I} \bigcup_{s > \delta, s \in \mathbb{Q}} \{\hat{w}^i_{T^{1/m}_{s,n,i}}: m \in \mathbb{N}, 1 \leq n \leq M_i^{1/m}\}} = \overline{\bigcup_{i \in I} \bigcup_{s \in \{\delta\} \cup (\mathbb{Q}\cap(\delta, \infty))} \{\hat{w}^i_{T^{1/m}_{s,n,i}}: m \in \mathbb{N}, 1 \leq n \leq M_i^{1/m}\}}
        \end{equation}
        a.s., where, we have again used Proposition \ref{proposition closure of range of continuous process} that the inclusion of the level $\delta$ in the union does not change the closure of the set of points covered. We work under this a.s. event for the rest of the proof. To show the first result, we first show that the left-hand-side of \eqref{equation support greater than delta} is a subset of 
        \begin{equation}\label{equation closure head snake range to infinity}
            \overline{\bigcup_{i \in I} \{\hat{w}^i_r: r \in [0,\sigma(f^i)], f^i_r > H^i_r(\delta)\}}.
        \end{equation}
        Let $i \in I$, $s>\delta$, $s \in \mathbb{Q}$, let $m \in \mathbb{N}$, and let $1 \leq n \leq M^{1/m}_i$. If $H^i_{T^{1/m}_{s, n, i}}(\delta)<f^i_{T^{1/m}_{s, n, i}}$ we have $\hat{w}^i_{T^{1/m}_{s, n, i}}$ is in \eqref{equation closure head snake range to infinity} and so we are done. Otherwise, suppose $H^i_{T^{1/m}_{s, n, i}}(\delta) = f^i_{T^{1/m}_{s, n, i}}$. As we also have that $T^{1/m}_{s, n, i}$ is a downcrossing time to level $s$, for the fixed $s>\delta$, this means that a dormancy event has caused the stopping at level $s$ for this downcrossing. Specifically, for the subordinator component $v^i$ of the snake, we have that $v^i_{S^{1/m}_{s, n, i}}(f^i_{T^{1/m}_{s, n, i}}) - v^i_{S^{1/m}_{s, n, i}}(f^i_{T^{1/m}_{s, n, i}}-) \geq s - \delta > 0$. Therefore we have that for $r \in [S^{1/m}_{s, n, i}, T^{1/m}_{s, n, i})$ that $f^i_r>H_r^i(\delta)$ because for all $r \in [S_{s,n,i}^{1/m}, T_{s,n,i}^{1/m})$ we have that $f_r^i > H_r^i(s)$. and $H_r^i(s) \geq H_r^i(\delta)$, since $s > \delta$ and for fixed $r$ the function $t \mapsto H_r(t)$ is increasing. Therefore the points $\hat{w}^i_r$, $r \in [S^{1/m}_{s, n, i}, T^{1/m}_{s, n, i})$, are in \eqref{equation closure head snake range to infinity}. As \eqref{equation closure head snake range to infinity} is a closed set and as the head of the snake is continuous we therefore have that $\hat{w}^i_{T^{1/m}_{s, n, i}}$ is in \eqref{equation closure head snake range to infinity} as required.  
        \par
        We now show for each $i \in I$ that the set 
        \begin{equation}\label{equation Hi grater than delta}
            \{\hat{w}^i_r: r \in [0,\sigma(f^i)], f^i_r > H^i_r(\delta)\}
        \end{equation}
        is a subset of the right-hand-side of \eqref{equation support greater than delta}. As the set $\{r \in [0,\sigma(f^i)], f^i_r > H^i_r(\delta)\}$ is open (see Lemma \ref{lemma age process properties}) it suffices to show that for a dense subset $D$ of $\{r \in [0,\sigma(f^i)], f^i_r > H^i_r(\delta)\}$, the points $\hat{w}^i_r$, $r \in D$, are in the right-hand-side of \eqref{equation support greater than delta}. We note that, for all $i \in I$, the local minimisers of the excursions $f^i$ are a.s. all strict local minimisers, that is to say, for $i \in I$, that if $x$ is a local minima then there exists $\varepsilon>0$ such that for all $y \in (x-\varepsilon, x+\varepsilon) \backslash \{x\}$ we have $f^i_y>f^i_x$ (i.e. a strict inequality rather than a weak inequality). We also have that, for all $i \in I$, the strict local minimisers form dense countable sets of $[0,\sigma(f^i)]$ a.s., see e.g. \cite[Result 2.9.12]{karatzas2012brownian}. We work under the intersection of these a.s. events for the remainder of the proof. 
        \par
        Under these a.s. events, it suffices to consider $i \in I$, $r \in [0, \sigma(f^i)]$, such that $f^i_r>H^i_r(\delta)$ and $r$ is a strict local minimiser of the excursion $f^i$. Let $r$ be such a point and let $u \coloneqq \inf \{u':H^i_r(u') = f^i_r \}$. Then as $(H^i_r(u'))_{u' \geq 0}$ is continuous and $H^i_r(\delta)<f^i_r$ we have $u > \delta$. We also have that as $f^i_r$ is a strict local minima of $f^i$ that $r$ must be a downcrossing time to level $u$ of the age process, i.e.\ there exists $k \in \mathbb{N}$, $1 \leq n \leq M_i^{1/k}$ such that $r = T^{1/k}_{u, n, i}$. If $u \in \mathbb{Q}$ then we have that $\hat{w}^i_r$ is in the right-hand-side of \eqref{equation support greater than delta}.
        \par
        Otherwise, we use the fact that $r$ is a strict local minimiser. Let $\varepsilon > 0$ such that $f^i_{r'}>f^i_r$ for all $r' \in [r-\varepsilon, r)$ and let $a = f^i_{r-\varepsilon}-f^i_r$. Let $u_{k'}$, $k' \in \mathbb{N}$ be a sequence such that $u_{k'} \downarrow u$, $u_{k'} \in \mathbb{Q} \cap (u, u+a)$. Then we can define the stopping times $\tau_{k'}$, $k' \in \mathbb{N}$ by $\tau_{k'} = \inf\{r'>S^{1/k}_{u, n, i}: \text{Res}^i_{r'}(u_{k'}) = 0\}$. We have that these $\tau_{k'}$ are the endpoints for downcrossings to levels $u_{k'}$ of the age process. Therefore the $\hat{w}^i_{\tau_{k'}}$, $k' \in \mathbb{N}$, are in the right-hand-side of \eqref{equation support greater than delta}. We also have that $\tau_{k'} \rightarrow r$ (as $r$ is the strict local minimiser there can be no other limit point for the sequence $\tau_{k'}$, $k' \in \mathbb{N}$) and so by the continuity of the head of the snake $w^i$ we have that $\hat{w}^i_r$ is in the right-hand-side of \eqref{equation support greater than delta}.
    \end{proof}
\end{lemma}
\begin{lemma}\label{lemma support sbm interval 2}
    Let $\delta >0$. In the same setting as Lemma \ref{lemma support sbm interval}, we have that
        \begin{equation*}
        \overline{\bigcup_{0 \leq s \leq \delta} \text{supp}(\mathcal{Y}_s)} = \text{supp}(\mu_{\rm d}) \cup \overline{\bigcup_{i \in I} \{\hat{w}^i_r: r \in [0,\sigma(f^i)], f^i_r = H^i_r(\delta)\}} \hspace{0.5cm} \mathbb{P}_{(\mu_{\rm d}, \mu_{\rm a})}\text{-a.s.}
    \end{equation*}
    \begin{proof}
        For $i \in I$ recall the measures $\mathcal{Z}_{s, i}$ from \eqref{equation components of total measure oosbm}. We first note that
        \begin{equation*}
            \overline{\bigcup_{0 \leq s \leq \delta} \text{supp}(\mathcal{Y}_s)} = \overline{\bigcup_{s \in (0,\delta) \cap \mathbb{Q}} \text{supp}(\mathcal{Z}_s)} = \text{supp}(\mu_d) \cup \overline{\bigcup_{i \in I} \bigcup_{s \in (0,\delta) \cap \mathbb{Q}} \text{supp}(\mathcal{Z}_{s, i})}.
        \end{equation*}
        To prove the result we now show for each $i \in I$ that
        \begin{equation*}
            \overline{\bigcup_{s \in (0,\delta) \cap \mathbb{Q}} \text{supp}(\mathcal{Z}_{s, i})} = \overline{\{\hat{w}^i_r: r \in [0,\sigma(f^i)], f^i_r = H^i_r(\delta)\}}.
        \end{equation*}
        The first inclusion follows from the proof of Lemma \ref{lemma supports sbm and oosbm} which states that 
        \begin{equation*}
            \text{supp} (\mathcal{Z}_{s,i}) = \overline{\{\hat{w}^i_{T^{1/m}_{s,n,i}}: m \in \mathbb{N}, 1 \leq n \leq M_i^{1/m}\}}
        \end{equation*}
        and we note that for all $s<\delta$, $s \in \mathbb{Q}$, $m \in \mathbb{N}, 1 \leq n \leq M_i^{1/m}$, we have that $H^i_{T^{1/m}_{s, n, i}}(\delta)=f^i_{T^{1/m}_{s, n, i}}$, since $T_{s,n,i}^{1/m}$ is a downcrossing time. From this, it follows that
        \begin{equation*}
            \bigcup_{s \in (0,\delta) \cap \mathbb{Q}} \text{supp}(\mathcal{Z}_{s, i}) \subseteq \overline{\{\hat{w}^i_r: r \in [0,\sigma(f^i)], f^i_r = H^i_r(\delta)\}}.
        \end{equation*}
        It remains to show the reverse inclusion. We consider
        \begin{equation*}
            \sum_{i \in I} \delta_{((b_i,s_i), f^i, x^{i, \delta})}
        \end{equation*}
        where $I$, $(b_i, s_i)$, $f^i$ are the same as for the original Brownian snake defined in \eqref{equation snake poisson random measure} and where $x^{i, \delta}$, $i \in I$, are defined as follows. For $i \in I$, we let $x^{i, \delta} = (w^{i, \delta}, v^i)$ where $v^i$ is the usual subordinator component for $x^i=(w^i, v^i)$, and where for $r \in [0,\sigma(f^i)]$, $u \in [0,f^i_r]$, we define $w^{i, \delta}_r(u) = w^i_r(u \wedge H^i_r(\delta))$. It follows that $x^{i, \delta}$ is a bivariate Brownian snake where the second component is a subordinator and the first component is Brownian path stopped at a time determined by the subordinator. We claim that
        \begin{equation*}
            \{\hat{w}^i_r: r \in [0,\sigma(f^i)], f^i_r = H^i_r(\delta)\} = \{\hat{w}^{i, \delta}_r: r \in [0,\sigma(f^i)], f^i_r = H^i_r(\delta)\}.
        \end{equation*}
        This is because for $r \in [0,\sigma(f^i)]$ with $f^i_r = H^i_r(\delta)$ we have that 
        \begin{equation*}
            \hat{w}^i_r = w^i_r(f^i_r) = w^i_r(f^i_r \wedge H^i_r(\delta)) = \hat{w}^{i, \delta}_r. 
        \end{equation*}
        We can also define a measure $\mathcal{Z}^{\delta}_{s, i}$ by
        \begin{equation*}
            \langle \mathcal{Z}^{\delta}_{s, i}, \Phi \rangle = \int_{r=0}^{\sigma(f^i)} \Phi(\hat{w}^{i, \delta}_r) d\ell_{s, i}(r)
        \end{equation*}
        for any Borel-measurable function $\Phi$. Like our original snake, $x^i$, our new snake, $x^{i, \delta}$, has head of the snake process $(\hat{w}^{i, \delta}_r)_{0 \leq r \leq \sigma(f^i)}$ which is continuous. We can apply the results of Lemmas \ref{lemma supports sbm and oosbm} and \ref{lemma support sbm interval} to $(\mathcal{Z}^{\delta}_{s, i}, s \geq 0)$. Specifically, we can let $(\mathcal{Y}^{\delta}_{s, i})_{s \geq 0}$ be a continuous modification of $(\mathcal{Z}^{\delta}_{s, i})_{s \geq 0}$ and note that for any $a>0$
        \begin{equation*}
            \overline{\bigcup_{s \geq a} \text{supp}(\mathcal{Y}^{\delta}_{s, i})} = \overline{\{\hat{w}^{i, \delta}_r: r \in [0,\sigma(f^i)], f^i_r > H^i_r(a)\}}.
        \end{equation*}
        Taking $a$ to $0$ we see that
        \begin{equation*}
            \overline{\bigcup_{s \geq 0} \text{supp}(\mathcal{Y}^{\delta}_{s, i})} = \overline{\{\hat{w}^{i, \delta}_r: r \in [0,\sigma(f^i)]\}}.
        \end{equation*}
        Now, we claim that, for $s\leq \delta$, $s \in \mathbb{Q}$ that $\text{supp}\mathcal{Z}^{\delta}_{s, i} = \text{supp}(\mathcal{Z}_{s, i})$. This is because, by Lemma \ref{lemma supports sbm and oosbm}, we have that 
        \begin{equation*}
            \text{supp}(\mathcal{Z}^{\delta}_{s, i}) = \overline{\{\hat{w}^{i, \delta}_{T^{1/m}_{s,n,i}}: m \in \mathbb{N}, 1 \leq n \leq M_i^{1/m}\}}
        \end{equation*}
        and as $T^{1/m}_{s,n,i}$ are downcrossing times for level $s$ we have that $\hat{w}^{i, \delta}_{T^{1/m}_{s,n,i}} = \hat{w}^i_{T^{1/m}_{s,n,i}}$ for all $n$, $m$, $s \in \mathbb{Q}$, $s \leq \delta$. Therefore we have that
        \begin{equation*}
            \overline{\bigcup_{0 \leq s \leq \delta} \text{supp}(\mathcal{Y}_{s, i})} = \overline{\bigcup_{0 \leq s \leq \delta} \text{supp}(\mathcal{Y}^{\delta}_{s, i})}.
        \end{equation*}
        Now, let $x \in \{\hat{w}^i_r: r \in [0,\sigma(f^i)], f^i_r = H^i_r(\delta)\}$. Then there exists $r \in [0,\sigma(f^i)]$ such that $x = \hat{w}^i_r = \hat{w}^{i, \delta}_r$. Therefore $x \in \overline{\bigcup_{s \geq 0} \text{supp}(\mathcal{Y}^{\delta}_{s, i})}$. If $x \in \text{supp}(\mathcal{Y}^{\delta}_{s, i})$ for some $s \leq \delta$ then we are done. Otherwise, we claim that 
        \begin{equation*}
            \overline{\bigcup_{s \geq \delta} \text{supp}(\mathcal{Y}^{\delta}_{s, i})} = \text{supp}(\mathcal{Y}^{\delta}_{\delta, i})
        \end{equation*}
        and from this the result follows. Now, following the same proof as Lemma \ref{lemma support sbm interval}, we have that 
        \begin{equation*}
            {\bigcup_{s \geq \delta} \text{supp}(\mathcal{Y}^{\delta}_{s, i})} = \overline{\{\hat{w}^{i, \delta}_r: r \in [0,\sigma(f^i)], f^i_r > H^i_r(\delta)\}}.
        \end{equation*}
        Let, $r \in [0,\sigma(f^i)]$ with $f^i_r > H^i_r(\delta)$. Let $T = \inf\{r' > r: f^i_{r'} = H^i_{r'}(\delta)\}$. Then it follows that $T = T^{1/m}_{\delta, n, i}$ for some $m$, $n$. This means that $\hat{w}^{i, \delta}_T \in \text{supp}(\mathcal{Y}^{\delta}_{\delta, i})$. We furthermore have that $H^i_{r'}(\delta) = H^i_r(\delta)$ and that $\hat{w}^{i, \delta}_r = \hat{w}^{i, \delta}_T$. Therefore $x \in \text{supp}(\mathcal{Y}^{\delta}_{\delta, i})$ as required.
    \end{proof}
\end{lemma}
We now state consequence of Lemma \ref{lemma support sbm interval}.
\begin{corollary}\label{corollary range sbm in range oosbm}
    Let $(\mathcal{X}_y)_{y \geq 0}$ be the $\text{SBM}(4)$ defined in Theorem \ref{theorem sbm snake} and let $(\mathcal{Y}^{(\rm d)}_s, \mathcal{Y}^{(\rm a)}_s)_{s \geq 0}$ be the $\text{ooSBM}(4, c, \Tilde{c})$ (defined in Corollary \ref{corollary continuous modifications}) which is the continuous modification of the process $(\mathcal{Z}^{(\rm d)}_s, \mathcal{Z}^{(\rm a)}_s)_{s \geq 0}$ (defined in Theorem \ref{theorem ooSBM ffd convergence}). Note that both $(\mathcal{X}_y)_{y \geq 0}$ and $(\mathcal{Y}^{(\rm d)}_s, \mathcal{Y}^{(\rm a)}_s)_{s \geq 0}$ are constructed on the same probability space $(\Omega, \mathcal{F}, \mathbb{P}_{(\mu_{\rm d},\mu_{\rm a})})$. As above let $\mathcal{Y}_s = \mathcal{Y}^{(\rm d)}_s \cup \mathcal{Y}^{(\rm a)}_s$ for $s \geq 0$. Let $\delta > 0$. Then we have $\mathbb{P}_{(\mu_{\rm d}, \mu_{\rm a})}$-a.s. that
    \begin{equation}\label{equation coupled domination}
        \overline{\bigcup_{y \geq \delta}\text{supp}(\mathcal{X}_y)} =  \overline{\bigcup_{i \in I} \{\hat{w}^i_t: f^i_t > \delta\}} \subseteq \overline{\bigcup_{i \in I} \{ \hat{w}^i_t: f^i_t > H^i_t(\delta)\}} = \overline{\bigcup_{s \geq \delta} \text{supp}(\mathcal{Y}_s)}.
    \end{equation}
    Secondly we have $\mathbb{P}_{(\mu_{\rm d}, \mu_{\rm a})}$-a.s. that 
    \begin{equation*}
        \mathcal{R}(\mathcal{X}) \subseteq \mathcal{R}(\mathcal{Y}^{(\rm d)}) \cup \mathcal{R}(\mathcal{Y}^{(\rm a)}).
    \end{equation*}
    \begin{proof}
        The first equality in \eqref{equation coupled domination} follows from the snake representation in Theorem \ref{theorem sbm snake} and the result in Proposition \ref{proposition closure of range of continuous process} means we can exclude the case where $f^i_t = \delta$. The second equality follows by Lemma \ref{lemma support sbm interval}. The subset relation follows by the fact that for $z \geq 0$, $i \in I$, $r \in [0,\sigma(f^i)]$ we have $H^i_r(z) \leq z$. 
        \par
        The second result follows from the fact that the range can be written as a countable union, e.g.
        \newline
        $\mathcal{R}(\mathcal{X}) = \bigcup_{\delta>0, \delta \in \mathbb{Q}}\overline{\bigcup_{y \geq \delta}\text{supp}(\mathcal{X}_y)}$ and similarly for $\mathcal{Y}$ and so we can apply the first result countably many times to obtain the second result here.
    \end{proof}
\end{corollary}
The next few results work in the same setting as Corollary \ref{corollary range sbm in range oosbm}.
\begin{proposition}\label{proposition range oosbm inclusion}
    In the setting of Corollary \ref{corollary range sbm in range oosbm}, we have, for each $\delta>0$, $\mathbb{P}_{(\mu_{\rm d}, \mu_{\rm a})}$-a.s. that
    \begin{equation*}
        \overline{\bigcup_{s \geq \delta} \text{supp}(\mathcal{Y}_s)} \subseteq \mathcal{R}(\mathcal{X}) \cup \text{supp}(\mu_{\rm d}) \cup \text{supp}(\mu_{\rm a}).
    \end{equation*}
    and therefore we have $\mathbb{P}_{(\mu_{\rm d}, \mu_{\rm a})}$-a.s.
    \begin{equation*}
        \mathcal{R}(\mathcal{Y}^{(\rm d)}) \cup \mathcal{R}(\mathcal{Y}^{(\rm a)}) = \mathcal{R}(\mathcal{Y}) \subseteq \mathcal{R}(\mathcal{X}) \cup \text{supp}(\mu_{\rm d}) \cup \text{supp}(\mu_{\rm a}).
    \end{equation*}
    \begin{proof}
        Consider the representation of $\overline{\bigcup_{s \geq \delta} \text{supp}(\mathcal{Y}_s)}$ given in Lemma \ref{lemma support sbm interval}, and let $i \in I$, $r \in [0,\sigma(f^i)]$, $s > 0$. We have that there exists $q \in \mathbb{Q}$ such that $H^i_r(s)>q>0$. Therefore, as can be seen from the proof of Corollary \ref{corollary range sbm in range oosbm}, we have that $\hat{w}^i_r \in \overline{\bigcup_{y \geq q} \text{supp}(\mathcal{X}_y)}$. We can repeat the process for a countable dense set of points $\hat{w}^i_r$ in $\overline{\bigcup_{s \geq \delta} \text{supp}(\mathcal{Y}_s)}$. Then, by taking the countable union over $q$ for $q \in \mathbb{Q}$ we can conclude that
        \begin{equation*}
            \overline{\bigcup_{s \geq \delta} \text{supp}(\mathcal{Y}_s)} \subseteq \overline{\mathcal{R}(\mathcal{X})} \hspace{0.5cm} \mathbb{P}_{(\mu_{\rm d}, \mu_{\rm a})}-\text{a.s.}
        \end{equation*}
        Therefore, taking a countable union over $\delta$ for $\delta \in \mathbb{Q}$ gives us that 
        \begin{equation*}
            \mathcal{R}(\mathcal{Y}) = \bigcup_{\delta >0}  \overline{\bigcup_{s \geq \delta} \text{supp}(\mathcal{Y}_s)}  = \bigcup_{\delta >0, \delta \in \mathbb{Q}} \overline{\bigcup_{s \geq \delta} \text{supp}(\mathcal{Y}_s)} \subseteq  \overline{\mathcal{R}(\mathcal{X})}.
        \end{equation*}
        The result then follows as $\overline{\mathcal{R}(\mathcal{X})} \subseteq \mathcal{R}(\mathcal{X}) \cup \text{supp}(\mu_{\rm d}) \cup \text{supp}(\mu_{\rm a})$. To see this first note that for each $i \in I$ the set $\{\hat{w}^i_t, t \in [0,\sigma(f^i)]\}$ is a closed set. However, as only finitely many excursions $f^i$ exceed height $y$ for any $y>0$ we have that any limit point outside of $\mathcal{R}(\mathcal{X})$ must be a limit point of the points $\{b_i, i \in I\}$ and therefore must be in the support of $\mu_{\rm d}$ or of $\mu_{\rm a}$ (this is noted in the remark after \cite[Corollary IV.9]{le1999spatial} and it is observation that motivates for their definition for $\mathcal{R}$ that we are also using here in this paper).
    \end{proof}
\end{proposition}
\begin{proposition}\label{proposition inclusion at time 0}
    In the setting of Corollary \ref{corollary range sbm in range oosbm}, for any $s>0$ we have $\mathbb{P}_{(\mu_{\rm d}, \mu_{\rm a})}$-a.s. that
    \begin{equation*}
        \text{supp}(\mu_{\rm a}) \subseteq \text{supp}(\mathcal{Y}_s).
    \end{equation*}
    \begin{proof}
        We consider the on/off Brownian snake given in Definition \ref{definition oo snake} on probability space $(\Omega, \mathcal{F}, \mathbb{P}_{(\mu_{\rm d}, \mu_{\rm a})})$. We assume that $\mu_{\rm a}$ is a non-zero measure otherwise the proposition is trivial.
        \par
        Let $x \in \text{supp}(\mu_{\rm a})$ and let $\varepsilon>0$. Let $B(x, \varepsilon/2)$ denote the ball of radius $\varepsilon/2$ and centre $x$, i.e. we have $B(x, \varepsilon/2) = \{y \in \mathbb{R}^d: \norm{y - x} < \varepsilon/2\}$. Then we have $\mu_{\rm a}(B(x, \varepsilon/2))>0$. Recall, for $y>0$ and an excursion $f$, that we define the hitting time of $y$ to be $T_y(f) = \inf\{t, f_t = y\}$ (and we set $T_y(f)$ to be infinity if $f$ does not hit $y$). Our aim is to partition the set $I$ into a collection of disjoint sets. For $k \in \mathbb{N}$ let $I_k$ be defined by
        \begin{equation*}
            I_k \coloneqq \{ i  \in I: \hspace{0.25cm} T_{1/k}(f^i) < \infty = T_{1/(k-1)}(f^i)\}.
        \end{equation*}
        We note that for $k \in \mathbb{N}$ that $n(\{f: T_{1/k}(f) < \infty = T_{1/(k-1)}(f)\}) = 1/2$. As $\sum_{i \in I}\delta_{f^i}$ is a Poisson random measure with intensity measure $(\norm{\mu_{\rm d}} + \norm{\mu_{\rm a}})n(\cdot)$ and the sets $I_k$ are disjoint, we have that $|I_k|$, $k \in \mathbb{N}$, are independent Poisson random variables each with mean $(\norm{\mu_{\rm d}} + \norm{\mu_{\rm a}})/2$.
        \par
        For a compound Poisson subordinator with drift $S$, let $\tau(S) = \inf\{u: S_{u-} \neq S_u\}$ be the time of the first discontinuity of $S$ and let $\Delta(S) = S_{\tau(S)} - S_{\tau(S)-}$ be the size of this discontinuity. We can then consider the sets $J_k$, $k \in \mathbb{N}$ given by
        \begin{equation*}
            J_k = \Big\{ i \in I_k: \hspace{0.25cm} s_i = 0, \hspace{0.25cm} \norm{b_i - x} < \frac{\varepsilon}{2}, \hspace{0.25cm} \tau(v^i_{T_{1/k}}) \in (0,1/2k), \hspace{0.25cm} \Delta(v^i_{T_{1/k}})>s \Big\}.
        \end{equation*}
        By the independence of the marking we have that $|J_k|$, $k \in \mathbb{N}$, are independent Poisson random variables with means given by
        \begin{equation*}
            \mathbb{E}_{(\mu_{\rm d}, \mu_{\rm a})}(|J_k|) = \frac{1}{2}\mu_{\rm a}(B(x, \varepsilon/2))(1-e^{-c/2k})(e^{-\Tilde{c}s})
            \geq \frac{c}{4k}\mu_{\rm a}(B(x, \varepsilon/2))e^{-\Tilde{c}s}e^{-c/2},
        \end{equation*}
        where, for the inequality, we have used the fact that for $z<c/2$ that $1-e^{-z} \geq ze^{-c/2}$. We consider subsets of the sets $J_k$, $k \in \mathbb{N}$, that relate to the following probabilities. For a $d$-dimensional Brownian motion $(W(t))_{t \geq 0}$ started at the origin on probability space $(\Omega, \mathcal{F}, P)$, for $k \in \mathbb{N}$, we have
        \begin{equation*}
            P\Big(\sup_{0 \leq s \leq 1/2k} \norm{W(s)} \leq \varepsilon/2\Big) \geq P\Big(\sup_{0 \leq s \leq 1/2} \norm{W(s)} \leq \varepsilon/2\Big) > 0.
        \end{equation*}
        We now consider a further restriction of the sets $J_k$, $k \in \mathbb{N}$, given by 
        \begin{equation*}
            J'_k \coloneqq \Big\{i \in J_k: \hspace{0.25cm} \sup_{u \in [0, 1/k]}\big\{\norm{w^i_{T_{1/k}}(u) - x}\big\} < \varepsilon\Big\}.
        \end{equation*}
        We again note that by the independence of the marking that the $|J'_k|$, $k \in \mathbb{N}$, are independent Poisson random variables with means that satisfy
        \begin{equation*}
            \mathbb{E}_{(\mu_{\rm d}, \mu_{\rm a})}(|J'_k|) \geq P\Big(\sup_{0 \leq s \leq 1/2k} \norm{W(s)} \leq \varepsilon/2\Big)\mathbb{E}_{(\mu_{\rm d}, \mu_{\rm a})}(|J_k|) \geq P\Big(\sup_{0 \leq s \leq 1/2} \norm{W(s)} \leq \varepsilon/2\Big)\mathbb{E}_{(\mu_{\rm d}, \mu_{\rm a})}(|J_k|)
        \end{equation*}
        for $k \in \mathbb{N}$. This follows as the starting point for the spatial Brownian motion component is $w^i_{T_{1/k}}(0) = b_i$ and that $\norm{x-b_i} <\varepsilon/2$ for $i \in J_k$, $k \in \mathbb{N}$. Now, we have that 
        \begin{equation*}
            \mathbb{P}_{(\mu_{\rm d}, \mu_{\rm a})}(|J'_k| \geq 1) \geq 1 - \exp(-(\mathbb{E}_{(\mu_{\rm d}, \mu_{\rm a})}(|J'_k|) \wedge 1)) \geq (1-e^{-1}) \wedge (e^{-1}\mathbb{E}_{(\mu_{\rm d}, \mu_{\rm a})}(|J'_k|))
        \end{equation*}
        (because for $x<1$ we have that $1-e^{-x} \geq e^{-1}x$). Therefore as $\mathbb{P}_{(\mu_{\rm d}, \mu_{\rm a})}(|J'_k| \geq 1)$, $k \in \mathbb{N}$, are not summable and the $|J'_k|$ are independent we have, by the Borel-Cantelli lemma, that $\{|J'_k| \geq 1\}$ occurs infinitely often $\mathbb{P}_{(\mu_{\rm d}, \mu_{\rm a})}$-a.s. In particular there exists $i \in I$ and $k \in \mathbb{N}$ such that $i \in J'_k$ a.s. 
        \par
        Let $i \in J'_k$ for some $k \in \mathbb{N}$. Then we have that $H^i_{T_{1/k}}(s)<1/k$ (where $H^i$ is the age process of $x^i$, see Definition \ref{definition age process}). Note that $T_{1/k}(x^i)$ is a stopping time for the Brownian snake $x^i$. Therefore, by the strong Markov property for the Brownian snake we can consider the snake $(x^i_{T_{1/k}(x^i)+t})_{t \geq 0}$.  For this snake we can consider the stopping time given by $T \coloneqq \{r>T_{1/k}(x^i): \text{Res}^i_r(s) = 0\}$. 
        \par
        We have that $T<\infty$ and that $\hat{w}^i_T \in \text{supp}(\mathcal{Y}_s)$ . We also have by the snake property that
        \newline
        $\hat{w}^i_T = w^i_{T_{1/k}(x^i)}(f^i_T)$ and that therefore $\norm{b_i - \hat{w}^i_T}<\varepsilon/2$. As we also have $\norm{b_i - x}<\varepsilon/2$, then $\norm{\hat{w}^i_T - x} < \varepsilon$. Therefore as $\hat{w}^i_T$ is in $\text{supp}(\mathcal{Y}_s)$ (by Lemma \ref{lemma supports sbm and oosbm}, see also Corollary \ref{corollary continuous modifications}), $\varepsilon$ was arbitrary and $\text{supp}(\mathcal{Y}_s)$ is a closed set, we have that $x$ is in $\text{supp}(\mathcal{Y}_s)$ a.s. (as we can take a countable sequence of $\varepsilon$ values that converge to $0$). We can repeat this process for a dense countable subset $A$ of $\text{supp}(\mu_{\rm a})$ and conclude that $A \subseteq \text{supp}(\mathcal{Y}_s)$ a.s. As $\text{supp}(\mathcal{Y}_s)$ is closed we can then conclude that $\text{supp}(\mu_{\rm a}) \subseteq \text{supp}(\mathcal{Y}_s)$ a.s. as required.
    \end{proof}
\end{proposition}
We can now prove Theorem \ref{theorem equality ranges}.
\begin{proof}[Proof of Theorem \ref{theorem equality ranges}]
We first consider the case $\gamma = 4$. In the setting of Corollary \ref{corollary range sbm in range oosbm}, we have $\mathbb{P}_{(\mu_{\rm d}, \mu_{\rm a})}$-a.s. that
\begin{equation*}
    \mathcal{R}(\mathcal{X}) \cup \text{supp}(\mu_{\rm d}) \cup \text{supp}(\mu_{\rm a}) \subseteq \mathcal{R}(\mathcal{Y}^{(\rm d)}) \cup \mathcal{R}(\mathcal{Y}^{(\rm a)})
\end{equation*}
by Corollaries \ref{corollary dormant support} and \ref{corollary range sbm in range oosbm} and by Proposition \ref{proposition inclusion at time 0}. We have the reverse inclusion by Proposition \ref{proposition range oosbm inclusion}. It follows that $\mathcal{R}(\mathcal{Y}^{(\rm d)}) \cup \mathcal{R}(\mathcal{Y}^{(\rm a)})$ is a closed set as $\text{supp}(\mu_{\rm d})$, $\text{supp}(\mu_{\rm a})$ are both closed and the closure of $\mathcal{R}(\mathcal{X})$ is contained in $\mathcal{R}(\mathcal{X}) \cup \text{supp}(\mu_{\rm d}) \cup \text{supp}(\mu_{\rm a})$. 
\par
For general $\gamma>0$ we consider the following modifications to our arguments. In Section \ref{section snake construction}, we replace the spatial motion $\xi=(B,S)$ which consisted of a standard d-dimensional Brownian motion $B$ on $\mathbb{R}^d$ (with mean $0$ and variance $1$) and an independent subordinator $S$. We replace $\xi$ with $\xi'=(B', S')$ where $B'$ is a d-dimensional Brownian motion with mean $0$ and variance $4/\gamma$ and the subordinator $S$ is of the same form as $S$ where we have substituted the parameters $c$, $\Tilde{c}$ with parameters $4c/\gamma$, $4\Tilde{c}/\gamma$. We then rescale time with $s \mapsto \gamma s$ for the process $(\mathcal{Z}_{\gamma s}^{(\varepsilon, \rm d)}, \mathcal{Z}_{\gamma s}^{(\varepsilon, \rm a)})_{s \geq 0}$ defined in Proposition \ref{proposition identification oobmm h erased tree}. The rescaling of time and the parameters results in, for $\varepsilon > 0$, $(\mathcal{Z}_{\gamma s}^{(\varepsilon, \rm d)}, \mathcal{Z}_{\gamma s}^{(\varepsilon, \rm a)})_{s \geq 0}$ having the distribution of $\text{ooBBM}(\gamma/\varepsilon, c, \Tilde{c})$. The same arguments in Theorem \ref{theorem ooSBM ffd convergence} apply and we obtain a process $(\mathcal{Z}_s^{(\rm d)}, \mathcal{Z}_s^{(\rm a)})_{s \geq 0}$ with the finite-dimensional distributions of $\text{ooSBM}(\gamma, c, \Tilde{c})$. Also the same rescaling of parameters and time results in $(\mathcal{X}_y, y \geq 0)$ in Theorem \ref{theorem sbm snake} being an $\text{SBM}(\gamma, c, \Tilde{c})$. The argument follows from here in exactly the same way as the case $\gamma = 4$ and we obtain the range coupling as required.
\end{proof}
We now state the Hausdorff dimension of the range and support of $\text{DSBM}(\gamma, c)$.
\begin{lemma}\label{lemma dimensions damped sbm}
    Let $\mu_{\rm a} \in \mathcal{M}_F(\mathbb{R}^d)$ be a non-zero finite measure and let $\gamma >0$. Consider $\text{DSBM}(\gamma, c)$, $(\Tilde{\mathcal{Y}}_s)_{s \geq 0}$, with initial measure $\Tilde{\mathcal{Y}}_0 = \mu_{\rm a}$. Then the dimension of the range and of the support of $(\Tilde{\mathcal{Y}}_s)_{s \geq 0}$ are those of $\text{SBM}(\gamma)$ given in Proposition \ref{proposition dimension range and support sbm}, i.e. for $s>0$
    \begin{equation*}
        \text{dim}(\text{supp}(\Tilde{\mathcal{Y}}_s)) = 2 \wedge d \hspace{0.25cm} \text{   on   } \{\Tilde{\mathcal{Y}}_s \neq 0\} \hspace{0.25cm} \Tilde{\mathbb{P}}_{\mu_{\rm a}}\text{-a.s.;} \hspace{0.5cm} \text{dim}(\mathcal{R}(\Tilde{\mathcal{Y}})) = 4 \wedge d \hspace{0.25cm} \Tilde{\mathbb{P}}_{\mu_{\rm a}}\text{-a.s.}
    \end{equation*}
    \begin{proof}
        The proof for the upper bounds, i.e. that $\text{dim}(\text{supp}(\Tilde{\mathcal{Y}}_s)) \leq 2 \wedge d$ and $\text{dim}(\mathcal{R}(\Tilde{\mathcal{Y}})) \leq 4 \wedge d$, immediately follows from the fact that $\Tilde{\mathcal{Y}}_s$ is dominated by $\mathcal{X}_s$. The proof for the lower bound for the support uses the same argument as \cite[Theorem IV.7]{le1999spatial} as the following condition is satisfied:
        \begin{equation}\label{equation support inequality lower bound}
            \Tilde{\mathbb{N}}_b\Big(\int_{\norm{y} \leq K} \int_{\norm{z} \leq K} \Tilde{\mathcal{Y}}_t(dy) \Tilde{\mathcal{Y}}_t(dz) |y-z|^{\varepsilon - 2}\Big) < \infty
        \end{equation}
        For the lower bound for the range, we define the total occupation measure $\Tilde{\mathcal{J}}$ of $\Tilde{\mathcal{Y}}$ by 
        \begin{equation*}
            \langle \Tilde{\mathcal{J}}, g \rangle = \int_{s=0}^\infty ds \langle \Tilde{\mathcal{Y}}_s, g \rangle
        \end{equation*}
        which is supported on $\mathcal{R}(\Tilde{\mathcal{Y}})$. The proof for the lower bound for the range then uses the same argument as in \cite[Theorem IV.7]{le1999spatial} as we have that for $d \geq 4$ and for every $K$, $\varepsilon$, $\delta>0$ that
        \begin{equation}\label{equation range inequality lower bound}
            \Tilde{\mathbb{N}}_b\Big(\int_{\varepsilon \leq \norm{y - b} \leq K} \int_{\varepsilon \leq \norm{z - b} \leq K} \frac{\mathcal{J}(dy)\mathcal{J}(dz)}{|y-z|^{4-\delta}}\Big) < \infty,
        \end{equation}
        as well as the fact that $\Tilde{\mathcal{Y}}$ is a non-trivial measure-valued process (i.e. a.s. not equal to the zero-measure for a positive initial interval of time $[0,y]$). Both the inequalities in \eqref{equation support inequality lower bound} and \eqref{equation range inequality lower bound} can be seen from a direct comparison with SBM which has the corresponding inequalities in the proof of \cite[Theorem IV.7]{le1999spatial}.
    \end{proof}
\end{lemma}
Our next result of the paper concerns the increasing support property of ooSBM.
\begin{lemma}[Range domination]\label{lemma oosbm range domination}
    As previously, let $(\mathcal{X}_y)_{y \geq 0}$ be the $\text{SBM}(4)$ defined in Theorem \ref{theorem sbm snake}; let $(\mathcal{Y}^{(\rm d)}_s, \mathcal{Y}^{(\rm a)}_s)_{s \geq 0}$ be the $\text{ooSBM}(4, c, \Tilde{c})$ (defined in Corollary \ref{corollary continuous modifications}) which is the continuous modification of the process $(\mathcal{Z}^{(\rm d)}_s, \mathcal{Z}^{(\rm a)}_s)_{s \geq 0}$ (defined in Theorem \ref{theorem ooSBM ffd convergence}); and let $(\Tilde{\mathcal{Y}}_s)_{s \geq 0}$ be the $\text{DSBM}(4, c)$ defined in Corollary \ref{corollary continuous modifications} which is the continuous modification of the process $(\Tilde{\mathcal{Z}}_s)_{s \geq 0}$ defined in Proposition \ref{proposition damped sbm}. Let $\mathcal{Y}_s = \mathcal{Y}^{(\rm d)}_s + \mathcal{Y}^{(\rm a)}_s$ for $s \geq 0$. Note that all these processes are defined on the same probability space $(\Omega, \mathcal{F}, \mathbb{P}_{(\mu_{\rm d},\mu_{\rm a})})$. Then we have
    \begin{equation*}
        \overline{\bigcup_{0 \leq s \leq t} \text{supp}(\Tilde{\mathcal{Y}}_s)} \subseteq \overline{\bigcup_{0 \leq s \leq t} \text{supp}(\mathcal{Y}_s)} \subseteq \overline{\bigcup_{0 \leq s \leq t} \text{supp}(\mathcal{X}_s)} \hspace{0.5cm} \forall t \geq 0 \hspace{0.25cm} \mathbb{P}_{(\mu_{\rm d}, \mu_{\rm a})}\text{-a.s.}
    \end{equation*}
    \begin{proof}
    Let $t \geq 0$. Firstly by Lemma \ref{lemma supports sbm and oosbm} we have
        \begin{equation*}
            \text{supp}(\Tilde{\mathcal{Y}}_s) = \overline{\bigcup_{i \in I}\{\hat{w}^i_r, r \in [0,\sigma(f^i)], f^i_r = s, s_i = 0, \hat{\Tilde{v}}^i_r=0\}} \hspace{0.25cm} \text{ for all   } s \in \mathbb{Q} \hspace{0.25cm} \mathbb{P}_{(\mu_{\rm d}, \mu_{\rm a})}\text{-a.s.,}
        \end{equation*}
        and therefore by Proposition \ref{proposition closure of range of continuous process} we have that 
        \begin{equation*}
            \overline{\bigcup_{0 \leq s \leq t} \text{supp}(\Tilde{\mathcal{Y}}_s)} = \overline{\bigcup_{i \in I}\bigcup_{0<s<t, s \in \mathbb{Q}}\hat{w}^i_r, r \in [0,\sigma(f^i)], f^i_r = s, s_i = 0, \hat{\Tilde{v}}^i_r=0\}}.
        \end{equation*}
        Now, we note that for $i \in I$, $r \in [0,\sigma(f^i)]$, if $f^i_r = s$ and $\hat{\Tilde{v}}^i_r=0$ then $H^i_r(s) = s$ and so $f^i_r = H^i_r(s)$. Furthermore as $s<t$ we have that $H^i_r(t) = H^i_r(s) = f^i_r$. Therefore, by Lemma \ref{lemma support sbm interval 2}, we have that $\hat{w}^i_r$ is in the support of $\overline{\bigcup_{0 \leq s \leq t}\text{supp}(\mathcal{Y}_s)}$. As this holds for all such $\hat{w}^i_r$ we have the first inclusion of this lemma.
        \par
        For the second inclusion we note that the age process $H^X_r(y) \leq y$ for any $r$ and any snake $X$ by definition of the age process in Definition \ref{definition age process} and so the result follows by Theorem \ref{theorem sbm snake} and Lemma \ref{lemma support sbm interval 2}. We have this result for all $t \in \mathbb{Q}$ simultaneously but we can extend this to all $t \in [0,\infty)$ simultaneously with Proposition \ref{proposition closure of range of continuous process} which shows us we only need to consider the supports over the rational time points.
    \end{proof}
\end{lemma}
\begin{proposition}\label{proposition dimension oosbm}
    Let $\mu_{\rm d}$, $\mu_{\rm a} \in \mathcal{M}_F(\mathbb{R}^d)$ be such that $\mu_{\rm d} + \mu_{\rm a}$ is a non-zero measure and let $(\mathcal{Y}_s^{(\rm d)}, \mathcal{Y}_s^{(\rm a)})_{s \geq 0}$ be an $\text{ooSBM}(4, c, \Tilde{c})$ started from $(\mu_{\rm d}, \mu_{\rm a})$. For any $s \geq 0$ let $\mathcal{Y}_s = \mathcal{Y}_s^{(\rm d)} +  \mathcal{Y}_s^{(\rm a)}$. Then, for any fixed $t \geq 0$, we have
    \begin{equation*}
        \text{dim}\Big( \overline{\bigcup_{0 \leq s \leq t} \text{supp}(\mathcal{Y}_s)} \Big) = \text{dim}(\text{supp}(\mu_{\rm d})) \vee \text{dim}(\text{supp}(\mu_{\rm a})) \vee (4 \wedge d) \hspace{0.5cm} \mathbb{P}_{(\mu_{\rm d}, \mu_{\rm a})}-\text{a.s.}
    \end{equation*}
    \begin{proof}
        This almost follows immediately from Lemma \ref{lemma oosbm range domination} and the dimension results given in Proposition \ref{proposition dimension range and support sbm} and Lemma \ref{lemma dimensions damped sbm}. First assume $\mu_{\rm a}$ is a non-zero finite measure. We wish to apply the dimension results in Proposition \ref{proposition dimension range and support sbm} and Lemma \ref{lemma dimensions damped sbm} for the ranges of the superprocesses to determine the dimension of the closure of the union of the supports from time $\delta$ to $t$, $\text{dim}(\overline{\bigcup_{\delta \leq s \leq t} \text{supp}(\mathcal{Y}_s)})$, for any $\delta > 0$. We note that the proof of the dimension result from \cite[Theorem IV.7]{le1999spatial} which shows that each constituent part of the SBM $\mathcal{X}$ formed by the snakes $(x^i, i \in \Tilde{I})$ has range dimension $4 \wedge d$. Now, some of these clusters a.s. have contour functions $f^i$ with height less than $t$, and the range from time $\delta$ of the component of the superprocess is equal to the closure of the union of supports from times $\delta$ to $t$. Therefore the dimension of the union of the supports is at least $4 \wedge d$. When we include the initial support of $\mathcal{X}_0$ which has dimension $\text{dim}(\text{supp}(\mu_{\rm d})) \vee \text{dim}(\text{supp}(\mu_{\rm a}))$ we have the dimension is at least $(\text{supp}(\mu_{\rm d})) \vee \text{dim}(\text{supp}(\mu_{\rm a})) \vee (4 \wedge d)$. The reverse inclusion is immediate so the dimension is equal to $(\text{supp}(\mu_{\rm d})) \vee \text{dim}(\text{supp}(\mu_{\rm a})) \vee (4 \wedge d)$ as required. 
        \par
        It remains to consider the case where $\mu_{\rm a} = 0$ and therefore that $\mu_{\rm d}$ is a non-zero finite measure. However, in this case we can show that there exists $T \in (0,t/2)$ and $n \in \mathbb{N}$ such that $\mathcal{Y}^{(\rm a)}_T(\mathbb{R}^d)>1/n$ a.s. To see this we recall from \cite[Proposition 1.9]{blath2023off} that the total mass processes $(q_s, p_s)_{s \geq 0} \coloneqq (\mathcal{Y}^{(\rm d)}_s(\mathbb{R}^d), \mathcal{Y}^{(\rm a)}_s(\mathbb{R}^d))_{s \geq 0}$ are governed by the stochastic differential equations
        \begin{equation*}
            dq_t = (cp_t - \Tilde{c}q_t) dt; \hspace{0.5cm} dp_t = (\Tilde{c}q_t - cp_t)dt + \sqrt{\gamma p_t} dB_t. 
        \end{equation*}
        Therefore, we can see that $p_s=0$ for all $s \in (0,t/2)$ leads to a contradiction. We also have that $(q_s)_{s \geq 0}$ and $(p_s)_{s \geq 0}$ are continuous processes (this can be seen for example by the continuity of $(\mathcal{Y}^{(\rm d)}_s)_{s \geq 0}$, $(\mathcal{Y}^{(\rm a)}_s)_{s \geq 0}$ with respect to the Prokhorov metric). Let the time $T_n$ be the first hitting time of $1/n$ for $(p_s)_{s \geq 0}$. Then the result follows for $t>T_n$ by the previous argument for the case where $\mu_{\rm a}$ is non-zero and the strong Markov property. We have a.s. that $t/2>T_n$ for some $n \in \mathbb{N}$ and therefore the result follows.
    \end{proof}
\end{proposition}
In Lemma \ref{lemma support of SBM thickening} we note the result from \cite[Theorem 1.1]{dawson1989super} that concerns the support of SBM. For a set $A \subseteq \mathbb{R}^d$ and for $\varepsilon' \geq 0$, we define the $\varepsilon'$-thickening of $A$ to be the set $A^{\varepsilon'}$ given by
\begin{equation*}
    A^{\varepsilon'} \coloneqq \{ x \in \mathbb{R}^d: |x-y| \leq \varepsilon' \text{ for some   } y \in A\}.
\end{equation*}
\begin{lemma}\label{lemma support of SBM thickening}
    Let $\mu \in \mathcal{M}_F(\mathbb{R}^d)$ and let $\gamma >0$. Let $(\mathcal{X}_r)_{r \geq 0}$ be an $\text{SBM}(\gamma)$ started from $\mu$ on a probability space $(\Omega, \mathcal{F}, \mathbb{P}_\mu)$. For each $a>2$ there $\mathbb{P}_{\mu}$-a.s. exists $\delta>0$ such that if $s$, $t \geq 0$ satisfy $0 \leq t-s<\delta$ then
    \begin{equation*}
        \text{supp}(\mathcal{X}_t) \subseteq \text{supp}(\mathcal{X}_s)^{aH(t-s)}
    \end{equation*}
    where $H:[0,\infty) \rightarrow [0, \infty)$ is the function given by $H(0) = 0$ and for $x>0$ we have
    \begin{equation*}
        H(x) \coloneqq \sqrt{x((\log 1/x) \vee 1))}.
    \end{equation*}
\end{lemma}
Note that the $\delta$ in Lemma \ref{lemma support of SBM thickening} is random. The important features of the function $H$ for us are that $H$ is continuous and increasing. We now give a simple corollary of Lemma \ref{lemma support of SBM thickening}.
\begin{corollary}\label{corollary SBM thickening}
    Let $\mu \in \mathcal{M}_F(\mathbb{R}^d)$ and $\gamma > 0$. Let $(\mathcal{X}_r)_{r \geq 0}$ be an $\text{SBM}(\gamma)$ started from $\mu$ on a probability space $(\Omega, \mathcal{F}, \mathbb{P}_\mu)$. Then $\mathbb{P}_{\mu}$-a.s. there exists $N<\infty$ such that for any $s$, $t \geq 0$ with $0 \leq t-s<1$ we have that
    \begin{equation*}
        \overline{\bigcup_{s \leq u \leq t} \text{supp}(\mathcal{X}_u)} \subseteq \text{supp}(\mathcal{X}_s)^{NH(t-s)}.
    \end{equation*}
    \begin{proof}
        If we take $a=3$ in Lemma \ref{lemma support of SBM thickening} and then for the random positive constant $\delta$ and for $0 < t-s<\delta$ we have that
        \begin{equation*}
            \bigcup_{s \leq u \leq t}\text{supp}(\mathcal{X}_u) \subseteq \bigcup_{s \leq u \leq t}\text{supp}(\mathcal{X}_s)^{aH(u-s)} = \text{supp}(\mathcal{X}_s)^{3H(t-s)}.
        \end{equation*}
        Now let $N = 4/\delta$ where $\delta$ is the (random) positive constant given in Lemma \ref{lemma support of SBM thickening} and let $s$, $t \geq 0$ with $0 < t-s < 1$. Then there exists $s_1, \dots, s_K$ be such that $s<s_1< \dots < s_K<t$ and $s_1-s<\delta$, $t-s_K<\delta$, $s_{i+1}-s_i<\delta$ for all $1, \dots, K-1$, and $K \leq 1/\delta$. Therefore, by applying Lemma \ref{lemma support of SBM thickening} we have 
        \begin{equation*}
            \bigcup_{s \leq u \leq s_1}\text{supp}(\mathcal{X}_u) \subseteq \text{supp}(\mathcal{X}_s)^{3H(s_1-s)} \subseteq \text{supp}(\mathcal{X}_s)^{3H(t-s)},
        \end{equation*}
        for $1 \leq i \leq K$ we have
        \begin{equation*}
            \bigcup_{s_i \leq u \leq s_{i+1}}\text{supp}(\mathcal{X}_u) \subseteq \text{supp}(\mathcal{X}_{s_i})^{3H(s_{i+1}-s_i)} \subseteq \text{supp}(\mathcal{X}_{s_i})^{3H(t-s)},
        \end{equation*}
        and finally
        \begin{equation*}
            \bigcup_{s_K \leq u \leq t}\text{supp}(\mathcal{X}_u) \subseteq \text{supp}(\mathcal{X}_{s_K})^{3H(t-s_K)} \subseteq \text{supp}(\mathcal{X}_{s_K})^{3H(t-s)}.
        \end{equation*}
        With the above we have
        \begin{equation*}
            \bigcup_{s_i \leq u \leq s_{i+1}}\text{supp}(\mathcal{X}_u) \subseteq \text{supp}(\mathcal{X}_{s_i})^{3H(t-s)} \subseteq \text{supp}(\mathcal{X}_{s_{i-1}})^{6H(t-s)} \subseteq \dots \subseteq \text{supp}(\mathcal{X}_s)^{3iH(t-s)},
        \end{equation*}
        since $\text{supp} (\mathcal{X}_{s_i}) \subseteq \cup_{s_{i-1} \leq u \leq s_i} \text{supp}(\mathcal{X}_u)$, and similarly $\bigcup_{s_K \leq u \leq t}\text{supp}(\mathcal{X}_u) \subseteq \text{supp}(\mathcal{X}_s)^{3KH(t-s)}$. Therefore we have
        \begin{equation*}
            \bigcup_{s \leq u \leq t}\text{supp}(\mathcal{X}_u) \subseteq \text{supp}(\mathcal{X}_s)^{3H(t-s)/\delta}
        \end{equation*}
        and increasing the $3$ to $4$ to define $N$ is simply so that we can ensure the closure of $\bigcup_{s \leq u \leq t}\text{supp}(\mathcal{X}_u)$ is contained in $\text{supp}(\mathcal{X}_s)^{NH(t-s)}$.
    \end{proof}
\end{corollary}
We can now apply this result to $\text{ooSBM}(4, c, \Tilde{c})$. These results all apply to $\text{ooSBM}(\gamma, c, \Tilde{c})$ for general $\gamma > 0$ with the simple time-change described in the proof of Theorem \ref{theorem equality ranges} earlier in this section.
\begin{lemma}\label{lemma oosbm finite propagation}
    Let $\mu_{\rm d}$, $\mu_{\rm a} \in \mathcal{M}_F(\mathbb{R}^d)$ and let $(\mathcal{Y}_t^{(\rm d)}, \mathcal{Y}_t^{(\rm a)})_{t \geq 0}$ be an $\text{ooSBM}(4, c, \Tilde{c})$ started from $(\mu_{\rm d}, \mu_{\rm a})$. Then we have $\mathbb{P}_{(\mu_{\rm d}, \mu_{\rm a})}$-a.s. there exists $N<\infty$ such that for all $0 \leq t < 1$ we have
    \begin{equation*}
        \overline{\bigcup_{0 \leq u \leq t} \text{supp}(\mathcal{Y}_u^{(\rm d)}) \cup \text{supp}(\mathcal{Y}_u^{(\rm a)})} \subseteq (\mu_{\rm d} \cup \mu_{\rm a})^{NH(t)}.
    \end{equation*}
    \begin{proof}
        This can be concluded from the result for the coupled processes in Lemma \ref{lemma oosbm range domination} and from Corollary \ref{corollary SBM thickening}.
    \end{proof}
\end{lemma}
\begin{corollary}\label{corollary oosbm finite propagation all times}
    Let $\mu_{\rm d}$, $\mu_{\rm a} \in \mathcal{M}_F(\mathbb{R}^d)$ and let $(\mathcal{Y}_t^{(\rm d)}, \mathcal{Y}_t^{(\rm a)})_{t \geq 0}$ be an $\text{ooSBM}(4, c, \Tilde{c})$ started from $(\mu_{\rm d}, \mu_{\rm a})$. Then $\mathbb{P}_{(\mu_{\rm d}, \mu_{\rm a})}$-a.s. there exist $(N_k, k \in \mathbb{N})$ finite constants such that for $k \in \mathbb{N}$ and for $s$, $t \in \mathbb{R}$ with $k \leq s < t < k+1$ we have
    \begin{equation*}
        \overline{\bigcup_{s \leq u \leq t} \text{supp}(\mathcal{Y}_u^{(\rm d)}) \cup \text{supp}(\mathcal{Y}_u^{(\rm a)})} \subseteq \text{supp}(\mathcal{Y}_s^{(\rm d)}) \cup \text{supp}(\mathcal{Y}_s^{(\rm a)})^{N_kH(t-s)}.
    \end{equation*}
    \begin{proof}
        Lemma \ref{lemma oosbm finite propagation}, along with the Markov property for ooSBM, establishes the claim for $k=0$. The general result can then be derived from the $k=0$ case and then again by the Markov property of ooSBM.
    \end{proof}
\end{corollary}
We can now prove the increasing monotone support property stated in Theorem \ref{theorem increasing support} and the dimension result in that theorem follows immediately from this.
\begin{proof}[Proof of Theorem \ref{theorem increasing support}]    
    Let $\mu_{\rm d}$, $\mu_{\rm a} \in \mathcal{M}_F(\mathbb{R}^d)$ and let $(\mathcal{Y}_t^{(\rm d)}, \mathcal{Y}_t^{(\rm a)})_{t \geq 0}$ be an $\text{ooSBM}(4, c, \Tilde{c})$ started from $(\mu_{\rm d}, \mu_{\rm a})$. We first note that 
    \begin{equation}\label{equation increasing support rationals}
        \text{supp}(\mathcal{Y}_r^{(\rm d)}) \cup \text{supp}(\mathcal{Y}_r^{(\rm a)}) \subseteq \text{supp}(\mathcal{Y}_s^{(\rm d)}) \cup \text{supp}(\mathcal{Y}_s^{(\rm a)}) \hspace{0.5cm} \forall r,s \in \mathbb{Q}, r < s \hspace{0.25cm} \mathbb{P}_{(\mu_{\rm d}, \mu_{\rm a})}\text{-a.s.}
    \end{equation}
    This follows for $r=0$ by Proposition \ref{proposition inclusion at time 0} and the representation for ooSBM given in Theorem \ref{theorem ooSBM ffd convergence} and it can be extended to general $r \in \mathbb{Q}$ by the Markov property. As $r$, $s \in \mathbb{Q}$ comprises of only countably many cases we have that all the events hold simultaneously $\mathbb{P}_{(\mu_{\rm d}, \mu_{\rm a})}$-a.s.
    \par
    We now show that $\mathbb{P}_{(\mu_{\rm d}, \mu_{\rm a})}$-a.s. for all $k \in \mathbb{N}$ and for all $r$, $s \geq 0$ with $k < r < s < k + 1$ we have $\text{supp}(\mathcal{Y}_r^{(\rm d)}) \cup \text{supp}(\mathcal{Y}_r^{(\rm a)}) \subseteq \text{supp}(\mathcal{Y}_s^{(\rm d)}) \cup \text{supp}(\mathcal{Y}_s^{(\rm a)})$. The increasing support property at all times then follows immediately from this along with \eqref{equation increasing support rationals}.
    \par
    We consider two cases: the first is that $s \in \mathbb{Q}$, and the second is that $r \in \mathbb{Q}$. The general result follows from these cases as there exists a rational number between any pair of distinct real numbers. Throughout the rest of the proof we work on the $\mathbb{P}_{(\mu_{\rm d}, \mu_{\rm a})}$-a.s. event given by the intersection of the event described in \eqref{equation increasing support rationals} and the event described in Corollary \ref{corollary oosbm finite propagation all times}.
    \par
    Consider the first case where $s \in \mathbb{Q}$ and $r \in \mathbb{R}$. There exists a sequence $(q_n)_{n \in \mathbb{N}}$ such that $q_n \rightarrow r$, $q_n \in \mathbb{Q}$ and $q_n < s$ for all $n \in \mathbb{N}$. Then we have that for all $n \in \mathbb{N}$ that 
    \begin{equation*}
        \mathcal{Y}_r^{(\rm d)}\Big((\mathbb{R}^d \backslash \text{supp}(\mathcal{Y}_s^{(\rm d)})) \cap (\mathbb{R}^d \backslash \text{supp}(\mathcal{Y}_s^{(\rm a)}))\Big) \leq \liminf_{n \rightarrow \infty} \mathcal{Y}_{q_n}^{(\rm d)}\Big((\mathbb{R}^d \backslash \text{supp}(\mathcal{Y}_s^{(\rm d)})) \cap (\mathbb{R}^d \backslash \text{supp}(\mathcal{Y}_s^{(\rm a)}))\Big) = 0,
    \end{equation*}
    by the Portmanteau Theorem and therefore $\text{supp}(\mathcal{Y}_r^{(\rm d)}) \subseteq \text{supp}(\mathcal{Y}_s^{(\rm d)}) \cup \text{supp}(\mathcal{Y}_s^{(\rm a)})$. The same type of argument also shows that $\text{supp}(\mathcal{Y}_r^{(\rm a)}) \subseteq \text{supp}(\mathcal{Y}_s^{(\rm d)}) \cup \text{supp}(\mathcal{Y}_s^{(\rm a)})$. 
    \par
    Now assume that $r \in \mathbb{Q}$ and $s \in \mathbb{R}$.  Let $(q_n, n \in \mathbb{N})$ with $q_n \rightarrow s$, $q_n \in (s, k+1)$, $q_n \in \mathbb{Q}$. We can see by Corollary \ref{corollary oosbm finite propagation all times} that
    \begin{equation*}
        \text{supp}(\mathcal{Y}_r^{(\rm d)}) \subseteq \bigcap_{n \in \mathbb{N}} (\text{supp}(\mathcal{Y}_{q_n}^{(\rm d)}) \cup \text{supp}(\mathcal{Y}_{q_n}^{(\rm a)})) \subseteq \bigcap_{n \in \mathbb{N}} (\text{supp}(\mathcal{Y}_s^{(\rm d)}) \cup \text{supp}(\mathcal{Y}_s^{(\rm a)}))^{N_kH(q_n-s)} =: A.
    \end{equation*}
    We show that the set $A$ is a subset of $\text{supp}(\mathcal{Y}_s^{(\rm d)}) \cup \text{supp}(\mathcal{Y}_s^{(\rm a)})$ and this concludes our proof. Let $x \in A$. Then for all $n \in \mathbb{N}$ there exists $y_n \in \text{supp}(\mathcal{Y}_s^{(\rm d)}) \cup \text{supp}(\mathcal{Y}_s^{(\rm a)})$ such that $|y_n - x| < N_kH(q_n-s)$. The sequence $(y_n, n \in \mathbb{N})$ converges to $x$ and by the closure of $\text{supp}(\mathcal{Y}_s^{(\rm d)}) \cup \text{supp}(\mathcal{Y}_s^{(\rm a)})$ we have that
    \newline
    $x \in \text{supp}(\mathcal{Y}_s^{(\rm d)}) \cup \text{supp}(\mathcal{Y}_s^{(\rm a)})$ as required.
    \par
    The dimension result now follows by the monotonic increasing support property and Proposition \ref{proposition dimension oosbm}. 
    \par
    The case for general $\gamma > 0$ follows from the $\gamma = 4$ case with the same time-change arguments used in the proof of Theorem \ref{theorem equality ranges} given earlier in this section.
\end{proof}
\section{Moment properties of ooSBM}\label{section moments}
In this section we work with $\text{ooBM}(c, \Tilde{c})$ from Definition \ref{definition bbm and oobbm} as well as $\text{ooSBM}(\gamma, c, \Tilde{c})$. As the natural state space for $\text{ooBM}(c, \Tilde{c})$ is $\mathbb{R}^d \times \{0,1\}$ and our state space for $\text{ooSBM}(\gamma, c, \Tilde{c})$ is $\mathcal{M}_F((\mathbb{R}^d)^2)$ we begin with Table \ref{table different spaces} which gives summary of the different representations given for processes in this paper. Every row in the table corresponds to a set of equivalent elements in the different representations.
\begin{table}[H]
\caption{Different spaces considered in this paper}
\label{table different spaces}
\begin{center}
\begin{tabular}{||c c c c||} 
 \hline
 $\mathbb{R}^d \times [0,\infty)$ & $\mathbb{R}^d \times \{0,1\}$ & $\mathcal{M}_F(\mathbb{R}^d \times \{0,1\})$ & $\mathcal{M}_F((\mathbb{R}^d)^2)$ \\ [0.5ex] 
 \hline\hline
 $(b,0)$ & $(b,1)$ & $\delta_{(b,1)}$ & $(0,\delta_b)$ \\ 
 \hline
 $(b,E)$ & $(b,0)$ & $\delta_{(b,0)}$ & $(\delta_b, 0)$ \\
 where $E \sim \text{Exp}(\Tilde{c})$ &  &  &  \\
 \hline
 -- & -- & $\mu_{\rm d} \otimes \delta_0 + \mu_{\rm a} \otimes \delta_1$ & $(\mu_{\rm d}, \mu_{\rm a})$ \\
 \hline
\end{tabular}
\end{center}
\end{table}
The first column in Table \ref{table different spaces} represents the state space for Brownian motion in $\mathbb{R}^d$ and an independent subordinator on $[0,\infty)$. This process was the spatial motion for the bivariate snake constructed in Definition \ref{definition oo snake} which we refer to as the on/off Brownian snake. Furthermore the process can be mapped to an ooBM, see Lemma \ref{lemma single particle motion with age process}, and the state space for ooBM is the second column $\mathbb{R}^d \times \{0,1\}$. The third and fourth columns are spaces of finite measures on $\mathbb{R}^d \times \{0,1\}$ and $(\mathbb{R}^d)^2$, two homeomorphic spaces and in this section we use both spaces to represent ooSBM. On the other hand, the second row of the table represents the initial state that corresponds to a single active particle at position $b$, and the third row represents the initial state that corresponds to a single dormant particle at position $b$. Finally, the fourth row corresponds to a general measure-valued initial condition for ooSBM in its two different parametrisations as either a measure on $\mathbb{R}^d \times \{0,1\}$ or as a measure on $(\mathbb{R}^d)^2$. For measures $\mu$, $\nu$ we write $\mu \otimes \nu$ to denote the product measure.
\par
We can calculate moment results using Theorem \ref{theorem ooSBM ffd convergence}. We recall on/off Brownian motion, see Definition \ref{definition bbm and oobbm} as well as Lemma \ref{lemma single particle motion with age process}. We let $P^{\%}_{(b,i)}$ and $E^{\%}_{(b,i)}$ for $(b,i) \in \mathbb{R}^d \times \{0,1\}$ be the probability measure and the expectation operator for an $\text{ooBM}(c, \Tilde{c})$, $(U_r, I_r)_{r \geq 0}$, started from $(b,i)$. We also recall $\xi = (B,S)$ where $B$ is a Brownian motion on $\mathbb{R}^d$ and $S$ is an independent subordinator with unit drift and jump measure $\pi(dx) = c \Tilde{c} e^{-\Tilde{c}x} dx$ and finally we recall the processes $(H^{(s)}(\xi), s \geq 0)$ and $(A^{(s)}(\xi), s \geq 0)$ from Lemma \ref{lemma single particle motion with age process}. All our results occur at fixed times so we can state the results in terms of the continuous modification $(\mathcal{Y}^{(\rm d)}_s, \mathcal{Y}^{(\rm a)}_s)_{s \geq 0}$ defined in Corollary \ref{corollary continuous modifications} and the same time-change rescaling arguments work as in the proof of Theorem \ref{theorem equality ranges} in Section \ref{section range coupling proof} so our results apply to $\text{ooSBM}(\gamma, c, \Tilde{c})$ for general $\gamma > 0$.
\begin{proposition}\label{proposition expectation oosbm}
Let $(b,i) \in \mathbb{R}^d \times \{0,1\}$ and consider an $\text{ooBM}(c, \Tilde{c})$, $(U_r, I_r)_{r \geq 0}$, started from $(b,i)$. Then, for any continuous bounded function $\Phi: \mathbb{R}^d \rightarrow [0,\infty)$ we have 
    \begin{equation*}
        \mathbb{E}_{(\mu_{\rm d}, \mu_{\rm a})} (\langle \mathcal{Y}^{(\rm a)}_s, \Phi \rangle ) = \int_{b \in \mathbb{R}^d} E^{\%}_{(b,1)}[\Phi(U_s)\mathbbm{1}_{I_s=1}] \mu_{\rm a}(db) + \int_{b \in \mathbb{R}^d} E^{\%}_{(b,0)}[\Phi(U_s)\mathbbm{1}_{I_s = 1}] \mu_{\rm d}(db);
    \end{equation*}
    \begin{equation*}
        \mathbb{E}_{(\mu_{\rm d}, \mu_{\rm a})} (\langle \mathcal{Y}^{(\rm d)}_s, \Phi \rangle ) = \int_{b \in \mathbb{R}^d} E^{\%}_{(b,1)}[\Phi(U_s)\mathbbm{1}_{I_s=0}] \mu_{\rm a}(db) + \int_{b \in \mathbb{R}^d} E^{\%}_{(b,0)}[\Phi(U_s)\mathbbm{1}_{I_s=0}] \mu_{\rm d}(db).
    \end{equation*}
    \begin{proof}
    We first consider the case $\gamma = 4$ and we deal with the active component. Recall from Definition \ref{definition oo snake} the on/off Brownian snake $\sum_{i \in I} \delta_{((b_i,s_i), f^i, x^i)}$ and also recall that $w^i$ denotes the Brownian component of the snake $x^i$. From Theorem \ref{theorem ooSBM ffd convergence} we have that
    \begin{equation*}
        \langle \mathcal{Y}^{(\rm a)}_s, \Phi \rangle = \sum_{i \in I: s_i<s} \int_{t=0}^{\sigma(f^i)} \Phi(\hat{w}^i_t)\ell^{\rm a}_{s, i}(dt)
    \end{equation*}
    and therefore we have
    \begin{equation*}
        \mathbb{E}_{(\mu_{\rm d}, \mu_{\rm a})} (\langle \mathcal{Y}^{(\rm a)}_s, \Phi \rangle) = \mathbb{E}_{(\mu_{\rm d}, \mu_{\rm a})} \Big(\sum_{i \in I: s_i<s} \int_{t=0}^{\sigma(f^i)} \Phi(\hat{w}^i_t)\ell^{\rm a}_{s, i}(dt)\Big).
    \end{equation*}
    By the compensation formula for Poisson random measures this is equal to 
    \begin{equation*}
            \int_{b \in \mathbb{R}^d} \mathbb{N}_{(b,0)}\Big(\int_{t=0}^{\sigma(f)} \Phi(\hat{W}_t)\ell^{\rm a}_s(dt)\Big) \mu_{\rm a}(db) + \int_{b \in \mathbb{R}^d} \int_{q=0}^s \mathbb{N}_{(b,q)}\Big(\int_{t=0}^{\sigma(f)} \Phi(\hat{W}_t)\ell^{\rm a}_s(dt)\Big) \nu_s(dq) \mu_{\rm d}(db)
        \end{equation*}
    where $\nu_s(dq) \coloneqq \mathbbm{1}_{0 \leq q \leq s}\Tilde{c}e^{-\Tilde{c}q}$. For the first excursion measure we can define $D = \mathbb{R}^d \times [0,\infty) \backslash \{s\}$ and set $\tau(\xi) = \inf\{y \geq 0: \xi_y \notin D\}$. Then by \eqref{equation local time general} from Lemma \ref{lemma snake exit open set general} we have, (recalling that $\xi =(B,S)$ and $\tau = H^{(s)}(\xi)$ and that the subordinator component in this case does not jump over the level $s$, meaning that $\tau < \infty$ iff $A^{(s)} = 1$) that
    \begin{equation}\label{equation expectation active}
        \mathbb{N}_{(b,0)}\Big(\int_{t=0}^{\sigma(f^i)} \Phi(\hat{W}_t)\ell^{\rm a}_s(dt)\Big) = E^\xi_{(b,0)}\Big[\Phi(B(\tau)) \mathbbm{1}_{\tau < \infty}\Big] = E^\xi_{(b,0)}\Big[\Phi(B(H^{(s)}(\xi)))\mathbbm{1}_{A^{(s)}(\xi) = 1}\Big].
    \end{equation}
    We have similar expressions for the other term and so we obtain
        \begin{multline*}
            \mathbb{E}_{(\mu_{\rm d}, \mu_{\rm a})} (\langle \mathcal{Y}^{(\rm a)}_s, \Phi \rangle) = \int_{b \in \mathbb{R}^d} E^\xi_{(b,0)}[\Phi(B(H^{(s)}(\xi)))\mathbbm{1}_{A^{(s)}(\xi) = 1}] \mu_{\rm a}(db) \\ + \int_{b \in \mathbb{R}^d} \int_{s=0}^q E^{\xi}_{(b,q)}[\Phi(B(H^{(s)}(\xi)))\mathbbm{1}_{A^{(s)}(\xi) = 1}] \Tilde{c}e^{-\Tilde{c}q} dq \mu_{\rm d}(db);
        \end{multline*}
    For the dormant case (with $\gamma = 4$) we cannot immediately apply Lemma \ref{lemma snake exit open set general} as we did for the active case. As in Theorem \ref{theorem ooSBM ffd convergence} we can apply the same result with $D = \mathbb{R}^d \times (s,\infty)$ and obtain
    \begin{equation}\label{equation expectation}
        \mathbb{N}_{(b,0)}\Big(\int_{t=0}^{\sigma(f^i)} \Phi(\hat{W}_t)\ell_s(dt)\Big) = E^\xi_{(b,0)}\Big[\Phi(B(\tau)) \mathbbm{1}_{\tau < \infty}\Big] = E^\xi_{(b,0)}\Big[\Phi(B(H^{(s)}(\xi)))\Big],
    \end{equation}
    and it is the difference of \eqref{equation expectation} and \eqref{equation expectation active} which gives
    \begin{equation*}
        \mathbb{N}_{(b,0)}\Big(\int_{t=0}^{\sigma(f^i)} \Phi(\hat{W}_t)\ell^{\rm d}_s(dt)\Big) =  E^\xi_{(b,0)}\Big[\Phi(B(H^{(s)}(\xi)))\mathbbm{1}_{A^{(s)}(\xi) = 0}\Big].
    \end{equation*}
    We obtain similar expressions for each of the expectations and obtain
        \begin{multline*}
            \mathbb{E}_{(\mu_{\rm d}, \mu_{\rm a})} (\langle \mathcal{Y}^{(\rm d)}_s, \Phi \rangle ) = \int_{b \in \mathbb{R}^d} E^\xi_{(b,0)}[\Phi(B(H^{(s)}(\xi)))\mathbbm{1}_{A^{(s)}(\xi) = 0}] \mu_{\rm a}(db) \\ + \int_{b \in \mathbb{R}^d} \int_{s=0}^q E^{\xi}_{(b,q)}[\Phi(B(H^{(s)}(\xi)))\mathbbm{1}_{A^{(s)}(\xi) = 0}] \Tilde{c}e^{-\Tilde{c}q} dq \mu_{\rm d}(db) + e^{-\Tilde{c}s} \int_{b \in \mathbb{R}^d} \Phi(b) \mu_{\rm d}(db).
        \end{multline*}
    The expectations $E^\xi$ with respect to $\xi$ are equal to the expectations $E^{\%}$ with respect to ooBM by the equality in distribution given in Lemma \ref{lemma single particle motion with age process} along with the fact that the final term is the expectation for the dormant component for the initial dormant mass that has not moved to the active state before time $s$. 
    \par
    The general $\gamma > 0$ case follows from rescaling parameters and a time-change for our construction of ooSBM, see the proof of Theorem \ref{theorem equality ranges} in Section \ref{section range coupling proof}.
\end{proof}
\end{proposition}
We state a corollary that connects the moments of ooSBM to ooBM. As ooBM has a natural state space in $\mathbb{R}^d \times \{0,1\}$ we rewrite everything in terms of measures on $\mathbb{R}^d \times \{0,1\}$, see Remark \ref{remark different measure representations}. We consider an ooSBM $(\mathcal{Y}''_s)_{s \geq 0}$ of the form $\mathcal{Y}''_s \coloneqq \mathcal{Y}^{(\rm d)}_s \otimes \delta_0 + \mathcal{Y}^{(\rm a)}_s \otimes \delta_1$ for dormant and active components $\mathcal{Y}^{(\rm d)}_s$, $\mathcal{Y}^{(\rm a)}_s$.
\begin{corollary}
    Let $\Phi_0$, $\Phi_1$ be continuous bounded functions $\mathbb{R}^d \rightarrow [0,\infty)$. Let $\Phi: \mathbb{R}^d \times \{0,1\} \rightarrow [0,\infty)$ be given by $\Phi((b,0)) = \Phi_0(b)$ and $\Phi((b,1)) = \Phi_1(b)$. Similarly, let $\mu_{\rm d}$, $\mu_{\rm a} \in \mathcal{M}_F(\mathbb{R}^d)$ and then let $\mu$ be the finite measure on $\mathbb{R}^d \times \{0, 1\}$ given by $\mu = \mu_{\rm d} \otimes \delta_0 + \mu_{\rm a} \otimes \delta_1$. Let $\gamma>0$ and let $\mathbb{E}_\mu$ be the expectation operator for an $\text{ooSBM}(\gamma, c, \Tilde{c})$, which we denote by $(\mathcal{Y}''_s)_{s \geq 0}$, on $\mathcal{M}_F(\mathbb{R}^d \times \{0,1\})$ with initial state $\mu$. Finally, let $(K^{\%}_s)_{s \geq 0}$ be the transition semigroup for $\text{ooBM}(c, \Tilde{c})$. Then we have
    \begin{equation*}
        \mathbb{E}_\mu(\langle \mathcal{Y}''_s, \Phi \rangle) = \langle \mu, K^{\%}_s \Phi \rangle.
    \end{equation*}
    \begin{proof}
        This result simply pulls together the different parts of the expectations given in Proposition \ref{proposition expectation oosbm}. Specifically, we have
        \begin{equation*}
            \mathbb{E}_\mu(\langle \mathcal{Y}''_s, \Phi \rangle) = \mathbb{E}_{(\mu_{\rm d}, \mu_{\rm a})} (\langle \mathcal{Y}^{(\rm d)}_s, \Phi_0 \rangle) + \mathbb{E}_{(\mu_{\rm d}, \mu_{\rm a})} (\langle \mathcal{Y}^{(\rm a)}_s, \Phi_1 \rangle).
        \end{equation*}
        The integrals in the equations correspond to the parts of the expectation, i.e. we have
        \begin{equation*}
            \langle \mu_{\rm d} \otimes \delta_0, K^{\%}_s \Phi \rangle = \int_{b \in \mathbb{R}^d} E^{\%}_{(b,0)}[\Phi_0(U_s)\mathbbm{1}_{I_s=0}] \mu_{\rm d}(db) + \int_{b \in \mathbb{R}^d} E^{\%}_{(b,0)}[\Phi_1(U_s)\mathbbm{1}_{I_s = 1}] \mu_{\rm d}(db)
        \end{equation*}
        and we also have
        \begin{equation*}
            \langle \mu_{\rm a} \otimes \delta_1, K^{\%}_s \Phi \rangle = \int_{b \in \mathbb{R}^d} E^{\%}_{(b,1)}[\Phi_0(U_s)\mathbbm{1}_{I_s=0}] \mu_{\rm a}(db) + \int_{b \in \mathbb{R}^d} E^{\%}_{(b,1)}[\Phi_1(U_s)\mathbbm{1}_{I_s=1}] \mu_{\rm a}(db).
        \end{equation*}
    \end{proof}
\end{corollary}
\begin{remark}
    In \cite[Proposition 2.21]{blath2023off} we note that the expectation that the first moment measure result given in Proposition \ref{proposition existence of ooSBM} is identical to the expectation result given for $\text{ooBBM}(x, c, \Tilde{c})$ for any $x > 0$. Therefore for any $\gamma>0$ we have the convergence in moments of $\text{ooBBM}(\gamma/\varepsilon, c, \Tilde{c})$ with particles rescaled to size $\varepsilon$ to the moments of $\text{ooSBM}(\gamma, c, \Tilde{c})$ as $\varepsilon$ goes to $0$ as well as the already established convergence in distribution.
\end{remark}
The expectations in Proposition \ref{proposition expectation oosbm} appear difficult to evaluate analytically because the length of time the ooBM is in the active state is random, and therefore the spatial location is an integral evaluated over Brownian paths of differing lengths. However, when we apply the formulas above with the function $\Phi \equiv 1$ the problem reduces to a two-state continuous-time Markov chain with states $\{0,1\}$ and rates $c$, $\Tilde{c}$. It is a simple calculation to show this has transition semigroup $P(t) = P_{ij}(t)$, for $t \geq 0$ and $i$, $j \in \{0,1\}$,
\begin{equation*}
    \begin{pmatrix}
        \frac{c+ \Tilde{c}e^{-t(c+\Tilde{c})}}{c+\Tilde{c}} & \frac{\Tilde{c}(1-e^{-t(c+\Tilde{c})})}{c+\Tilde{c}} \\
        \frac{c(1-e^{-t(c+\Tilde{c})})}{c+\Tilde{c}} & \frac{ce^{-t(c+\Tilde{c})} + \Tilde{c}}{c+\Tilde{c}}
    \end{pmatrix}
\end{equation*}
and for any $b \in \mathbb{R}^d$ we have that $P^{\%}_{(b,i)}(I_t = j) = P_{ij}(t)$. We are now in a position to prove Theorem \ref{theorem expectation total masses}.
\begin{proof}[Proof of Theorem \ref{theorem expectation total masses}]
    We use Proposition \ref{proposition expectation oosbm} with $\Phi \equiv 1$. We first observe that for an $\text{ooBM}(c, \Tilde{c})$,
    \newline
    $(U_r, I_r)_{r \geq 0}$, we have
    \begin{equation*}
        E^{\%}_{(b,1)}[\mathbbm{1}_{I_t=1}] = P_{11}(t) = \frac{ce^{-(c + \Tilde{c})t} + \Tilde{c}}{c + \Tilde{c}}; 
    \end{equation*}
    and similarly for the other ooBM terms. Therefore we have
    \begin{equation*}
        \mathbb{E}(\norm{\mathcal{Y}^{(\rm a)}_t}) = \mathbb{E}_{(\mu_{\rm d}, \mu_{\rm a})} (\langle \mathcal{Y}^{(\rm a)}_t, \Phi \rangle) = \int_{b \in \mathbb{R}^d} \frac{ce^{-t(c+\Tilde{c})} + \Tilde{c}}{c+\Tilde{c}} \mu_{\rm a}(db) + \int_{b \in \mathbb{R}^d} \frac{\Tilde{c}(1-e^{-t(c+\Tilde{c})})}{c+\Tilde{c}} \mu_{\rm d}(db)
    \end{equation*}
    which is equal to 
    \begin{equation*}
        \norm{\mu_{\rm d}}\frac{\Tilde{c}(1-e^{-t(c+\Tilde{c})})}{c+\Tilde{c}} + \norm{\mu_{\rm a}} \frac{ce^{-t(c+\Tilde{c})} + \Tilde{c}}{c+\Tilde{c}}.
    \end{equation*}
    The method is the same to calculate $\mathbb{E}(\norm{\mathcal{Y}^{(\rm d)}_t})$.
\end{proof}
\vspace{0.5cm}
\renewcommand{\abstractname}{Acknowledgements}
\begin{abstract}
    This research has been supported (MB) by the EPSRC Centre for Doctoral Training in Mathematics of Random Systems: Analysis, Modelling and Simulation (EP/S023925/1) while at the University of Oxford from October 2020 to October 2024. MB also acknowledges support from the German Research Foundation through grant 443227151 while at the Universit\"at zu L\"ubeck from November 2024 to August 2025 as well as the support from the EPSRC Grant EP/W006227/1 while at the University of Warwick from September 2025 onwards. This research has also been supported (DJ) by DFG IRTG 2544 Berlin-Oxford by DFG under Germany’s Excellence Strategy – The Berlin Mathematics Research Center MATH+ (EXC-2046/1, project ID 390685689, BMS Stipend). The author (DJ) acknowledges funding by the Deutsche Forschungsgemeinschaft (DFG, German Research Foundation) – CRC/TRR 388 "Rough Analysis, Stochastic Dynamics and Related Fields“ – Project ID 516748464.
    \par
    Both authors would like to thank (in alphabetical order) Jochen Blath, Alison Etheridge and Matthias Winkel for their helpful discussions with this project, as well as  Jean-Fran\c{c}ois Delmas and Christina Goldschmidt for the helpful discussion during the PhD thesis viva (MB).
\end{abstract}

\printbibliography

\end{document}